\documentclass[10pt]{amsart}
\usepackage{geometry}
\usepackage{amsfonts,amssymb,verbatim,amsmath,amsthm,latexsym,textcomp,amscd}
\usepackage{latexsym,amsfonts,amssymb,epsfig,verbatim}
\usepackage[mathscr]{euscript}
\usepackage{amsmath,amsthm,amssymb,latexsym,graphics,textcomp}
\usepackage{graphicx}
\usepackage{color}
\usepackage{url}
\usepackage{enumerate}
\usepackage{xcolor}
\RequirePackage[bookmarks, bookmarksopen=true, plainpages=false, pdfpagelabels, pdfpagelayout=SinglePage, breaklinks = true]{hyperref}
\usepackage[noabbrev,capitalize,nameinlink]{cleveref}
\hypersetup{
	colorlinks,
	linkcolor={red!50!black},
	citecolor={green!50!black},
	urlcolor={blue!50!black}
}
\usepackage{centernot}
\usepackage{mathtools}
\usepackage{microtype}
\usepackage{tikz-cd}
\usepackage{quiver}

\newtheorem{theorem}{Theorem}[section]
\newtheorem{prop}[theorem]{Proposition}
\newtheorem{lemma}[theorem]{Lemma}
\newtheorem{coro}[theorem]{Corollary}
\newtheorem{definition}[theorem]{Definition}

\newtheorem{rmk}[theorem]{Remark}

\newtheorem{subtheorem}{Theorem}[subsection]
\newtheorem{subprop}[subtheorem]{Proposition}
\newtheorem{sublemma}[subtheorem]{Lemma}
\newtheorem{subcoro}[subtheorem]{Corollary}
\newtheorem{subdefinition}[subtheorem]{Definition}

\newtheorem{subrmk}[subtheorem]{Remark}

\newcommand\CC{{\mathcal C}}
\newcommand\DD{{\mathcal D}}

\newcommand\FF{{\mathcal F}}
\newcommand\GG{{\mathcal G}}
\newcommand\HH{{\mathcal H}}

\newcommand\LL{{\mathcal L}}
\newcommand\MM{{\mathcal M}}

\newcommand\PP{{\mathcal P}}

\newcommand\PMF{{\PP\kern-2pt\MM\FF}}

\newcommand\PML{{\PP\kern-2pt\MM\LL}}

\newcommand\R{{\mathbb R}}

\newcommand{\fsubd}{\mathrel{{\scriptstyle\searrow}\kern-1ex^d\kern0.5ex}}
\newcommand{\bsubd}{\mathrel{{\scriptstyle\swarrow}\kern-1.6ex^d\kern0.8ex}}
\newcommand{\fsubeq}{\mathrel{\raise-.7ex\hbox{$\overset{\searrow}{=}$}}}
\newcommand{\bsubeq}{\mathrel{\raise-.7ex\hbox{$\overset{\swarrow}{=}$}}}

\newcommand{\tsh}[1]{\left\{\kern-.9ex\left\{#1\right\}\kern-.9ex\right\}}

\newcommand{\Dis}{\hbox{Dis}}

\newcommand{\argmx}{\hbox{argmax}}
\newcommand{\argmn}{\hbox{argmin}}

\newcommand{\cone}{\hbox{Cone}}
\newcommand{\shad}{\hbox{Shad}}

\title[GH-convergence of maximal Gromov hyperbolic spaces and their boundaries]{Gromov-Hausdorff convergence of maximal Gromov hyperbolic spaces and their boundaries}

\author{Kingshook Biswas}
\address{Stat-Math Unit, Indian Statistical Institute, 203 B. T. Rd., Kolkata 700108, India}
\email{kingshook@isical.ac.in}
\author{Arkajit Pal Choudhury}
\address{Stat-Math Unit, Indian Statistical Institute, 203 B. T. Rd., Kolkata 700108, India}
\email{rkjtpalchoudhury\_r@isical.ac.in}
\begin{document}
	\begin{abstract} 
		The relation between negatively curved spaces and their boundaries is important for various rigidity problems. In \cite{biswas2024quasi}, the class of Gromov hyperbolic spaces called maximal Gromov hyperbolic spaces was introduced, and the boundary functor $X \mapsto \partial X$ was shown to give 
		an equivalence of categories between maximal Gromov hyperbolic spaces (with morphisms being isometries) and a class of compact quasi-metric spaces 
		called quasi-metric antipodal spaces (with morphisms being Moebius homeomorphisms). The proof of this equivalence involved the construction of 
		a filling functor $Z \mapsto {\mathcal M}(Z)$, associating to any quasi-metric antipodal space $Z$ a maximal Gromov hyperbolic space ${\mathcal M}(Z)$. 
		
		\medskip
		
		We study the ``continuity" properties of the boundary and filling functors. We show that convergence of a sequence of quasi-metric antipodal spaces 
		(in a certain sense called ``almost-isometric convergence") implies convergence (in the Gromov-Hausdorff sense) of the associated maximal Gromov hyperbolic 
		spaces. Conversely, we show that convergence of maximal Gromov hyperbolic spaces together with a natural hypothesis of ``equicontinuity" on the boundaries 
		implies convergence of boundaries. We use this to show that Gromov-Hausdorff convergence of a sequence of proper, geodesically complete CAT(-1) spaces implies 
		Gromov-Hausdorff convergence of their boundaries equipped with visual metrics. We also show that convergence of maximal Gromov hyperbolic spaces to a 
		maximal Gromov hyperbolic space with finite boundary implies convergence of boundaries. 
	\end{abstract}
	
	\bigskip
	
	\maketitle
	
	\tableofcontents
	
	\section{Introduction}
	
	\medskip
	
	The relation between the geometry of negatively curved spaces and that of their boundaries is a well-studied topic, which is an important part of 
	various rigidity theorems and problems for negatively curved spaces, such as the Mostow rigidity theorem and the Marked length spectrum rigidity problem. 
	For a Gromov hyperbolic space $X$, the geometry of the boundary $\partial X$ is encoded in a function (defined up to a bounded multiplicative error), called the {\it cross-ratio} of the boundary, which is a 
	positive function of quadruples of distinct points in $\partial X$, generalizing the classical Euclidean cross-ratio on the boundary $S^{n-1}$ of the real hyperbolic $n$-space 
	$\mathbb{H}^n$. A central problem is to understand to what extent the cross-ratio on $\partial X$ determines the geometry of $X$. 
	
	\medskip
	
	A map $f : \partial X \to \partial Y$ between boundaries of Gromov hyperbolic spaces $X, Y$ is said to be {\it power quasi-Moebius} if it distorts cross-ratios at most by 
	raising to a power and multiplying by a constant (see for eg. \cite{jordi} for the precise definition). Under mild assumptions on $X$ and $Y$, it is known that any such map extends to a quasi-isometry $F : X \to Y$ (\cite{bonk-schramm} for power quasi-symmetric boundary maps, \cite{jordi} for general power quasi-Moebius maps). For a general Gromov hyperbolic space 
	the cross-ratio is defined only up to a bounded multiplicative error. If the Gromov hyperbolic space is however {\it boundary continuous} (the Gromov inner product extends continuously 
	to the boundary), then the cross-ratio is well-defined, without any error term. Examples of such spaces include CAT(-1) spaces, and as shown in \cite{biswas2024quasi}, CAT(0) Gromov 
	hyperbolic spaces.  For such spaces, one can talk of {\it Moebius maps} between their boundaries, which are maps between boundaries which preserve cross-ratios. Isometries between boundary continuous Gromov hyperbolic spaces extend to Moebius maps between their boundaries. The converse, or rigidity of Moebius maps, is an important open problem: does any Moebius map 
	$f : \partial X \to \partial Y$ between boundaries of boundary continuous Gromov hyperbolic spaces $X, Y$ extend to an isometry $F : X \to Y$? As explained 
	in \cite{biswas2015moebius}, an affirmative answer to this question would lead to a solution of the marked length rigidity problem of Burns and Katok \cite{burns-katok} for 
	closed negatively curved manifolds. Rigidity of Moebius maps is known to hold only in certain special cases, either when $X$ and $Y$ are both rank one symmetric spaces of non-compact type normalized to have upper curvature bound $-1$ (\cite{bourdon2}), or when $X, Y$ are 
	both geodesically complete trees (\cite{Hughes}, \cite{beyrer-schroeder}). There are also local and infinitesimal rigidity results for Moebius maps between boundaries of negatively curved 
	Hadamard manifolds, proved in \cite{biswas4}. For general boundary continuous spaces $X, Y$, the hyperbolic filling techniques of \cite{bonk-schramm} and 
	\cite{jordi} are too rough to give anything more than quasi-isometric extensions of boundary maps, even with the stronger hypothesis that the boundary map is Moebius (under this hypothesis, their results do however give almost-isometric extensions, i.e. $(1, C)$-quasi-isometric for some $C > 0$). 
	
	\medskip
	
	In \cite{biswas2024quasi}, a large class of spaces was introduced for which rigidity of Moebius maps does indeed hold, namely the class of 
	{\it maximal Gromov hyperbolic spaces}. These spaces are proper, geodesic, geodesically complete and boundary continuous; for convenience, we will call Gromov hyperbolic 
	spaces satisfying these four properties {\it good Gromov hyperbolic spaces}. The boundaries of good Gromov hyperbolic spaces are certain compact quasi-metric spaces called {\it quasi-metric antipodal spaces} (we refer to \Cref{preliminaries} for the definition of antipodal spaces). Any quasi-metric antipodal space is equipped with the cross-ratio of its quasi-metric. 
	A {\it hyperbolic filling} of a quasi-metric antipodal space $Z$ is a pair $(X, f)$ where $X$ is a good Gromov hyperbolic space and $f : \partial X \to Z$ is a Moebius 
	homeomorphism. A {\it maximal Gromov hyperbolic space} is a good Gromov hyperbolic space $X$ which is maximal amongst all hyperbolic fillings of its boundary 
	$\partial X$, in the sense that for any hyperbolic filling $( Y, g : \partial Y \to \partial X)$ of $\partial X$, the Moebius homeomorphism $g : \partial Y \to \partial X$ extends 
	to an isometric embedding $G : Y \to X$. 
	
	\medskip
	
	In \cite{biswas2024quasi}, it was shown that any quasi-metric antipodal space $Z$ admits a hyperbolic filling by a unique maximal Gromov hyperbolic space ${\mathcal M}(Z)$. 
	The space ${\mathcal M}(Z)$, called the {\it Moebius space} associated to $Z$, depends only on the cross-ratio on $Z$, and is given by the space of all Moebius antipodal functions on $Z$ equipped with a certain metric 
	(we refer to \Cref{preliminaries} for the definitions). The association $Z \mapsto {\mathcal M}(Z)$ is functorial, and any Moebius homeomorphism $f : Z_1 \to Z_2$ between antipodal 
	spaces extends canonically to an isometry $F : {\mathcal M}(Z_1) \to {\mathcal M}(Z_2)$ between the associated Moebius spaces. It was also shown in \cite{biswas2024quasi} 
	that any maximal Gromov hyperbolic space $X$ is canonically isometric to the Moebius space ${\mathcal M}(\partial X)$, thus the class of maximal Gromov hyperbolic spaces 
	coincides with the class of Moebius spaces. It follows that any Moebius map $f : \partial X \to \partial Y$ between boundaries of maximal Gromov hyperbolic spaces $X, Y$ extends to 
	an isometry $F : X \to Y$. From this one obtains an equivalence of categories between the categories of maximal Gromov hyperbolic spaces (with isometries as morphisms) and quasi-metric antipodal spaces (with Moebius homeomorphisms as morphisms). 
	
	\medskip
	
     There is a close connection between maximal Gromov hyperbolic spaces and injective metric spaces (we refer to \cite{lang2013injective} for basic 
	definitions and properties of injective metric spaces). In \cite{biswas2024quasi} it is shown that a good Gromov hyperbolic space $X$ is maximal if and only if it is injective. 
     Thus maximal Gromov hyperbolic spaces are the same as injective, good Gromov hyperbolic spaces. Moreover, the injective hull of any good Gromov hyperbolic space 
     is a maximal Gromov hyperbolic spaces. The Moebius space ${\mathcal M}(Z)$ associated to a quasi-metric antipodal space $Z$ is the injective hull of any hyperbolic filling of $Z$. 
     The results obtained in this article about maximal Gromov hyperbolic spaces may also be viewed as results about injective good Gromov hyperbolic spaces, and injective 
     hulls of good Gromov hyperbolic spaces.  
	
	\medskip
	
	Our aim in this article is to study the ``continuity" properties of the two functors giving the equivalence of categories between maximal Gromov hyperbolic spaces and 
	quasi-metric antipodal spaces, namely the boundary 
	functor $X \mapsto \partial X$, and the filling functor $Z \mapsto {\mathcal M}(Z)$. For the maximal Gromov hyperbolic spaces, which are proper metric spaces, 
	the natural notion of convergence is that of Gromov-Hausdorff convergence, where Gromov-Hausdorff convergence of non-compact proper metric spaces means Gromov-Hausdorff 
	convergence of closed balls of radius $R > 0$ for every $R > 0$. For quasi-metric antipodal spaces, which are compact quasi-metric spaces but not necessarily metric spaces, 
	we introduce a notion of {\it almost-isometric convergence} (AI-convergence) of quasi-metric spaces, which reduces to Gromov-Hausdorff convergence in the case of metric spaces.  
	Our first result below may then be interpreted as saying that the filling functor $Z \mapsto {\mathcal M}(Z)$ is continuous. We prove it for the larger class of antipodal spaces 
	(see \Cref{preliminaries} for the definition of antipodal spaces), for which one may also define the Moebius spaces ${\mathcal M}(Z)$. As shown in \cite{biswas2024quasi}, there is an equivalence 
	of categories between a category of certain metric spaces called maximal Gromov product spaces (which contain the maximal Gromov hyperbolic spaces as a proper subcategory) 
	and the category of antipodal spaces. Gromov product spaces are a class of metric spaces introduced in \cite{biswas2024quasi} which contains the class of good 
	Gromov hyperbolic spaces, and whose boundaries are antipodal spaces. One can define as before the notion of a maximal Gromov product space, and it is shown in 
	\cite{biswas2024quasi} that any maximal Gromov product space is a Moebius space ${\mathcal M}(Z)$ associated to some antipodal space $Z$. Antipodal spaces are semi-metric spaces and for these one also has a notion of almost-isometric convergence. We then have:  
	
	\medskip
	

	
	\begin{theorem}\label{forward}
		Let $\{(Z_n,\rho_n)\}_{n\ge 1}$ be a sequence of antipodal spaces which converge to an antipodal space $(Z,\rho_0)$ in the sense of almost-isometric convergence. 
		Then the Moebius spaces $(\MM(Z_n),\rho_n)$ converge to the Moebius space $(\MM(Z),\rho_0)$ in Gromov-Hausdorff sense.
		
		\medskip
		
		In particular, if $\{ (X_n, x_n) \}_{n \ge 1}, (X, x)$ are maximal Gromov product spaces such that $(\partial X_n, \rho_{x_n}) \to (\partial X, \rho_x)$ in the sense 
		of almost-isometric convergence, then $(X_n, x_n) \to (X, x)$ in the sense of Gromov-Hausdorff convergence.

         \medskip

     \end{theorem}

 \medskip

     The results obtained in \cite{biswas2024quasi} imply that the injective hull $E(X)$ of a good Gromov hyperbolic space $X$  (or more generally a Gromov product space), 
     is determined by the Moebius structure on its boundary $\partial X$. As an immediate Corollary of the above Theorem, we have the following result which says that 
     if the Moebius structures on the boundaries of a sequence of good Gromov hyperbolic spaces (or more generally Gromov product spaces) converge, then the injective 
     hulls of the spaces converge:

\medskip

\begin{coro} \label{hullsconverge}

If  $\{ (X_n, x_n) \}_{n \ge 1}, (X, x)$ are good Gromov hyperbolic spaces (or more generally Gromov product spaces) 
such that $(\partial X_n, \rho_{x_n}) \to (\partial X, \rho_x)$ in the sense 
of almost-isometric convergence, then the injective hulls $(E(X_n), x_n) \to (E(X), x)$ in the sense of Gromov-Hausdorff convergence.

\end{coro}

	\medskip
	
	The continuity of the boundary functor $X \mapsto \partial X$ is more delicate. Gromov-Hausdorff convergence of non-compact Gromov hyperbolic spaces only requires convergence 
	of balls of finite radius, while the boundary at infinity of a Gromov hyperbolic space depends on the asymptotic geometry of the space near infinity, so it is not immediately 
	clear whether convergence of the spaces should imply convergence of their boundaries, without any further hypotheses on the geometry at infinity. We show that under 
	a natural hypothesis on the sequence of boundaries, called {\it equicontinuity} of antipodal spaces (see \Cref{preliminaries} for the definition), convergence of Gromov product 
	spaces does indeed imply convergence of their boundaries. 
	
	\medskip
	
	\begin{theorem}\label{backward}
		Let $\{(X_n,x_n)\}_{n\ge 1}, (X, x)$ be Gromov product spaces, and suppose that $(X_n,x_n) \to (X,x)$ in the Gromov-Hausdorff sense.  
		If the boundaries $\{(\partial_P X_n,\rho_{x_n})\}_{n\ge 1}$ form an equicontinuous family of antipodal spaces, then 
		$(\partial_P X_n,\rho_{x_n}) \to (\partial_P X,\rho_x)$ in the sense of almost-isometric convergence.
	\end{theorem}
	
	\medskip
	
	We remark that equicontinuity of boundaries is also a necessary condition for their AI-convergence, and is hence a natural hypothesis to make. 
     We also note that the Gromov product spaces in the above Theorem are 
	not required to be maximal. Equicontinuity of boundaries holds for example when all the boundaries are metric spaces, which is the case for visual metrics on the boundary 
	of a CAT(-1) space. As a corollary of the previous Theorem, we do get a result on continuity of the boundary functor $X \mapsto \partial X$ when restricted to the class 
	of proper, geodesically complete CAT(-1) spaces:
	
	\medskip
	
	\begin{theorem}\label{CAT(-1) main}
		Let $\{ (X_n,x_n) \}_{n \ge 1}, (X, x)$ be proper geodesically complete CAT$(-1)$ spaces. If $(X_n,x_n) \to (X,x)$ in Gromov-Hausdorff sense, then the associated visual metric spaces $(\partial X_n, \rho_{x_n}) \to (\partial X, \rho_x)$ in the usual Gromov-Hausdorff sense.

	\end{theorem}
	
	\medskip
	
	As another corollary of \Cref{backward} we have:
	
	\medskip
	
	\begin{coro}
		Let $\{(Z_n,\rho_n)\}_{n\ge 1}$ be an equicontinuous family of antipodal spaces, and let $(Z,\rho_0)$ be an antipodal space. 
		If $(\MM(Z_n,\rho_n),\rho_n) \to (\MM(Z,\rho_0),\rho_0)$ in the Gromov-Hausdorff sense, then $(Z_n,\rho_n) \to (Z,\rho_0)$ in the sense 
		of almost-isometric convergence.
	\end{coro}
	
	\medskip
	
	The class of maximal Gromov product spaces with finite boundary was studied in \cite{biswas2024polyhedral} (in this case the spaces are Gromov hyperbolic), where it 
	was shown that such a space $X$ is isometric to an $l_{\infty}$-polyhedral complex embedded in $({\mathbb R}^n, || . ||_{\infty})$, where $n = \# \partial X$. 
	We show that the boundary functor $X \mapsto \partial X$ is continuous at such points (i.e. maximal Gromov product spaces with finite boundary): 
	
	\medskip
	
	\begin{theorem}\label{backward finite}
		Let $\{(X_n,x_n)\}_{n\ge 1}$ be a sequence of maximal Gromov product spaces, and let $X$ be a maximal Gromov product space with finite boundary. 
		If $(X_n,x_n) \to (X,x)$ in the Gromov-Hausdorff sense, then $(\partial_P X_n,\rho_{x_n}) \to (\partial_P X,\rho_x)$ in the sense of almost-isometric 
		convergence.
	\end{theorem}
	
	\medskip
	
     As an immediate Corollary of the above Theorem and the previous results, we have the following:

\medskip

\begin{coro} \label{finitebdrys}

Let $\{(X_n, x_n)\}_{n\ge 1}$, $(X, x)$ be maximal Gromov hyperbolic spaces with finite boundaries. Then $(X_n,x_n) \to (X,x)$ in the Gromov-Hausdorff sense if 
and only if  $(\partial X_n,\rho_{x_n}) \to (\partial X,\rho_x)$ in the sense of almost-isometric convergence.
\end{coro}

\medskip
 
	Thus if one restricts oneself to the subcategories of maximal Gromov hyperbolic spaces with finite boundary and of finite antipodal spaces respectively, 
	then the boundary and filling functors are in a certain sense ``homeomorphisms" between these categories. Moreover, using \Cref{forward}, one can show that these 
	subcategories are ``dense" in the categories of maximal Gromov product spaces and of antipodal spaces respectively. In particular, any maximal Gromov product space 
	is a Gromov-Hausdorff limit of $l_{\infty}$-polyhedral complexes with finite boundary (see \Cref{PolyComplex subsection}):
	
\medskip

	\begin{theorem}\label{PolyComplex}
		Any maximal Gromov product space $X$ is a Gromov-Hausdorff limit of finite dimensional $l_{\infty}$-polyhedral complexes $X_n$, which are maximal Gromov product (hyperbolic) spaces with finite boundary. 

\medskip

     In particular, any injective, good Gromov hyperbolic space is a Gromov-Hausdorff limit of finite dimensional $l_{\infty}$-polyhedral complexes. 
	\end{theorem}

\medskip
	
	Finally, we remark that Shibahara \cite{shibahara2021gromov} has proven results on the relation 
	between the Gromov-Hausdorff convergence of proper Gromov hyperbolic spaces and that of their boundaries when equipped with certain metrics called uniform metrics. 
     However, the uniform metrics on the boundaries are in general different from the visual quasi-metrics (although they are bi-Lipschitz equivalent).
	
	\medskip
	
	\section{Preliminaries and definitions}\label{preliminaries}
	
	\medskip
	
	In this section, we shall recall some preliminaries and develop some tools and notions which will be useful in the paper.
	
	\subsection{Antipodal spaces and their associated Moebius spaces}\label{antipodal}\hfill\\
	
	\noindent The content presented in this subsection and the next one is based on definitions and results developed by Biswas. We direct the reader to consult \cite{biswas2024quasi} for a more comprehensive discussion.
	
	\begin{subdefinition}[{\bf Separating function} or {\bf Semi-metric}]\label{semi-metric}
		Let $Z$ be a compact metrizable space with at least four points.
		A function $\rho_0\colon  Z\times Z\to [0,\infty)$ is called separating (or a `semi-metric') if $\colon$ 
		\medskip
		
		(1) it is continuous
		\medskip
		
		(2) it is symmetric, i.e. $\rho(\xi,\eta)=\rho(\eta,\xi)$ for all $\xi,\eta\in Z$
		\medskip
		
		(3) it satisfies positivity, i.e. $\rho(\xi,\eta)=0$ if and only if $\xi=\eta$.
		\medskip
		
		\noindent For this discussion we shall call the pair $(Z,\rho_0)$ a compact `semi-metric' space, where $Z$ is compact metrizable space equipped with separating function $\rho_0$ (cf. \cite{wilson1931semi}).
	\end{subdefinition}
	\medskip
	
	\noindent The `cross-ratio' with respect to a separating function $\rho_0$  for a quadruple of distinct points
	$\xi,\xi',\eta,\eta' \in Z$ is defined to be
	\begin{equation*}
		[\xi,\xi',\eta,\eta']_{\rho_0} \coloneqq 
		\frac{\rho_0(\xi,\eta)\rho_0(\xi',\eta')}{\rho_0(\xi,\eta')\rho_0(\xi',\eta)}
	\end{equation*}
	
	\noindent Also two separating function $\rho_0,\rho_1$ on $Z$ is said to be {\it Moebius equivalent} if they have same cross-ratio
	\begin{equation*}
		[\xi, \xi', \eta, \eta']_{\rho_0} = [\xi, \xi', \eta, \eta']_{\rho_1}
	\end{equation*}
	for all distinct $\xi, \xi', \eta, \eta' \in Z$. A continuous map $f\colon(Z_1,\rho_1)\to(Z_2,\rho_2)$ between compact semi-metric spaces is called a {\it Moebius embedding} if it preserves cross-ratios for quadruples of distinct points. Along with this, if $f$ is surjective, then it is called {\it Moebius homeomorphism} and the two compact semi-metric spaces are {\it Moebius equivalent}.
	
	\begin{sublemma}[{\bf G}eometric {\bf M}ean-{\bf V}alue {\bf T}heorem, \cite{biswas2015moebius},\cite{biswas2024quasi} Lemma 2.2]\label{GMVT}
		Two separating functions $\rho_0, \rho_1$ are Moebius equivalent if and only if there exists a positive continuous function $\phi \colon  Z \to (0, \infty)$ such that
		\begin{equation*}
			\rho_1(\xi, \eta)^2 = \phi(\xi) \phi(\eta) \rho_0(\xi, \eta)^2
		\end{equation*}
		for all $\xi, \eta \in Z$.
		Such a function $\phi$ is unique. It is called derivative and is denoted by $\frac{d\rho_1}{d\rho_0}$. 
	\end{sublemma}
	
	\medskip
	
	\noindent As a consequence of the GMVT, one can define the following metric space as in \cite{biswas2024quasi}.
	
	\begin{subdefinition}[{\bf Unrestricted Moebius space of a separating function, \cite{biswas2024quasi}}]
		Let $\rho_0$ be a separating function on a compact metrizable space $Z$ with at least four points. The Unrestricted Moebius space $\mathcal{UM}(Z)$ of the separating function is the metric space
		
		\begin{equation*}
			\mathcal{UM}(Z,\rho_0) \coloneqq  \{ \rho \ | \ \rho \ \text{is a separating function Moebius equivalent to} \ \rho_0 \},
		\end{equation*}
		equipped with the metric
		\begin{equation*}
			d_{\MM(Z)}( \rho_1, \rho_2 ) \coloneqq  \left|\left| \log \frac{d\rho_2}{d\rho_1} \right|\right|_{\infty},
		\end{equation*}
		for $\rho_1, \rho_2 \in \mathcal{UM}(Z,\rho_0)$.
		
		\medskip	
	\end{subdefinition}
	
	We also have the map
	\label{embedding}
	\begin{align*}
		i_{\rho_0} \colon  (\mathcal{UM}(Z,\rho_0), d_{\MM})& \to ( C(Z),\|\cdot\|_{\infty} ) \\
		\rho & \mapsto \log \frac{d\rho}{d\rho_0}
	\end{align*}
	which is a  surjective isometry with the inverse given by the map \label{isometry}
	\begin{align*}
		E_{\rho_0}\colon ( C(Z), \|\cdot\|_{\infty} ) & \to  (\mathcal{UM}(Z,\rho_0), d_{\MM(Z)})\\
		\tau & \mapsto E_{\rho_0}(\tau)
	\end{align*}
	where for $\tau \in C(Z)$, the separating function $\rho = E_{\rho_0}(\tau) \in \mathcal{UM}(Z)$ is defined by
	\begin{equation*}
		\rho(\xi, \eta)^2 \coloneqq  e^{\tau(\xi)} e^{\tau(\eta)} \rho_0(\xi, \eta)^2
	\end{equation*}
	for all $\xi, \eta  \in Z$. Then $\log \frac{d\rho}{d\rho_0} = \tau$.
	\medskip
	
	\noindent Now we define $\colon$ 
	
	\begin{subdefinition}[{\bf Antipodal function}, \cite{biswas2024quasi}]\label{antipodal function}
		Let $Z$ be a compact metrizable space. An antipodal function on $Z$
		is a separating function $\rho_0$ which satisfies the following two properties$\colon$ \medskip
		
		(1) $\rho_0$ has diameter one, i.e. $\sup_{\xi, \eta \in Z} \rho_0(\xi, \eta) = 1$.\medskip
		
		(2) $\rho_0$ is antipodal, i.e. for all $\xi \in Z$ there exists $\eta \in Z$ such that $\rho_0(\xi, \eta) = 1$.\medskip
		
		\noindent An antipodal space is a compact semi-metric space $(Z,\rho_0)$ such that $\rho_0$ is an antipodal function. 
		
	\end{subdefinition}

	\noindent A {\it quasi-metric antipodal space} is an antipodal space where the antipodal function is a quasi-metric.
	\medskip
	
	\noindent Now we define the space of our interest$\colon$ 
	
	\begin{subdefinition}[{\bf Moebius space of antipodal function, \cite{biswas2024quasi}}]
		Let $(Z,\rho_0)$ be an antipodal space. The Moebius space $\MM(Z,\rho_0)$ is defined to be the metric space
		\begin{equation*}
			\MM(Z,\rho_0) \coloneqq  \{ \rho \in \mathcal{UM}(Z,\rho_0) \ | \ \rho \ \text{is an antipodal function} \ \}
		\end{equation*}
		(equipped with the metric $d_{\MM}$ as defined above).
	\end{subdefinition}
	
	\noindent Observe that $\mathcal{UM}(Z,\rho_0)$ is not necessarily proper. But we have
	$\MM(Z,\rho_0)$ is always proper (i.e. closed bounded balls are compact) and a closed subset of $\mathcal{UM}(Z,\rho_0)$(see  \cite[Lemma 2.7]{biswas2024quasi}).
	
	We recall the following properties of $\MM(Z,\rho_0)$ from  \cite{biswas2024quasi}.
	
	\begin{subprop}[Biswas, \cite{biswas2024quasi}]
		The space $\MM(Z,\rho_0)$ is 
		\medskip
		
		\begin{enumerate}
			\item unbounded
			\medskip
			\item contractible 
			\medskip
			\item geodesic, i.e. any two points can be joined by a geodesic
			\medskip
			\item geodesically complete, i.e. any geodesic segment can be extended to a bi-infinite geodesic.
		\end{enumerate}
	\end{subprop}
	
	\noindent {\bf Notation$\colon$ } 
	\begin{enumerate}[{\it (I)}]
		\item For brevity of notation sometimes we shall write $\mathcal M(Z)$ instead of $\MM(Z,\rho_0)$. Whenever we write $\mathcal M(Z)$, it should be understood that there exists a fixed antipodal function $\rho_0$ for which $\mathcal M(Z)\coloneqq \mathcal M(Z,\rho_0)$. Also, because of the identification above, we shall often treat elements of $\MM(Z)$, that is, antipodal functions $\rho$, as a continuous function $\tau_{\rho}=\log\frac{d\rho}{d\rho_0}$ on $Z$.
		
		\item For a metric space $X$, unless otherwise mentioned, the set
		\begin{equation*}
			B_{X}(o,R)\coloneqq \{\ x\in X \ | \ d_{X}(o,x)\le R \ \}\subset X
		\end{equation*} denotes closed ball of radius $R>0$ centered at $o$ in $X$. For example,
		$B_{\MM(Z)}(\rho,R)\subset \MM(Z)$ denotes a closed ball of radius $R>0$ around $\rho$ in $\MM(Z)$. 
	\end{enumerate}
	
	\medskip
	
	\subsection{Gromov product and Moebius spaces}\label{Gromov ip}\hfill\\
	
	\noindent Consider a metric space $(X,d)$, for any point $x\in X$ we define the Gromov product with respect to $x$,
	\begin{equation*}
		(y|z)_x\coloneqq \frac{1}{2}\big(d(y,x)+d(z,x)-d(y,z)\big).
	\end{equation*}
	The following \Cref{gromov ip upper bound} relates the Gromov product in the Moebius space $\MM(Z,\rho_0)$, with respect to  $\rho_0$, with the antipodal function $\rho_0$, which will be helpful later. Also see \cite[Theorem 6.3]{biswas2024quasi}.
	
	\begin{sublemma}[Biswas, \cite{biswas2024quasi} Lemma 6.2]\label{gromov ip upper bound}
		Let $\alpha,\beta\in \MM(Z,\rho_0)$, suppose $\xi\in \argmx \frac{d\alpha}{d\rho_0}$ and $\eta\in \argmx\frac{d\beta}{d\rho_0}$ then we have 
		\begin{equation*}
			(\alpha|\beta)_{\rho_0}\le \log\frac{1}{\rho_0(\xi,\eta)}
		\end{equation*}
		(with the convention that the right-hand side in the inequality is $+\infty$ if $\xi=\eta$).
	\end{sublemma}
	\medskip
	
	Now we recall the notion of {\it Gromov product space} introduced in \cite[Section 5]{biswas2024quasi}.
	
	\begin{subdefinition}[{\bf Gromov product space, \cite{biswas2024quasi}}]
		Let $X$ be a proper, geodesic, and geodesically complete metric space. We say that $X$ is a Gromov product space if there exists a second countable Hausdorff compactification of $\hat{X}$ of $X$ such that for any $x\in X$, the Gromov product map $(\cdot|\cdot)_x\colon  X\times X\to [0,+\infty)$ extends to a continuous function $(\cdot|\cdot)_x\colon \hat{X}\times \hat{X}\to [0,+\infty]$ with the property that for $\xi,\eta\in \hat{X}\setminus X$, we have $(\xi|\eta)_x=+\infty$ if and only if $\xi=\eta$
		
		We call $\hat{X}$ the Gromov product compactification of $X$,
		and $\partial_P X\coloneqq \hat{X}\setminus X$ is defined to be the Gromov product boundary of $X$.
	\end{subdefinition} 
	For a Gromov product space, the Gromov compactification is unique up to equivalence (see \cite{biswas2024quasi} Proposition 5.2).
	
	Examples of Gromov product spaces include {\it``good"} (i.e. proper, geodesic, geodesically complete, boundary continuous) Gromov hyperbolic spaces say $X$. The visual compactification of $X$ is the Gromov product compactification, then the Gromov boundary $\partial X$ is the Gromov product boundary. In particular, proper geodesically complete CAT$(-1)$ spaces are Gromov product spaces. Another class of examples for Gromov product spaces, which are not necessarily Gromov hyperbolic spaces, arise in Hilbert geometry. Consider strictly convex bounded domains $\Omega\subseteq \R^n$, having $C^1$-smooth boundary, equipped with the Hilbert metric $d_{\Omega}$ (see \cite{papadopoulos2014handbook} for exposition). Then $(\Omega,d_{\Omega})$ are Gromov product spaces (see \cite{benoist2006convex}), and the Gromov product boundary is  $\partial \Omega$, the boundary of the domain in $\R^n$. 
	
	Given a Gromov product space $X$, the Gromov product boundary $\partial_P X$, equipped with the visual antipodal function defined by $\rho_x\coloneqq \exp(-(\cdot|\cdot)_x)\colon \partial_P X\times \partial_P X\to [0,1]$ for any $x\in X$, is an antipodal space. Any two visual antipodal functions are Moebius equivalent. Biswas showed that, if we consider the Moebius space associated with $(\partial_P X,\rho_x)$, that is $\MM(\partial_P X)$, as above, then we have a natural isometric embedding $i_X\colon X\to\MM(\partial_P X),\ a\mapsto \rho_a$ called {\it visual embedding} (see \cite{biswas2024quasi}, cf. \cite{biswas2015moebius}). Moreover, the image of $X$ under the visual embedding is $7\delta$-dense in $\MM(\partial X)$, when $X$ is a good Gromov ($\delta$)-hyperbolic space (see \cite[Theorem 1.4]{biswas2024quasi}).   
	
	\begin{sublemma}[Biswas, \cite{biswas2024quasi}, Lemma 5.7] \label{Moebius homeomorphism}
		Let $X$ and $Y$ be Gromov product spaces. Fix some $x\in X$ and $F\colon X\to Y$ be an isometric embedding with $F(x)=y$. Then $F$ extends to a Moebius embedding of antipodal spaces, $f\colon (\partial_P X,\rho_x)\to (\partial_P Y,\rho_y)$. If $F$ is an isometry, then the Moebius map $f$ is a homeomorphism. 
	\end{sublemma}
	
	Consider an antipodal space $(Z,\rho)$. A {\it filling} of $(Z,\rho)$ is a pair $(X,f)$, where $X$ is a Gromov product space equipped with a Moebius homeomorphism $f\colon \partial_P X \to Z$ \footnote{Here $f$ preserves the canonical cross-ratios given by the visual metrics.}. If $X$ is a CAT$(-1)$ space then we say $(Z,\rho)$ is `CAT$(-1)$ {\it  fillable}'. We say two fillings $(X,f)$ and $(Y,g)$ are equivalent if the Moebius homeomorphism $g^{-1}\circ f\colon \partial_P X\to \partial_P Y$ extends to an isometry $H\colon X\to Y$ and we write $(X,f)\simeq(Y,g)$. We say $(X,f)\leq (Y,g)$ if the Moebius homeomorphism $g^{-1}\circ f$ extends to an isometric embedding. This relation defines a partial ordering on the equivalence classes of fillings of $(Z,\rho)$ (see \cite[Proposition 5.13]{biswas2024quasi}). 
	\begin{subrmk}\label{induced isometry}
		Let $(Z_1,\rho_1)$ and $(Z_2,\rho_2)$ be two antipodal spaces and let $f\colon (Z_2,\rho_2)\to (Z_1,\rho_1)$ be a Moebius homeomorphism between them. Then this induces an isometry by push-forward $f_*\colon \MM(Z_2)\to \MM(Z_1)$ defined by
		\begin{equation*}
			f_*(\rho)(\xi,\eta)\coloneqq \rho\ (f^{-1}(\xi),f^{-1}(\eta)) \quad \hbox{ for }\ \xi,\eta\in Z_1
		\end{equation*}
		for $\rho\in \MM(Z_2)$ (see \cite[Section 6]{biswas2024quasi}). 
		
		
		In particular, if $(X,\varphi)$ is a filling of the antipodal space $(Z_1,\rho_1)$, then, the Moebius homeomorphism $\varphi\colon \partial X\to Z_1$ induces an isometry $\varphi_*\colon \MM(\partial X)\to \MM(Z_1)$, with inverse given by the pull-back $\varphi_*\colon \MM(Z_1)\to \MM(\partial X)$, 
		where 
		$\varphi^*=(\varphi^{-1})_*$.
	\end{subrmk}
	
	\begin{subdefinition}[{\bf Maximal Gromov product space, \cite{biswas2024quasi}}]
		Let $X$ be a Gromov product space, and let $Z=\partial_P X$ be its Gromov product boundary equipped with a visual antipodal function $\rho_x$ for some $x\in X$. We say $X$ is a maximal Gromov product space if for any filling $(Y,g)$ of $(Z,\rho_x)$ we have $(Y,g)\leq(X,id_{Z})$.
	\end{subdefinition} 
	
	For an antipodal space $(Z,\rho_0)$, the Moebius space $\MM(Z)$ is a maximal Gromov product space and the Gromov product boundary $\partial_P \MM(Z)$ equipped with the visual antipodal function based at $\rho_0$ is identified with $(Z,\rho_0)$ (see \cite[ Theorem 6.3]{biswas2024quasi}). The Moebius space $\MM(Z)$ is the unique (up to isometry) maximal element among the fillings of $(Z,\rho_0)$. \begin{subprop}[Biswas, \cite{biswas2024quasi}, Theorem 1.7]\label{structure theorem}
		Let $X$ be a Gromov product space. $X$ is maximal if and only if $X$ is isometric to $\MM(\partial_P X)$, via the visual embedding $i_X\colon  X\to \MM(\partial_P X),(x\mapsto \rho_x)$.
	\end{subprop}
	
	Given an antipodal space $(Z, \rho_0)$, the Moebius space $\MM(Z, \rho_0)$ is a Gromov hyperbolic space if and only if $(Z, \rho_0)$ is a quasi-metric space (see \cite[Theorem 1.2]{biswas2024quasi}).  
	
	Consequently, observe that for a quasi-metric antipodal space $(Z, \rho_0)$, every good filling $(X, f)$ is a hyperbolic filling (i.e., $X$ is a good Gromov hyperbolic space) because $X$ is isometrically embedded in the Gromov hyperbolic space $\MM(Z, \rho_0)$. Also $\MM(Z,\rho_0)$ is a maximal Gromov product space. Thus, $\MM(Z, \rho_0)$ is the maximal hyperbolic filling of the quasi-metric antipodal space $(Z, \rho_0)$, and is therefore called a {\it maximal Gromov hyperbolic space} (see \cite[Section 8]{biswas2024quasi}). If $(Z,\rho_0)$ is an antipodal space of cardinality $n<\infty$, then it is a quasi-metric space, and in that case, the maximal Gromov product space $\MM(Z,\rho_0)$ is a maximal Gromov hyperbolic space with $n$ points on the boundary. For discussion on the polyhedral structure of maximal Gromov hyperbolic spaces with finite boundary, see \cite{biswas2024polyhedral}.

	There is an equivalence of categories between antipodal spaces and maximal Gromov product spaces, where the morphisms in the first category is given by Moebius homeomorphisms and in the latter by isometries (\cite[Section 6]{biswas2024quasi}).
	
	\medskip
	
	\subsection{Injective metric spaces, maximal Gromov product spaces and geodesic bi-combing}\label{geodesic bicombing subsection}\hfill\\
	
	\noindent We recall the definition of {\it injective metric spaces} 
	(also known as {\it hyperconvex spaces}), first introduced by Aronszajn and Panitchpakdi \cite{aronszajn1956extension}. A metric space $X$ is called injective if for any metric space $A$ and any subspace $B\subset A$ any $1-$Lipschitz map $f \colon B\to X$ has a $1-$Lipschitz extension $F \colon A\to X$ such that $F|_B=f$ (see \cite[Chapter 3]{petrunin2023pure}, \cite[Section 2]{lang2013injective}). Every injective metric space $X$ is complete and geodesic. Examples of injective spaces are:  the real line, closed intervals in the real line, geodesically complete metric trees and $(\R^n,\|\cdot\|_\infty)$ (more generally $l^\infty(\mathcal K)$ for any arbitrary set $\mathcal K$). There is a well-known characterization of injective metric spaces, which we shall state here without proof.
	
	\medskip
	
	\begin{subprop}[see \cite{aronszajn1956extension}, \cite{lang2013injective}]\label{hyperconvex thm}
		Let $X$ be a metric space. Then the following are equivalent:
		\begin{enumerate}[{\it (1)}]
			\item X is an injective metric space.\medskip
			
			\item If $i \colon X\to Y$ is an isometric embedding into any metric space $Y$ there exists a 1-Lipschitz map $\pi \colon Y\to X$ such that $\pi\circ i= id|_X$, i.e. $X$ is an `absolute 1-Lipschitz retract'.\medskip
			
			\item If $\{B_X(x_i,R_i)\}_{i\in I}$ is a family of closed balls in $X$ such that 
			\begin{equation*}\label{intersection}
				R_i+R_j\ge d_X(x_i,x_j)
			\end{equation*} for any $i,j\in I$ then 
			$\cap_{i\in I}B(x_i,R_i)\neq \emptyset$, i.e. $X$ is `hyperconvex'.\label{hyp}
		\end{enumerate}
	\end{subprop}
	\noindent There is a close connection between the injective metric spaces and maximal Gromov product spaces.
	
	\begin{subtheorem}[Biswas,\cite{biswas2024quasi}, Theorem 1.10]\label{injectivity theorem}
		Let $X$ be a Gromov product space.
		Then $X$ is maximal if and only if $X$ is injective.
		In particular, for any antipodal space $(Z,\rho)$, the Moebius space $\MM(Z,\rho)$ is injective.
	\end{subtheorem}
	
	\begin{subdefinition}[{\bf Geodesic bi-combing}]\label{geodesic bicombing}
		A geodesic bi-combing $\Gamma$ on a metric space $(X,d)$ is a map $\Gamma\colon X \times X \times [0,1]\to X$ such that for every ordered pair $(x,y)\in X\times X$ the curve $\Gamma_{xy}\coloneq\Gamma(x,y,\cdot)$ is a constant speed parametrization of a geodesic from $x$ to $y$,i.e. $\Gamma_{xy}(0)=x$, $\Gamma_{xy}(1)=y$ and \begin{equation*}
			d(\Gamma_{xy}(s),\Gamma_{xy}(t))=(t-s)\cdot d(x,y)
		\end{equation*} for all $0\le s\le t\le 1$.
	\end{subdefinition}
	We say a geodesic bi-combing $\Gamma$ on a metric space $(X,d)$ is {\it convex} if for all $x,y,p,q\in X$ the function \begin{equation*}
		t\mapsto d(\Gamma_{xy}(t),\Gamma_{pq}(t))
	\end{equation*}
	is convex on $[0,1]$. And, we say $\Gamma$ is {\it conical} if for all $t\in [0,1]$
	\begin{equation*}
		d(\Gamma_{xy}(t),\Gamma_{pq}(t))\le(1-t)\cdot d(x,p)+t\cdot d(y,q).
	\end{equation*}Each of these two conditions implies $\Gamma$ is continuous.
	\begin{subprop}[\cite{lang2013injective}, Proposition 3.8]
		Every injective metric space $X$ admits a conical geodesic bi-combing $\Gamma$ such that, for all $x,y \in X$ and $t \in [0,1]$,\\
		\begin{enumerate}[(1)]
			\item $\Gamma_{xy}(t) = \Gamma_{yx}(1-t)$ (we say $\Gamma$ is `reversible').\\
			\item $I \circ \Gamma_{xy} = \Gamma_{I(x)I(y)}$ for every isometry $I\colon X\to X$.
		\end{enumerate}
	\end{subprop}
	For a proper geodesic metric space $X$, existence of a geodesic bi-combing guarantees $X$ admits a {\it convex} geodesic bi-combing (see \cite[Theorem 3.4]{descombes2015bicombings}). Thus:
	\begin{subrmk}
		Any maximal Gromov product space $X$ admits a convex geodesic bi-combing $\Gamma$. (from \Cref{injectivity theorem})
	\end{subrmk}
	
	\medskip
	
	\subsection{Semi-metric spaces and antipodal spaces}\hfill\\
	
	\noindent In this subsection, we collect several simple facts which will be useful later on. We recall that a (compact) semi-metric space is a compact metrizable space equipped with a separating function (\Cref{semi-metric}). 
	\begin{sublemma}
		\label{lemma 1.1}
		Let $(Z,\rho)$ be a compact semi-metric space and let $X$ be a compact subset of $C(Z)$. For every $\epsilon>0$ there exists $\delta>0$ such that if $\xi,\eta\in Z$ and $\rho(\xi,\eta)<\delta$ then 
		\begin{equation*}
			\big|f(\xi)-f(\eta)\big|<\epsilon
		\end{equation*}  
		for all $f\in X$.	
	\end{sublemma}
	\begin{proof}
		Suppose not, then there exists $\epsilon_0>0$ such that we have sequences $\{\xi_n\},\{\eta_n\}$ in $Z$ and $\{f_n\}$ in $X$ with $\rho(\xi_n,\eta_n)\to 0$ as $n\to \infty$ and $f_n(\xi_n)-f_n(\eta_n)>\epsilon_0$ for all $n$. As $Z$ a compact metric space and $X$ compact subset we pass to sub-sequence $\xi_n\to \xi\in Z$, $\eta_n\to \eta\in Z$ and $f_n\to f\in X$, i.e. $f_n$ converges to $f $ uniformly, as $n\to \infty$. By continuity of $\rho$ we have $\rho(\xi_n,\eta_n)\to \rho(\xi,\eta)=0$, as $n\to \infty$, hence $\xi=\eta$. And by uniform convergence of $f_n$ we have $f_n(\xi_n)\to f(\xi)$ and $f_n(\eta_n)\to f(\eta)$, as $n\to \infty$, hence $(f_n(\xi_n)-f_n(\eta_n))\to f(\xi)-f(\eta)=0$, as $n\to \infty$. This contradicts the initial hypothesis $f_n(\xi_n)-f_n(\eta_n)>\epsilon_0$. 
	\end{proof}
	
	\medskip
	
	\noindent This Lemma has the following consequences, which will be used several times later.
	\begin{subcoro}
		\label{Coro 1}
		Let $(Z,\rho_0)$ be an antipodal space, and let ${B}_{\MM(Z)}(\rho_0,R)$ denote the closed ball of radius $R>0$ around $\rho_0$ in the Moebius space $\MM(Z)$.  For every $\epsilon>0$ there exists $\delta>0$ such that if $\xi,\eta\in Z$ and $\rho_0(\xi,\eta)<\delta$ then 
		\begin{equation*}
			\big|\tau_\rho(\xi)-\tau_\rho(\eta)\big|<\epsilon
		\end{equation*}
		for all $ \rho\in {B}_{\MM(Z)}(\rho_0,R)$
	\end{subcoro}
	
	\medskip
	\noindent In the absence of the triangle inequality (for semi-metric spaces), the following two equivalent facts will be crucially utilized multiple times later.
	
	\begin{subcoro}\label{Coro 2}
		Let $(Z,\rho)$ be a compact semi-metric space. For every $\epsilon>0$ there exists $\delta>0$ such that if $\xi,\xi',\eta,\eta'\in Z$ with $\max\{\rho(\xi,\xi'),\rho(\eta,\eta')\}<\delta$ then \begin{equation*}
			\big|\rho(\xi,\eta)-\rho(\xi',\eta')\big|<\epsilon.
		\end{equation*}
	\end{subcoro}
	\medskip
	
	\begin{subcoro}
		\label{Coro 3}
		Let $(Z,\rho)$ be an antipodal space. For every $\epsilon>0$ there exists $\delta>0$ such that if $\xi,\eta\in Z$ and $\rho(\xi,\eta)<\delta$ then
		\begin{equation*}
			\sup_{\zeta\in Z}\ \big|\rho(\xi,\zeta)-\rho(\eta,\zeta)\big|<\epsilon
		\end{equation*}
	\end{subcoro}
	\medskip
	
	\begin{sublemma}
		\label{rho distance}
		Let $(Z,\rho)$ be a compact semi-metric space, and let $A$ be a closed subset of $Z$. Then the function $f\colon Z\to \mathbb{R}$ defined by $f(\xi)=\rho(\xi,A)\coloneqq \inf\{\ \rho(\xi,\eta)\ |\ \eta\in A\ \}$ is continuous.
	\end{sublemma}
	\begin{proof} 
		Observe for every $\xi\in Z$ there exists $\eta\in A$ such that $\rho(\xi,\eta)=\rho(\xi,A)$. 
		If possible let us suppose $f$ is not continuous. Then there exists a point $\xi \in Z $ and a sequence $\{\xi_n\}$ in $Z$ such that $\xi_n\to \xi$ as $n\to \infty$ and $|f(\xi_n)-f(\xi)|>c>0$ for some $c>0$. Now there exists $\eta_n\in Z$ such that $f(\xi_n)=\rho(\xi_n,A)=\rho(\xi_n,\eta_n)$ for all $n\in \mathbb{N}$. There is a sub-sequence of $\{\eta_n\}$ such that we have $\eta_{n_k}\to \eta\in A$ then from continuity of $\rho$ we get,
		$$\lim_{k\to\infty}f(\xi_{n_k})=\lim_{k\to\infty}\rho(\xi_{n_k},\eta_{n_k})=\rho(\xi,\eta)\ge\rho(\xi,A)=f(\xi)$$
		by the definition of $\rho(\xi,A)$. Again there exists $\tilde{\eta}\in A$ such that $f(\xi)=\rho(\xi,\tilde{\eta})$, then $f(\xi_{n_k})\le \rho(\xi_{n_k},\tilde{\eta})$ for all $k\in \mathbb{N}$. Therefore, 
		$$\lim_{k\to\infty}f(\xi_{n_k})\le \lim_{k\to\infty}\rho(\xi_{n_k},\tilde{\eta})=\rho(\xi,\tilde{\eta})=f(\xi)$$
		So we get 
		$$\lim_{k\to\infty}f(\xi_{n_k})=f(\xi)$$ 
		which implies $|f(\xi_n)-f(\xi)|<c$ for infinitely many $n$, this contradicts the initial hypothesis.
	\end{proof}
	
	\medskip
	
	\subsection{Discrepancy, Antipodal flow and the Antipodalization map $\PP_\infty$ }\label{discrepancy}\hfill\\
	
	\noindent Let $(Z,\rho_0)$ be an antipodal space. Given a bounded function $\tau$ on $Z$ and $\rho\in \MM(Z)$, recall that the {\it discrepancy} of $\tau$ with respect to $\rho$ is the function $D_{\rho}(\tau)$ (see Definition 3.1 \cite{biswas2024quasi}), for $\xi \in Z$ defined by,
	\begin{equation}\label{discrepancy definition}
		D_{\rho}(\tau)(\xi)\coloneqq \sup_{\eta\in Z\setminus\{\xi\}}\tau(\xi)+\tau(\eta)+\log\rho(\xi,\eta)^2.
	\end{equation}
	Observe that $D_{\rho}(\tau)$ might not be a continuous function on $Z$, but it is bounded with
	\begin{equation}
		\big\|D_{\rho}(\tau)\big\|_\infty\le 2\big\|\tau\big\|_\infty.
	\end{equation}
	(see \eqref{discrepancy 2 lipschitz}) below and note $\tau\equiv0$ then $D_{\rho}(\tau)\equiv 0$. 
	
	For $\tau\in C(Z)$, a continuous function, we have that $D_\rho(\tau)$ is continuous, and moreover if $\bar{\rho}\coloneqq E_\rho(\tau)$ then $\colon$ \\ 
	\begin{center}
		$\bar{\rho}$ is an antipodal function, i.e. $\bar{\rho}\in \MM(Z)$
		if and only if $D_\rho(\tau)\equiv 0$.
	\end{center}
	\hfill\\
	\noindent Now, note the following fact about the discrepancy.
	
	\begin{sublemma}[cf. \cite{biswas2024quasi} Lemma 3.2]\label{discrepancy 2 lipschitz}
		Let $\tau_1, \tau_2$ be two bounded functions on antipodal space $(Z,\rho)$, then 
		\begin{equation}\label{discrepancy 2 lipschitz ineq}
			\big\|D_{\rho}(\tau_1)-D_{\rho}(\tau_2)\big\|_\infty\le2\cdot\big\|\tau_1-\tau_2\big\|_\infty.
		\end{equation}
	\end{sublemma}
	
	\begin{proof}
		We know,  $\tau_1\le \tau_2 + \|\tau_1-\tau_2\|_\infty$ hence,
		\begin{equation*}
			\tau_1(\xi)+\tau_1(\eta)+\log\rho(\xi,\eta)^2\le\tau_2(\xi)+\|\tau_1-\tau_2\|_\infty+\tau_2(\eta)+\|\tau_1-\tau_2\|_\infty+\log\rho(\xi,\eta)^2
		\end{equation*}
		for all $\xi,\eta\in Z$ and $\xi\not=\eta$.
		Then 
		\begin{align*}
			\sup_{\eta\in Z\setminus \{\xi\}}\tau_1(\xi)+\tau_1(\eta)+\log\rho(\xi,\eta)^2&\le \sup_{\eta\in Z\setminus \{\xi\}}\tau_2(\xi)+\tau_2(\eta)+\log\rho(\xi,\eta)^2+2\cdot\|\tau_1-\tau_2\|_\infty\\ \implies D_\rho(\tau_1)(\xi)& \le D_\rho(\tau_2)(\xi)+2\cdot\|\tau_1-\tau_2\|_\infty
		\end{align*}
		for all $\xi\in Z$. Similarly, $ D_\rho(\tau_2)(\xi) \le D_\rho(\tau_1)(\xi)+2\cdot\|\tau_1-\tau_2\|_\infty$
		for all $\xi\in Z$. Hence the result. 
	\end{proof}
	
	We now recall the definition of the antipodalization map. Given an antipodal space $(Z,\rho_0)$, fix $\rho\in \MM(Z)$.  Consider the ODE on Banach space $(C(Z),\|\cdot\|_\infty)$
	\begin{equation}\label{ODE}
		\frac{d}{dt}(\tau_t)=-D_{\rho}(\tau_t), 
	\end{equation} which is defined by the vector field $-D_{\rho}\colon (C(Z),\|\cdot\|_\infty) \to (C(Z),\|\cdot\|_\infty)$ on $(C(Z),\|\cdot\|_\infty)$, where $D_{\rho}$ is the discrepancy. It is called the {\it $\rho$-antipodal flow}. For every $\tau\in C(Z)$, the unique solution to the $\rho$-antipodal flow exists, that is a $C^1$ curve in the Banach space $(C(Z),\|\cdot\|_\infty)$, $t\in I\subset \R \mapsto \tau_t\in C(Z)$, such that $\tau_0=\tau$ and $\tau_t$ is defined for all time $t\ge0$ (i.e. $I=(-a,\infty)$ for some $a>0$).   From \cite{biswas2024quasi}, Theorem 3.15, we have the limit 
	\begin{equation*}
		\tau_\infty=\tau_\infty(\tau,\rho)\coloneqq \lim_{t\to \infty} \tau_t
	\end{equation*}
	exists in $C(Z)$. Moreover, we have an exponential convergence estimate,
	\begin{equation}\label{expconv}
		\big\|\tau_\infty-\tau_t\big\|_\infty\le 4 \cdot \big\|D_{\rho_1}(\tau)\big\|_\infty\cdot e^{-t/2}
	\end{equation}
	for $t\ge 0$, also $D_\rho(\tau_\infty)\equiv 0$ hence $\rho_\infty\coloneqq E_{\rho}(\tau_\infty)\in \MM(Z)$.
	
	\begin{subdefinition}{\bf (The antipodalization map, \cite{biswas2024quasi})}\label{antipodalization}
		Let $(Z,\rho_0)$ be an antipodal space. The antipodalization map is the map
		\begin{align*}
			\PP_\infty \colon  C(Z) \times \MM(Z) &\to \MM(Z)\\
			(\quad\tau\quad,\quad\rho\quad) &\mapsto\ \rho_\infty\ \coloneqq\ E_{\rho}(\tau_\infty),
		\end{align*}
		where, $\tau_\infty\coloneqq \tau_\infty(\tau, \rho)$ as defined above.
	\end{subdefinition}
	\noindent For detailed discussion on antipodal flow, long-term existence of solution to ODE \eqref{ODE} and properties of the antipodalization map, see \cite[Section 3]{biswas2024quasi}.
	
	Here, we recall a fact about the antipodalization map without proof, which will be used later. This is a direct consequence of Lemma 4.1 in \cite{biswas2024quasi}.
	
	\begin{sublemma}(Biswas, \cite{biswas2024quasi})\label{convexcombi}
		Let $(Z,\rho_0)$ be an antipodal space. For $\rho\in \MM(Z)$ and for $t\in[0,1]$ define \begin{equation*}
			\rho_t\coloneqq \PP_\infty(t\cdot\tau_\rho,\rho_0)
		\end{equation*}
		then $d_{\MM(Z)}(\rho_t,\rho_0)=t\cdot d_{\MM(Z)}(\rho,\rho_0)$. 
	\end{sublemma}
	
	
	\subsection{Partitions of unity on compact semi-metric spaces and smoothing operators}\label{PoU}\hfill\\
	
	\noindent For a given finite open cover $\mathcal{U}\coloneqq \{U_i\}_{i=1}^k$ of a compact semi-metric space $(Z,\rho)$ consider the collection of functions $\{f_i\colon Z\to \mathbb{R}\}_{i=1}^k$ defined by $f_i(\xi)=\rho(\xi,U_i^c)$ where $U_i^c=Z\setminus U_i$. Then by \Cref{rho distance} $f_i$'s are continuous and $f_i\ge 0$. Observe that $\sum_if_i\colon Z\to R$ is a strictly positive function on $Z$. Then, we can define 
	\begin{equation*}
		P_i\coloneqq \frac{f_i}{\sum_if_i}\colon Z\to \mathbb{R}
	\end{equation*} 
	such that $0\le P_i\le1$ for every $P_i$ and $\sum_i P_i\equiv 1$ on $Z$. We call this collection $\{P_i\}_{i=1}^k$ a \textbf{partition of unity} on $(Z,\rho)$ with respect to the finite cover $\mathcal{U}$. 
	
	For our purposes, in a compact semi-metric space $(Z, \rho)$, we define open balls as $B_{\rho}(\xi, \delta) \coloneqq  \{\eta \in Z \ | \ \rho(\xi, \eta) < \delta\}$ for $\xi \in Z$. Then, for a finite $\delta$-net $F \subset Z$ with respect to $\rho$, the corresponding collection $\mathcal{U} \coloneqq  \{B_{\rho}(\zeta, \delta)\}_{\zeta \in F}$ will form a finite open cover of $Z$. The partition of unity on $(Z,\rho)$ with respect to the cover $\mathcal U$, will be called a ``{\it partition of unity on $(Z,\rho)$ associated to the $\delta$-net $F$}", denoted by $\{P_\zeta\}_{\zeta\in F}$. Corresponding to this partition of unity we define the {\bf smoothing operator} on $(B(Z),\|\cdot\|_\infty)$, the space of bounded real valued functions on $Z$, as the linear map
	\begin{equation*}
		\begin{split}
			C\colon (B(Z),\|\cdot\|_{\infty})\to (C(Z),\|\cdot\|_\infty),
		\end{split}
	\end{equation*}
	defined by,
	\begin{equation}
		C(\tau)(\xi)=\sum_{\zeta\in F}\tau(\zeta)P_\zeta(\xi), \quad \hbox{for} \quad \xi\in Z,
	\end{equation} 
	for $\tau\in (B(Z),\|\cdot\|_\infty)$. It is clear from the definition that $C(\tau)$ is continuous for any $\tau \in (B(Z),\|\cdot\|_\infty)$.
	\begin{sublemma}
		\label{smoothing of tau lemma}
		Given $\epsilon>0$, let $\tau$ be a bounded function on $(Z,\rho)$ with property that there exists $\delta>0$ such that, for any $\xi,\eta\in Z$ if $\rho(\xi,\eta)<\delta$ then
		\begin{equation*}
			\big|\tau(\xi)-\tau(\eta)\big|<\epsilon.
		\end{equation*}
		Consider a finite $\delta$-net $F\subset Z$ with respect to $\rho$. Let $\{P_\zeta\}_{\zeta\in F}$ denote the partition of unity on $(Z,\rho)$ associated to the finite $\delta$-net $F$. Then $C(\tau)$ satisfies, 
		\begin{equation*}
			\big\|C(\tau)-\tau\big\|_\infty<\epsilon \quad \hbox{and} \quad \big\|C(\tau)\big\|_\infty\le\big\|\tau\big\|_\infty.
		\end{equation*} 
	\end{sublemma}
	\begin{proof}
		Given $\xi\in Z$ observe, $\zeta\in F$ such that $P_\zeta(\xi)>0$  if and only if $\rho(\xi,\zeta)<\delta$. Define $S(\xi)\coloneqq \{\eta\in F\ |\ \rho(\xi,\eta)<\delta\}$. Therefore, we can write
		\begin{equation*}
			\sum_{\zeta\in S(\xi)}P_\zeta(\xi)=1\ \text{ and }\ C(\tau)(\xi)=\sum_{\zeta\in S(\xi)}\tau(\zeta)\cdot P_\zeta(\xi)
		\end{equation*}
		for $\xi\in Z$. Then we have, 
		\begin{align*}
			\Big|C(\tau)(\xi)-\tau(\xi)\Big|&=\Big|\sum_{\zeta\in S(\xi)}\tau(\eta_i)\cdot P_\zeta(\xi)-\tau(\xi)\bigg(\sum_{\eta_i\in S(\xi)}P_\zeta(\xi)\bigg)\Big|\\
			&=\Big|\sum_{\zeta\in S(\xi)}(\tau(\eta_i)-\tau(\xi))\cdot P_\zeta(\xi)\Big|\\
			&\le \sum_{\zeta\in S(\xi)}\Big|\tau(\zeta)-\tau(\xi)\Big|\cdot P_\zeta(\xi)\\
			&<\epsilon.
		\end{align*}
		Therefore, we have $\|C(\tau)-\tau\|_\infty<\epsilon$
		and also from definition $\|C(\tau)\|_\infty\le\|\tau\|_\infty$.
	\end{proof}

	\medskip
	
	\subsection{Equicontinuous families of semi-metric spaces}\label{equicontinuity definition}\hfill\\
	
	\noindent This notion will be useful, and we shall see this come up in our discussion often.
	\begin{subdefinition}
		Given a family of compact semi-metric spaces $\{(Z_\lambda,\rho_\lambda)\}_{\lambda\in \Lambda}$, it is said to be equicontinuous if one of the following equivalent conditions hold$\colon$ 
		\begin{enumerate}
			\item for every $\epsilon>0$ there exists $\delta=\delta(\epsilon)>0$ such that for all $\lambda\in \Lambda$, if $\xi,\eta\in Z_\lambda$ with $\rho_\lambda(\xi,\eta)<\delta$ then 
			\begin{equation*}
				\sup_{\zeta\in Z_\lambda}\ \big|\rho_\lambda(\xi,\zeta)-\rho_\lambda(\eta,\zeta)\big|<\epsilon.
			\end{equation*}
			\medskip
			
			\item for every $\epsilon>0$ there exists $\delta=\delta(\epsilon)>0$ such that for all $\lambda\in \Lambda$, if $\xi,\xi',\eta,\eta'\in Z_\lambda$ with $\max\{\rho_\lambda(\xi,\xi'),\rho_\lambda(\eta,\eta')\}<\delta$ then 
			\begin{equation*} 
				\quad\ \big|\rho_\lambda(\xi,\eta)-\rho_\lambda(\xi',\eta')\big|<\epsilon.
			\end{equation*}
		\end{enumerate}
	\end{subdefinition}
	
	Any family of compact metric spaces is always equicontinuous by the triangle inequality. We shall later on see that a sequence of CAT$(-1)$ fillable antipodal spaces is equicontinuous (see \Cref{equicontinuity lemma} and \Cref{equicontinuity remark}).
	\medskip
	
	\subsection{Gromov-Hausdorff convergence}\hfill\\
	
	\noindent  In this subsection, we provide a brief overview of Gromov-Hausdorff (in short ``GH") distance and convergence concerning compact metric spaces. This was first introduced by Edwards \cite{edwards1975} later rediscovered by Gromov \cite{gromov1981structures},\cite{gromov1981group} (see \cite{tuzhilin2016invented} for an interesting discussion on the history of GH distance). Subsequently, we delve briefly into the case of pointed metric spaces that are not necessarily compact, which is of greater interest to us. For details, one can refer to \cite[Chapters 7, 8]{burago2022course}, also see \cite{bridson1999metric}, \cite{petrunin2023pure}, \cite{alexander2024alexandrov}.
	
	\medskip
	Given two metric spaces $X$ and $Y$, the {\it Gromov-Hausdorff distance} between them, denoted by $d_{GH}(X,Y)$ is defined to be the infimum of all numbers $r>0$ with the property that there exists a metric space $Z$ and isometric embeddings $f_1\colon X\to Z$ and $f_2\colon Y\to Z$ such that the Hausdorff distance $d_H$ between $f_1(X)$ and $f_2(Y)$ in $Z$, $d_H(f_1(X),f_2(Y))\le r$.
	\medskip
	
	\begin{subrmk}[\cite{burago2022course} Theorem 7.3.30]\label{dGH}
		The Gromov–Hausdorff distance $d_{GH}$ is non-negative, symmetric, and satisfies the triangle inequality; moreover
		$d_{GH}(X,Y)=0$ if and only if X and Y are isometric. Therefore, $d_{GH}$ defines a finite metric on the space of isometry classes of compact metric spaces.
	\end{subrmk}
	
	Given two metric spaces $X$ and $Y$, for $\epsilon>0$ a function $f\colon X\to Y$ (not necessarily continuous) is said to be an {\it $\epsilon$-isometry} if the {\it distortion} of $f$ is less than $\epsilon$, and $f(X)$ is an $\epsilon$-net of $Y$, i.e.
	\begin{equation}\label{epsiloniso}
		\Dis(f)\coloneqq \sup_{x_1,x_2\in X} \big|d_Y(f(x_1),f(x_2))-d_X(x_1,x_2)\big|<\epsilon \quad \text{and}\quad Y=\bigcup_{x\in X}B(f(x),\epsilon).
	\end{equation}
	
	\begin{subrmk}\label{e-iso}
		Given two metric spaces $X$ and $Y$, the following implications are valid (will use the above terminology) : 
		\begin{enumerate}[(i)]
			\item if $f\colon X\to Y$ is an $\epsilon$-isometry then for any $x_0\in X$ there exists $g\colon Y\to X$ a $3\epsilon$-isometry such that $g\circ f(x_0)=x_0$.
			\item if $X$, $Y$ are both compact and $d_{GH}(X,Y)<\epsilon$ then there exists a $2\epsilon$-isometry $f\colon X\to Y$
			\item if $X$, $Y$ are both compact and there exists an $\epsilon$-isometry $f\colon X\to Y$ then $d_{GH}(X,Y)<2\epsilon$.
		\end{enumerate}
		(see \cite[Chapter 7]{burago2022course}, cf. \cite{bridson1999metric}).
	\end{subrmk}
	\medskip
	
	\noindent We say, a sequence $\{X_n\}_{n\ge 1}$ of compact metric spaces converges to a compact metric space $X$ if $d_{GH} (X_n,X) \to0$
	as $n\to \infty$. We denote it by, $X_n\xrightarrow{d_{GH}}X$ and $X$ is called the Gromov–Hausdorff limit of $\{X_n\}$. The GH-limit, if it exists, is unique up to an isometry. Furthermore, from \Cref{e-iso}, we can infer that $X_n\xrightarrow{d_{GH}}X$ if and only if there exists $\epsilon_n$-isometries $f_n\colon  X_n\to X$ such that $\epsilon_n \to 0+$.
	
	Next, we recall the definition of the GH-Convergence for `pointed' metric spaces, which are proper and geodesic (not necessarily compact).   
	
	\begin{subdefinition}[{\bf Pointed Gromov-Hausdorff convergence}]\label{GHconv}
		A pointed metric space is a pair $(X, p)$ where $X$ is a metric space with $p$ a point in $X$.
		A sequence $\{(X_n,p_n )\}_{n\ge 1}$ of pointed proper, geodesic metric spaces converges in the
		Gromov–Hausdorff sense to another pointed proper geodesic metric space $(X, p)$ if the following
		holds. For every $R > 0$ and $\epsilon > 0$ there exists $n _0\ge1$ such that for
		every $n \ge n_0$ there is a  map $f \colon  B_{X_n}(p_n,R) \to B_X(p,R)\subset X$
		with$\colon$ \\
		
		(1) $f (p_n ) = p$;\\
		
		(2) $f$ is an $\epsilon$-isometry.\\
		
		\noindent Or equivalently, (from \Cref{e-iso}(i)), for every $R>0$ and $\epsilon>0$ there exists $m_0\ge 1$ such that for every $n\ge m_0$ there is an $\epsilon$-isometry $g\colon B(p,R) \to B(p_n,R)$ with $g(p)=p_n$.\\
		
		\noindent And we write $(X,p_n)\xrightarrow{GH\ conv.}(X,p)$, and $(X,p)$ is called a (pointed) GH-limit of the sequence. 
	\end{subdefinition}  
	
	Although a similar definition exists for general non-compact pointed GH-convergence (see \cite{burago2022course},\cite{bridson1999metric}), that is equivalent to the given definition for the proper geodesic case.
	
	\begin{rmk}\label{unique GH limit}
		Given $\{(X_n,p_n)\}_{n\ge 1}$ a sequence as in \Cref{GHconv}, if a pointed GH-limit exists, then it is unique up to a `pointed isometry'. For two GH-limits $(X,p)$ and $(X',p')$ there exists an isometry $f\colon (X,p)\to(X',p')$,  with $f(p)=p'$ (cf \cite[ Chapter 8]{burago2022course}). 
	\end{rmk}
	
	Next, we recall the Gromov pre-compactness criterion for pointed metric spaces (see \cite{burago2022course}, \cite{bridson1999metric}).
	
	\begin{subtheorem}{\bf (Gromov's Pre-compactness Theorem)}\label{GCT}
		Let $(X_n,x_n)$ be a sequence of pointed metric spaces. If for every $R>0$ and $\epsilon>0$, there exists an integer $N=N(R,\epsilon)$ such that the closed ball $B_{X_n}(x_n, R)\subset X_n$ admits an $\epsilon$-net with at most $N$ points, then a sub-sequence of $(X_n,x_n)$ converges in the pointed Gromov-Hausdorff sense to a complete metric space. 
	\end{subtheorem}
	
	In the following two paragraphs we shall briefly discuss the pointed GH-convergence of proper geodesically complete CAT$(-1)$ spaces.
	
	Note that for a sequence of proper geodesic metric spaces $(X_n,x_n)$, the sub-sequential limit $(X,x)$ obtained by Gromov's Pre-compactness Theorem ensures that it will also be proper and geodesic (see Exercises 8.1.8 and 8.1.9 in \cite{burago2022course}). In fact, if $\{(X_n,x_n)\}_{n\ge 1}$ is a sequence of proper geodesic CAT$(-1)$ spaces, then the limit is a proper geodesic CAT$(-1)$ space as well (see \cite{bridson1999metric}).
	
	In the terminology of \cite{lytchak2019geodesically}, a metric space $X$ is called {\it locally geodesically complete} if every local geodesic $\gamma\colon [a,b]\to X$ for any $a<b$, extends as a local geodesic to a larger interval $[a-\epsilon,b+\epsilon]$. For CAT$(-1)$ spaces, local geodesics coincide with geodesics. If $X$ is a complete, locally geodesically complete CAT$(-1)$ space, then it is {\it geodesically complete}, meaning any geodesic segment can be extended to a bi-infinite geodesic. In Example 4.2 of \cite{lytchak2019geodesically}, it is observed that if $\{(X_n,x_n)\}_{n\ge 1}$ is a sequence of locally geodesically complete pointed CAT$(-1)$ spaces, then so is their GH-limit. Thus, for proper, geodesically complete CAT$(-1)$ spaces, their GH-limit is also proper and geodesically complete.
	
	\bigskip
	
	\section{Gromov-Hausdorff-like distance between compact semi-metric spaces}
	
	\medskip
	We can similarly define the notion of $\epsilon$-isometry for maps between semi-metric spaces. We will sometimes refer to them as {\it rough isometries}.
	\begin{definition}
		Let $(Z_1,\rho_1)$ and $(Z_2,\rho_2)$ be two compact semi-metric spaces. For $\epsilon>0$ a function $f\colon (Z_1,\rho_1)\to (Z_2,\rho_2)$ is said to be an $\epsilon$-isometry, if distortion of $f$ is less than $\epsilon$ and $f(Z_1)$ is an $\epsilon$-net in $Z_2$ with respect to $\rho_2$, i.e.
		\begin{equation}
			\Dis(f)\coloneqq \sup_{\xi,\eta\in Z_1}\big|\rho_2(f(\xi),f(\eta))-\rho_1(\xi,\eta)\big|<\epsilon \quad \text{and} \quad Z_2=\bigcup_{\zeta\in Z_1}B_{\rho_2}(f(\zeta),\epsilon)
		\end{equation}
	\end{definition}
	
	Now we define the Gromov-Hausdorff like distance between compact semi-metric spaces and we call it the {\it Almost-Isometry} (in short {\it ``AI"}) distance.
	\begin{definition}
		Let $(Z_1,\rho_1)$ and $(Z_2,\rho_2)$ be two compact semi-metric spaces. We define the AI-distance between $(Z_1,\rho_1)$ and $(Z_2,\rho_2)$ to be $\colon$ 
		\begin{equation*}
			\begin{split}
				&d_{AI}((Z_1,\rho_1),(Z_2,\rho_2))
				\coloneqq \inf \{\ \epsilon>0\ |\ \exists\  \epsilon\text{-isometries }f\colon (Z_1,\rho_1)\to(Z_2,\rho_2)\text{, }g\colon (Z_2,\rho_2)\to(Z_1,\rho_1)\ \}
			\end{split}
		\end{equation*} 
	\end{definition}
	
	\noindent Note, this is similar to Edward's definition for the distance between compact metric spaces, \cite{edwards1975}, see discussion in \cite{tuzhilin2016invented}.  We have the following from definition$\colon$ 
	\begin{lemma}\label{AI=0}
		$d_{AI}((Z_1,\rho_1),(Z_2,\rho_2))=0 $ if and only if there exists an isometry $f\colon (Z_1,\rho_1)\to(Z_2,\rho_2)$.
	\end{lemma}
	\begin{proof}
		Suppose there exists an isometry $f\colon (Z_1,\rho_1)\to(Z_2,\rho_2)$ then easy to see that the AI-distance is $0$.
		
		Now suppose $d_{AI}((Z_1,\rho_1),(Z_2,\rho_2))=0$ then without loss of generality there exists $\epsilon_n$-isometries $f_n\colon (Z_1,\rho_1)\to(Z_2,\rho_2)$ and $g_n\colon (Z_2,\rho_2)\to (Z_1,\rho_1)$ with $\epsilon_n \to 0+$. $Z_1$ is a compact metric space, so we have a countable dense subset $D$. Passing to a sub-sequence of ${f_n}$, we have ${f_n(x)}$ converges for each $x\in D$. Hence define $f\colon D\to Z_2$ by $f(x)=\lim_{n\to\infty}f_n(x)$. Also $f$ is an isometric embedding from $D$ in to $(Z_2,\rho_2)$, i.e. for $x,x'\in D$ we have $\rho_2(f(x),f(x'))=\rho_1(x,x').$ By using continuity (\Cref{Coro 2} and \ref{Coro 3}) and positivity of separating functions we can uniquely extend this $f$ to an isometric embedding of $(Z_1,\rho_1)$ into $(Z_2,\rho_2)$, i.e. $\rho_2(f(x),f(x'))=\rho_1(x,x')$ for all $x,x'\in Z_1$. Using $g_n$'s we can similarly define an isometric embedding from $g\colon (Z_2,\rho_2)\to(Z_1,\rho_1)$. 
		
		Consider $h\coloneqq f\circ g\colon (Z_2,\rho_2)\to (Z_2,\rho_2)$ which is an isometric embedding. We show that $h$ is surjective and hence $f\colon (Z_1,\rho_1)\to (Z_2,\rho_2)$ is an isometry. If $h$ is not surjective, then there exists $y\in Z_2\setminus h(Z_2)$ exists. Now $h$ being continuous, $h(Z_2)$ is compact therefore $\rho_2(y,h(Z_2))=c>0$. Observe that the set $\{h^n(y)\ |\ n\in \mathbb N\}$ is an infinite set and is $c$-separated with respect to $\rho_2$, i.e. any for any $m,n\in \mathbb N$ (say $m>n$) then $\rho_2(h^n(y),h^m(y))=\rho_2(y,h^{m-n}(y))\ge c$. However, this is impossible since $Z_2$ is a compact metrizable space. Hence, $h$ must be surjective. Therefore, $f$ is a surjection and an isometry.	
	\end{proof}
	
	\medskip
	
	We say a sequence of compact semi-metric spaces $\{(Z_n,\rho_n)\}_{n\ge 1}$ {\it AI-converges} to another compact semi-metric space $(Z,\rho_0)$ if 
	\begin{equation*}
		d_{AI}((Z_n,\rho_n),(Z,\rho_0))\to 0 \quad \hbox{as} \quad n\to\infty
	\end{equation*}
	and we denote it by $(Z_n,\rho_n)\xrightarrow{AI\ conv.}(Z,\rho_0)$, and $(Z,\rho_0)$ is called an {\it AI-limit} of the sequence. In order to have AI-convergence, one require $\epsilon_n$-isometries $f_n \colon (Z_n,\rho_n) \to (Z,\rho_0)$ and $g_n \colon (Z,\rho_0) \to (Z_n,\rho_n)$ with $\epsilon_n\to 0+$. However, the following \Cref{inverse} shows that, it is enough to have $\epsilon_n$-isometries $f_n \colon (Z_n,\rho_n) \to (Z,\rho_0)$, with $\epsilon_n\to 0+$, for AI-convergence. 
	
	\begin{lemma}\label{inverse}
		Let $\{(Z_n,\rho_n)\}_{n\ge 1}$ be a sequence of compact semi-metric spaces and let $(Z,\rho_0)$ be another compact semi-metric space such that we have $\epsilon_n$-isometries $f_n\colon (Z_n,\rho_n)\to(Z,\rho_0)$ with $\epsilon_n\to 0+$. Then $(Z_n,\rho_n)\xrightarrow{AI\ conv.}(Z,\rho_0)$.
	\end{lemma}
	\begin{proof}
		We will show that, for any fixed $\delta>0$ there exist maps $g_n\colon (Z,\rho_0)\to(Z_n,\rho_n)$ which are $\delta$-isometries for all $n$ sufficiently large. Therefore, for all $n$ sufficiently large we will have $d_{AI}((Z_n,\rho_n),(Z,\rho_0))<\delta$.
		
		We know by continuity property of $\rho_0$ (from \Cref{Coro 2}) we can choose $\delta_1,\delta_2>0$ with $\delta=\delta_0>\delta_1>\delta_2$ such that, if $\xi,\xi',\eta,\eta'\in Z$ with $\max\{\rho_0(\xi,\xi'),\rho_0(\eta,\eta')\}<\delta_i$ then
		\begin{equation}\label{delta i}
			|\rho_0(\xi,\eta)-\rho_0(\xi',\eta')|<\frac{\delta_{i-1}}{2},
		\end{equation} (for $i=1,2$, since we shall be using this twice)
		
		Now for every $n$ we consider $F_n$ a finite $(2\epsilon_n)$-net of $Z_n$ with respect to $\rho_n$ such that $\rho_n(\xi,\eta)\ge 2\epsilon_n$ for all distinct $\xi,\eta\in F_n$ (this is always possible because $Z_n$'s are compact). Observe that $f_n$ maps $F_n$ injectively onto $G_n\coloneqq f_n(F_n)\subset Z$ for all $n$, because		
		\begin{equation*}
			\rho_0(f_n(\xi),f_n(\eta))\ge\rho_n(\xi,\eta)-\epsilon_n>2\epsilon_n-\epsilon_n>0
		\end{equation*}
		for distinct $\xi ,\eta\in F_n$.
		
		Let $\xi\in Z$, for every $n$ there exists  $\xi_n\in Z_n$ such that $\rho_0(\xi,f_n(\xi_n))<\epsilon_n$ (as $f_n$'s are $\epsilon_n$-isometries), moreover there exists $\tilde{\eta_n}\in F_n$ such that $\rho_n(\xi_n,\tilde{\eta_n})<2\epsilon_n$. By hypothesis $\epsilon_n \to 0+$, so for all $n$ sufficiently large $\rho_0(\xi,f_n(\tilde{\xi_n}))<\epsilon_n<\delta_2$ implies
		\begin{equation*}
			\begin{split}
				|\rho_0(\xi,f_n(\tilde{\eta_n}))&-\rho_0(f_n(\xi_n),f_n(\tilde{\eta_n}))|<\frac{\delta_1}{2}\\
				\implies \rho_0(\xi,f_n(\tilde{\eta_n}))&<\rho_0(f_n(\xi_n),f_n(\tilde{\eta_n}))+\frac{\delta_1}{2}\\
				&\le \rho_n(\xi_n,\tilde{\eta_n}) +\epsilon_n+\frac{\delta_1}{2}\\
				&\le 3\epsilon_n+\frac{\delta_1}{2}<\delta_1.
			\end{split}
		\end{equation*}
		Therefore, for all $n$ sufficiently large $G_n$ is a $\delta_1$-net of $Z$ with respect to $\rho_0$, (hence a $\delta$-net as well). Since $f_n$ maps $F_n$ injectively to $G_n$, define $g_n\colon G_n\to F_n$ such that 
		\begin{equation}\label{injectivity relation}
			(g_n\circ f_n)|_{F_n}= \text{id}_{F_n} \quad \text{and} \quad (f_n\circ g_n)|_{G_n}=\text{id}_{G_n} 
		\end{equation}
		For each $\xi\in Z$ define the set 
		\begin{equation*}
			S_n(\xi)\coloneqq \{\eta\in F_n\ |\ \rho_0(\xi,f_n(\eta))<\delta_1\}\subset F_n.
		\end{equation*}
		Note for all $n$ sufficiently large $S_n(\xi)\not=\emptyset$, since $G_n=f_n(Z_n)$ is a $\delta_1$-net of $Z$. So, axiom of choice lets us extend functions $g_n\colon (Z,\rho_0)\to(Z_n,\rho_n)$ for all $n$ sufficiently large such that $g_n(\xi)\in S_n(\xi)$, for each $\xi\in Z\setminus G_n$. Clearly $g_n(Z)=F_n$, which is a $(2\epsilon_n)$-net with respect to $\rho_n$. By the above construction, for $\xi_1,\xi_2\in Z$ and for all $n$ sufficiently large, we have
		\begin{equation*}
			\max\{\rho_0(\xi_1,f_n(g_n(\xi_1))),\rho_0(\xi_2,f_n(g_n(\xi_2)))\}<\delta_1.
		\end{equation*}
		Thus, by the choice of $\delta_1$ above \eqref{delta i} and $f_n$ being $\epsilon_n$-isometry, we get
		\begin{equation*}
			\begin{split}
				\big|\rho_n(g_n(\xi_1),g_n(\xi_2))-\rho_0(\xi_1,\xi_2)\big|
				\le\ &\big|\rho_n(g_n(\xi_1),g_n(\xi_2))-\rho_0(f_n(g_n(\xi_1)),f_n(g_n(\xi_2)))\big|\\
				+&\big|\rho_0(f_n(g_n(\xi_1)),f_n(g_n(\xi_2)))-\rho_0(\xi_1,\xi_2)\big|\\
				\le& \  \epsilon_n+\frac{\delta}{2}\  <\  \delta.
			\end{split}
		\end{equation*}
		Therefore, for all $n$ large $g_n$ as defined above is a $\delta$-isometry.
	\end{proof}
	
	\medskip
	
	The technique of constructing rough $\delta$-isometries $g_n$, as described above, is standard. 
	This technique will be used multiple times.
	
	\begin{rmk}\label{AIconv}
		It is worth noting from \Cref{inverse} we have $(Z_n,\rho_n)\xrightarrow{AI\ conv.}(Z,\rho_0)$ if and only if there exist $\epsilon_n$-isometries $f_n\colon (Z_n,\rho_n)\to(Z,\rho_0)$ with $\epsilon_n \to 0+$. Moreover if $\rho_n$'s and $\rho_0$ are metrics then $(Z_n,\rho_n)\xrightarrow{AI\ conv.}(Z,\rho_0)$ is equivalent to $(Z_n,\rho_n)\xrightarrow{d_{GH}}(Z,\rho_0)$.
	\end{rmk}
	
	Similar to what we observed in \Cref{inverse}, the next \Cref{conditional inverse} demonstrates that for AI-convergence, it suffices to have $\epsilon_n$-isometries $g_n\colon (Z, \rho_0) \to (Z_n, \rho_n)$ with $\epsilon_n \to 0^+$, provided that $\{(Z_n, \rho_n)\}_{n \ge 1}$ is an equicontinuous family of semi-metrics.
	\begin{lemma}\label{conditional inverse}
		Let $\{(Z_n,\rho_n)\}_{n\ge 1}$ be a sequence of equicontinuous compact semi-metric spaces and let $(Z,\rho_0)$ be another compact semi-metric space. If there exist $\delta_n$-isometries $g_n\colon (Z,\rho_0)\to(Z_n,\rho_n)$ such that $\delta_n\to 0+$, then $(Z_n,\rho_n)\xrightarrow{AI\ conv.}(Z,\rho_0)$.
	\end{lemma}
	\begin{proof}
		It is enough to show, given any $\epsilon>0$ there exist $\epsilon$-isometries $f_n\colon (Z_n,\rho_n)\to (Z,\rho_0)$, for all $n$ sufficiently large. Hence, we will get $d_{AI}((Z_n,\rho_n),(Z,\rho_0))<\epsilon$ for all $n$ sufficiently large.
		
		We know $\{(Z_n,\rho_n)\}_{n\ge 1}$ being equicontinuous can choose $\epsilon_1,\epsilon_2>0$ with $\delta=\epsilon_0>\epsilon_1>\epsilon_2$ such that, (for $i=1,2$) for all $n$, if $\xi,\xi',\eta,\eta'\in Z_n$ with $\max\{\rho_0(\xi,\xi'),\rho_0(\eta,\eta')\}<\epsilon_i$ then
		\begin{equation}\label{epsilon j}
			\big|\rho_n(\xi,\eta)-\rho_n(\xi',\eta')\big|<\frac{\epsilon_{i-1}}{2},
		\end{equation} (for $i=1,2$, since we shall be using this twice).
		
		Fix $\epsilon_3<\frac{\epsilon_2}{2}$. Now consider $G$ a finite $\epsilon_3$-net of $Z$ with respect to $\rho_0$ such that $\rho_0(\xi,\eta)\ge\epsilon_3$ for all distinct $\xi,\eta\in G$, this is possible by the compactness of $Z$. Observe that $g_n$ maps $G$ injectively onto $F_n\coloneqq g_n(G)\subset Z_n$ for all $n$ sufficiently large, because
		\begin{equation*}
			\rho_n(g_n(\xi),g_n(\eta))\ge\rho_0(\xi,\eta)-\delta_n>\epsilon_3-\delta_n>0
		\end{equation*}
		for distinct $\xi ,\eta\in G$, since $\delta_n\to 0+$.
		
		Let $\xi_n\in Z_n$, there exists  $\xi\in Z$ such that $\rho_n(\xi_n,g_n(\xi))<\delta_n$ (as $g_n$'s are $\delta_n$-isometries), moreover there exists $\tilde{\eta}\in G$ such that $\rho_0(\xi,\tilde{\eta})<\epsilon_3$. By hypothesis $\delta_n \to 0+$, so for all $n$ large $\rho_n(\xi_n,g_n(\xi))<\delta_n<\epsilon_2$ implies 
		\begin{equation*}
			\begin{split}
				\big|\rho_n(\xi_n,g_n(\tilde{\eta}))&-\rho_n(g_n(\xi),g_n(\tilde{\eta}))\big|<\frac{\epsilon_1}{2}\\
				\implies \rho_n(\xi_n,g_n(\tilde{\eta}))&<\rho_n(g_n(\xi),g_n(\tilde{\eta}))+\frac{\epsilon_1}{2}\\
				&\le \rho_0(\xi,\tilde{\eta}) +\delta_n+\frac{\epsilon_1}{2}\\
				&\le \epsilon_3+\delta_n+\frac{\epsilon_1}{2}<\epsilon_1.
			\end{split}
		\end{equation*}
		Therefore, for all $n$ sufficiently large $F_n$ is a $\epsilon_1$-net of $Z_n$ with respect to $\rho_n$ (hence a $\epsilon$-net as well). Define map $f_n\colon F_n\to G$ for all $n$ sufficiently large, such that
		\begin{equation*}
			f_n\circ g_n|_{G}= \text{id}_{G} \quad \text{and} \quad g_n\circ f_n|_{F_n}=\text{id}_{F_n}.
		\end{equation*}
		For each $\xi\in Z_n$ for $n$ sufficiently large define the set 
		\begin{equation*}
			S^n(\xi)\coloneqq \{\eta\in G\ |\ \rho_n(\xi,g_n(\eta))<\epsilon_1\}\subset G.
		\end{equation*}
		Here also note, $S^n(\xi)\not=\emptyset$ since for all $n$ sufficiently large as $G_n=f_n(Z_n)$ is a $\epsilon_1$-net of $Z_n$. Then axiom of choice lets us extend the functions $f_n\colon (Z_n,\rho_n)\to(Z,\rho_0)$ for all $n$ sufficiently large such that $f_n(\xi)\in S^n(\xi)$ for $\xi\in Z_n\setminus F_n$. Then we have $f_n(Z_n)=G$, which is an $\epsilon_3$-net with respect to $\rho_0$. 
		
		Next we use the fact that distortion of $f\colon F_n\to G$ is less or equal to $\delta_n$ and refer to the equicontinuity of $\rho_n$'s as given in \eqref{epsilon j} to conclude$\colon$  for all $n$ sufficiently large we have $f_n$, as defined above, is an $\epsilon$-isometry.
	\end{proof}
	
	Also, as a consequence of the above lemmas, we have the uniqueness of AI-limit up to isometry for AI-convergence of sequences of compact semi-metric spaces.
	
	\medskip
	
	\begin{prop}
		Let $\{(Z_n,\rho_n)\}_{n\ge 1}$ be a sequence compact semi-metric spaces and let $(Z,\rho_0)$ be another compact semi-metric space. If $(Z_n,\rho_n)\xrightarrow{AI\ conv.}(Z,\rho_0)$ then this limit is unique up to isometry. 
	\end{prop}
	\begin{proof}
		Let $(Z,\rho_0)$ and $(Z',\rho'_0)$ be two AI-limits of the sequence of compact semi-metric spaces $\{(Z_n,\rho_n)\}_{n\ge 1}$. We will show that for every $\delta>0$, we have $d_{AI}((Z',\rho'_0),(Z,\rho_0))<2\delta$. Then $d_{AI}((Z',\rho'_0),(Z,\rho_0))=0$ and hence there exists an isometry from $(Z',\rho'_0)$ to $(Z,\rho_0)$ by \Cref{AI=0}.  
		
		Fix $\delta>0$. By continuity of $\rho_0$ (from \Cref{Coro 1}) choose $\delta_1$ with $0<\delta_1<\delta$ such that if $\xi,\xi',\eta,\eta'\in Z$ with $\max\{\rho_0(\xi,\xi'),\rho_0(\eta,\eta')\}<\delta_1$ then 
		\begin{equation*}
			\big|\rho_0(\xi,\eta)-\rho_0(\xi',\eta')\big|<\delta.
		\end{equation*} 
		
		From the given hypothesis, without loss of generality there exists $\epsilon_n$-isometries $f_n\colon (Z_n,\rho_n)\to(Z,\rho_0)$ and $f'_n\colon (Z_n,\rho_n)\to(Z',\rho'_0)$ with $\epsilon_n \to 0+$. As in the proof of \Cref{inverse}, for every $n$ let $F_n$ be a finite $(2\epsilon_n)$-net of $Z_n$ with respect to $\rho_n$ such that $\rho_n(\xi,\eta)\ge 2\epsilon_n$ for all distinct $\xi,\eta\in F_n$. Observe that both $f_n$ and $f'_n$ maps $F_n$ injectively onto $G_n\coloneqq f_n(F_n)\subset Z$ and $G'_n\coloneqq f'_n(F_n)\subset Z'$ respectively for all $n$. Therefore, for all $n$, $\#G_n=\#G'_n$.
		
		Similarly, given $\delta_1>0$, for all $n$ sufficiently large $G_n$ and $G'_n$ are $\delta_1$-net of $Z$ and $Z'$ respectively. Consider $g_n\colon G_n\to F_n$ and $g'_n\colon G'_n\to F_n$ such that 
		\begin{equation*}
			(f_n\circ g_n)|_{G_n}=\text{id}_{G_n} \quad \hbox{and}\quad (f'_n\circ g'_n)|_{G'_n}=\text{id}_{G'_n}
		\end{equation*}
		Then $\varphi_n\coloneqq f'_n\circ g_n\colon G_n\to G'_n$ is a bijection from $G_n$ to $G'_n$ with inverse $\psi_n\coloneqq (\varphi_n)^{-1} =f_n\circ g'_n$. Observe that, the distortion of both $\varphi_n\colon G_n\to G'_n$ and $\psi_n\colon G'_n\to G_n$ is a less or equal to $2\epsilon_n$. 
		
		For any $\xi\in Z$, for all $n$ sufficiently large the set  
		\begin{equation*}
			P_n(\xi)\coloneqq \{\ \zeta\in G'_n\ |\  \rho_0(\xi,\psi_n(\zeta))<\delta_1\ \}\subset G'_n
		\end{equation*}
		is non-empty, since $G_n$ is a $\delta_1$-net. By axiom of choice we extend the map $\varphi_n$ to $\tilde{\varphi}_n\colon (Z,\rho_0)\to (Z',\rho'_0)$, such that for $\xi\in Z\setminus G_n$ we define $\tilde{\varphi}_n(\xi)\in P_n(\xi)$. Hence $\varphi_n(Z)=G'_n$ is a $\delta_1$-net of $Z'$. From the choice of $\delta_1$, we can argue that the extended map $\tilde{\varphi}_n$ has distortion $2\delta$. Therefore, $\tilde{\varphi}_n$ is $2\delta$-isometry for all $n$ sufficiently large.
		
		Similarly, for all $n$ sufficiently large, we can extend $\psi_n$ to a $2\delta$-isometry $\tilde{\psi}_n: (Z',\rho'_0) \to (Z,\rho_0)$.
		
		\medskip
		
		Thus, we get $d_{AI}((Z',\rho'_0),(Z,\rho_0))<2\delta$.
	\end{proof}
	
	\medskip
	
	\begin{rmk}
		From the above we have observed that for any $(Z_1,\rho_1)$ and $(Z_2,\rho_2)$ compact semi metric spaces 
		\begin{enumerate}[(a)]
			\item $d_{AI}((Z_1,\rho_1),(Z_2,\rho_2))<\infty$, i.e. $d_{AI}$ assigns finite value,
			\item $d_{AI}((Z_1,\rho_1),(Z_2,\rho_2))=d_{AI}((Z_2,\rho_2),(Z_1,\rho_1))$, i.e. $d_{AI}$ is symmetric,
			\item $d_{AI}((Z_1,\rho_1),(Z_2,\rho_2))=0$ if and only if $(Z_1,\rho_1)$ and $(Z_2,\rho_2)$ are isometric,
		\end{enumerate} Therefore, $d_{AI}$ defines a separating function on $\mathscr S$, the isometry classes of compact semi-metric spaces. Finite semi-metric spaces are dense in $(\mathscr S,d_{AI})$. One can also show that antipodal spaces form a closed subset of $(\mathscr S,d_{AI})$, in the sense that AI-limits of antipodal spaces are antipodal spaces. 
	\end{rmk}
	One might expect $d_{AI}$ to share similar properties with the Gromov-Hausdorff (GH) distance. While a comprehensive comparison is beyond the scope of this article, we restrict our discussion to the properties and aspects of $d_{AI}$ that are pertinent to our current objectives.

	\bigskip
	
	\section{Proof of \Cref{forward}}
	
	Let $\{(Z_n,\rho_n)\}_{n\ge 1}$ be a sequence of antipodal spaces and let $(Z,\rho_0)$ be another antipodal space such that $(Z_n,\rho_n)\xrightarrow{AI\ conv.}(Z,\rho_0)$, i.e. for each $n$ there exists $\epsilon_n$-isometry $f_n\colon (Z_n,\rho_n)\to (Z,\rho_0)$ with $\epsilon_n \to 0+$. From this hypothesis we will prove that $$(\MM(Z_n),\rho_n)\xrightarrow{GH conv.}(\MM(Z),\rho_0).$$

	In order to prove this, for any fixed $R>0$  for every $n$, we construct  $\FF_n\colon {B}_{\MM(Z)}(\rho_0,R)\to {B}_{\MM(Z_n)}(\rho_n,R)$ with $\FF_n(\rho_0)=\rho_n$, such that given any $\epsilon>0$ for all $n$ sufficiently large these $\FF_n$'s become $\epsilon$-isometry. 
	
	\medskip
	
	
	
	\noindent Now, we shall state and prove a few facts which will be crucial in the proof (of \Cref{forward}).
	
	\begin{lemma}\label{tauf_n}
		Let $\alpha,\beta\in B_{\MM(Z)}(\rho_0,R) \subset \MM(Z)$, then 
		\begin{equation}
			\big\|\tau_\alpha\circ f_n-\tau_\beta\circ f_n\big\|_\infty\to \big\|\tau_\alpha-\tau_\beta\big\|_\infty=d_{\MM(Z)}(\alpha,\beta)
		\end{equation}
		as $n\to \infty$, uniformly in $\alpha, \beta\in B_{\MM(Z)}(\rho_0,R)$.
	\end{lemma}
	\begin{proof}
		It is straightforward to see $\|\tau_\alpha\circ f_n-\tau_\beta\circ f_n\|_\infty\le\|\tau_\alpha-\tau_\beta\|_\infty$.
		
		Given $\epsilon>0$ we know by \Cref{Coro 1} there exists $\delta\coloneqq \delta(\epsilon,R)>0$ such that, if $\xi,\eta\in Z$ with $\rho_0(\xi,\eta)<\delta$ then 
		\begin{equation*}
			|\tau_\rho(\xi)-\tau_\rho(\eta)|<\epsilon
		\end{equation*} for all $\rho\in B_{\MM(Z)}(\rho_0,R)$. Since $\epsilon_n \to 0+$ we have $\epsilon_n<\delta$ for all $n$ sufficiently large. Suppose $\xi\in Z$ then there exists $\zeta_n\in Z_n$ such that $\rho_0(\xi,f_n(\zeta_n))<\epsilon_n<\delta$ for each $n$ sufficiently large. The maps $f_n$'s are $\epsilon_n$-isometry, which implies
		\begin{equation*}
			\big|(\tau_\alpha(\xi)-\tau_\beta(\xi))-(\tau_\alpha\circ f_n(\zeta_n)-\tau_\beta\circ f_n(\zeta_n))\big|\le\big|\tau_\alpha(\xi)-\tau_\alpha\circ f_n(\zeta_n)\big|+\big|\tau_\beta(\xi)-\tau_\beta\circ f_n(\zeta_n)\big|<2\epsilon.
		\end{equation*}
		Therefore, for any $\xi\in Z$
		\begin{equation*}
			\big|(\tau_\alpha(\xi)-\tau_\beta(\xi))\big|<\big|\tau_\alpha(f_n(\zeta_n))-\tau_\beta(f_n(\zeta_n))\big|+2\epsilon\\
			\le \big\|\tau_\alpha\circ f_n-\tau_\beta\circ f_n\big\|_\infty+2\epsilon
		\end{equation*}
		Hence for all $n$ sufficiently large we have $\|\tau_\alpha-\tau_\beta\|\le \|\tau_\alpha\circ f_n-\tau_\beta\circ f_n\|_\infty+2\epsilon$. 
	\end{proof}
	
	\medskip
	
	\begin{lemma}\label{tau composed f_n}
		Fix $R>0$ and $\epsilon>0$. Then there exists $\delta=\delta(\epsilon,R)>0$ such that for all $n$ sufficiently large (depending on $R$ and $\epsilon$), whenever $\xi,\eta\in Z_n$ with $\rho_n(\xi,\eta)<\delta$ we have
		\begin{equation*}
			\big|\tau_\rho\circ f_n(\xi)-\tau_\rho\circ f_n(\eta)\big|<\epsilon
		\end{equation*}
		for all $ \rho\in {B}_{\MM(Z)}(\rho_0,R) $.
	\end{lemma}
	\begin{proof}
		Given $\epsilon>0$, by \Cref{Coro 1} there exists $\delta\coloneqq \delta(\epsilon, R)>0$ such that  if $\xi,\eta\in Z$ and $\rho_0(\xi,\eta)<2\delta$ then 
		\begin{equation*}
			\big|\tau_\rho(\xi)-\tau_\rho(\eta)\big|<\epsilon
		\end{equation*}
		for all $ \rho\in {B}_{\MM(Z)}(\rho_0,R)$.
		As $\epsilon_n \to 0+$, for all $n$ sufficiently large $\epsilon_n<\delta$. So for all $n$ sufficiently large if  $\xi,\eta\in Z_n$ with $\rho_n(\xi,\eta)<\delta$ then $f_n$ being $\epsilon_n$-isometry we have $\rho_0(f_n(\xi),f_n(\eta))<\rho_n(\xi,\eta)+\epsilon_n<2\delta$.
		Therefore, for all $n$ by \Cref{Coro 1} we have, if $\xi,\eta\in Z_n$ and $\rho_n(\xi,\eta)<\delta$ then 
		\begin{equation*}
			\big|\tau_\rho\circ f_n(\xi)-\tau_\rho\circ f_n(\eta)\big|<\epsilon
		\end{equation*}
		for all $ \rho\in {B}_{\MM(Z)}(\rho_0,R) $. 
	\end{proof}
	
	\medskip
	
	Now we have the following Lemma as a consequence of the above \Cref{tau composed f_n} and the \Cref{smoothing of tau lemma} proved in \Cref{PoU}. Given a fixed $\delta>0$, for each $n$, by compactness of $(Z_n,\rho_n)$ we can fix a finite $\delta$-net $F_n\subset Z_n$  with respect to $\rho_n$. Then consider $\{P_\zeta^n\}_{\zeta\in F_n}$, the partition of unity on $(Z_n,\rho_n)$ associated to the $\delta$-net $F_n$ and the corresponding smoothing operator $C_n\colon (B(Z_n),\|\cdot\|_\infty)\to (C(Z_n),\|\cdot\|_\infty)$ as defined in \Cref{PoU}.
	
	\begin{lemma}
		\label{smoothing of taufn}
		Fix $R>0$, given $\epsilon>0$ suitably choose $\delta=\delta(R,\epsilon)>0$ such that it satisfies the claim of \Cref{tau composed f_n} for all $n$ sufficiently large. For each $n$ and for all $\rho \in {B}_{\MM(Z)}(\rho_0,R)$ consider the bounded function $\tau_\rho\circ f_n\colon Z_n\to \mathbb{R}$. 
		Then for all $n$ sufficiently large, the continuous function $C_n(\tau_\rho\circ f_n)$ satisfies, 
		\begin{equation*}
			\big\|C_n(\tau_\rho\circ f_n)-\tau_\rho\circ f_n\big\|_\infty<\epsilon \quad \hbox{ with } \quad \big\|C_n(\tau_\rho\circ f_n)\big\|_\infty\le\big\|\tau_\rho\circ f_n\big\|_\infty,
		\end{equation*}  for $\rho\in B_{\MM(Z)}(\rho_0,R)$.
	\end{lemma}
	\begin{proof}
		A direct consequence of \Cref{tau composed f_n} and \Cref{smoothing of tau lemma}. 
	\end{proof}
	
	\medskip
	
	We already know that for $\rho\in B_{\MM(Z)}(\rho_0,R)\in\MM(Z)$, the function $\tau_\rho\circ f_n\colon  Z_n\to \R$ is bounded  for every $n$ and $\|\tau_\rho\circ f_n\|_\infty\le \|\tau_\rho\|_\infty=d_{\MM(Z)}(\rho_0,\rho)\le R$,
	moreover from inequality \eqref{discrepancy 2 lipschitz ineq}, we get
	\begin{equation}\label{discrepancy less than 2R}
		\big\|D_{\rho_n}(\tau_\rho\circ f_n)\big\|_\infty\le 2\big\|\tau_\rho\big\|_\infty\le 2R.
	\end{equation}
	\noindent In fact we show that, for fixed $R>0$ we have the discrepancy $D_{\rho_n}(\tau_\rho\circ f_n)$ uniformly close to $0$ for all $n$ sufficiently large, which will be crucial in concluding \Cref{distortion}.

	\begin{lemma}
		\label{discrepancy tau composed f_n}
		Fix $R>0$, given $\epsilon>0$, there exits $N_0\in \mathbb{N}$ such that for all $n>N_0$ we have,
		\begin{equation}
			\sup_{\rho\ \in\  B_{\MM(Z)}(\rho_0,R)} \big\|D_{\rho_n}(\tau_\rho\circ f_n)\big\|_\infty\ <\ \epsilon.
		\end{equation} 
	\end{lemma}
	\begin{proof}
		Note that the singleton ${(Z,\rho_0)}$ is an equicontinuous family. Since $\epsilon_n\to0+$, given $\delta>0$, for all $n$ sufficiently large $\epsilon_n<\delta$, and hence for all $n$ sufficiently large $f_n$'s are $\delta$-isometries.  So, the claimed result follows from \Cref{discrepancy main}.
	\end{proof}
	
	\medskip
	
	Before proceeding further, we would like to define the following retraction map.
	
	\begin{definition}\label{ret}
		For any antipodal space $(Z,\rho_0)$ we define the \textbf{retraction} on $\MM(Z)$ onto a metric ball ${B}_{\MM(Z)}(\rho_0,R)$, $\pi^R\colon \MM(Z)\to {B}_{\MM(Z)}(\rho,R)$ by,
		\begin{equation*}
			\pi^R(\rho)\coloneqq 
			\begin{cases}
				\PP_\infty\left(\ \dfrac{R}{\|\tau_\rho\|_\infty}\cdot\tau_\rho\ ,\ \rho_0\ \right) \quad &\text{if } \, d_{\MM(Z)}(\rho,\rho_0)=\|\tau_\rho\|_\infty>R\\
				\rho \quad &\text{if } \, d_{\MM(Z)}(\rho,\rho_0)=\|\tau_\rho\|_\infty\le R\\
			\end{cases}
		\end{equation*}
		Here, $\tau_\rho=\log\frac{d\rho}{d\rho_0}$.
	\end{definition}
	
	The following lemma is immediate from the definition.
	
	\begin{lemma}
		\label{retraction lemma}
		For all $\rho\in \MM(Z)$, we have $d_{\MM(Z)}(\pi^R(\rho),\rho_0)\le R$.
		Moreover, if $d_{\MM(Z)}(\rho,\rho_0)>R$ then 
		\begin{equation*}
			R=d_{\MM(Z)}(\pi^R(\rho),\rho_0)=d_{\MM(Z)}(\rho,\rho_0)- d_{\MM(Z)}(\pi^R(\rho),\rho). 
		\end{equation*}
	\end{lemma}
	\begin{proof}
		If  $d_{\MM(Z)}(\rho,\rho_0)=\|\tau_\rho\|_\infty\le R$ then by definition $d_{\MM(Z)}(\pi^R(\rho),\rho)=0$.\newline
		Observe that for $d_{\MM(Z)}(\rho,\rho_0)=\|\tau_\rho\|_\infty >R$ we have by \Cref{convexcombi} 
		\begin{equation*}
			d_{\MM(Z)}(\pi^R(\rho),\rho_0)=d_{\MM(Z)}\left(\PP_\infty\left(\dfrac{R}{\|\tau_\rho\|_\infty}\cdot\tau_\rho\ ,\ \rho_0\right),\rho_0\right) =\dfrac{R}{\|\tau_\rho\|_\infty}d_{\MM(Z)}(\rho,\rho_0)=R.
		\end{equation*}
		We know that $d_{\MM(Z)}(\rho,\pi^R(\rho))+d_{\MM(Z)}(\pi^R(\rho),\rho_0)=d_{\MM(Z)}(\rho,\rho_0)$.
		Hence, we have the result.
	\end{proof}
	
	\medskip
	
	\begin{proof}[{\bf Proof of \Cref{forward}}]
		Fix $R>0$. We give the construction of the map
		\begin{equation*}
			\FF_n\colon {B}_{\MM(Z)}(\rho_0,R)\to {B}_{\MM(Z_n)}(\rho_n,R),
		\end{equation*} and show it has the desired properties (as discussed in the beginning of the section). 
		
		Let $\rho\in {B}_{\MM(Z)}(\rho_0,R)$ we want to define $\FF_n(\rho)\in {B}_{\MM(Z_n)}(\rho_n,R)$, which is induced by some continuous function $\tilde{\tau_\rho}\in C(Z_n)$ with discrepancy $D_{\rho_n}(\tilde{\tau_{\rho}})\equiv0$ (see  \Cref{discrepancy}). Our initial guess would be to take pull back, $f_n^*(\tau_\rho)=\tau_\rho\circ f_n$, but this need not always be continuous. So we make it continuous by composing with the smoothing operator $C_n$ (as encountered in \Cref{smoothing of taufn}) and get $C_n(\tau_\rho\circ f_n)\in C(Z_n)$. But, this might not have discrepancy $D_{\rho_n}(C_n(\tau_\rho\circ f_n))$ identically zero. Therefore, we compose it with the antipodalization map to get $\PP_\infty(C_n(\tau_\rho\circ f_n),\rho_n)$ (an antipodal function).   Finally compose it with the retraction $\pi^R_n\ \colon \ \MM(Z_n)\to {B}_{\MM(Z_n)}(\rho_n, R)$) as defined in \Cref{ret} (to put it inside $B_{\MM(Z_n)}(\rho_n, R)$) and define it to be $\FF_n(\rho)$. Here is a sketch:  
		\medskip
		
		\begin{equation}\label{FF}
			\begin{tikzcd}[column sep=1.25em,row sep=1em]
				{B_{\MM(Z)}(\rho_0,R)} & {C(Z)} & {C(Z_n)} && {\MM(Z_n)} & {B_{\MM(Z_n)}(\rho_n,R)} \\
				\rho & {\tau_{\rho}} & {C_n(\tau_{\rho}\circ f_n)} && {\mathcal{P}_\infty(C_n(\tau_\rho\circ f_n),\rho_n)} & {\pi^R_n(A_n(\rho)))}
				\arrow["{{{i_{\rho_0}}}}", from=1-1, to=1-2]
				\arrow["{{{\mathcal {F}_n}}}"{description, pos=0.5}, curve={height=-50pt}, from=1-1, to=1-6]
				\arrow["{{{C_n\circ f_n^*}}}", from=1-2, to=1-3]
				\arrow["{{{\mathcal{P}_\infty(\ \cdot\ ,\rho_n)}}}", from=1-3, to=1-5]
				\arrow["{{{\pi^R_n}}}", from=1-5, to=1-6]
				\arrow[maps to, from=2-1, to=2-2]
				\arrow["{{{A_n}}}"{description}, curve={height=40pt}, dashed, maps to, from=2-1, to=2-5]
				\arrow[maps to, from=2-2, to=2-3]
				\arrow[maps to, from=2-3, to=2-5]
				\arrow[maps to, from=2-5, to=2-6]
			\end{tikzcd}
		\end{equation}
		\hfill\\
		
		\noindent Observe by definition if $\rho=\rho_0$, then $\tau_\rho\equiv0$ and hence $\FF_n(\rho_0)=\rho_n$ for all $n$. Given any $\epsilon>0$, we want to show that, for all $n$ sufficiently large $\FF_n$ is an $\epsilon$-isometry 
		(see \eqref{epsiloniso}), that is $\colon$ 
		\medskip
		\begin{enumerate}[{\bf (A)}]
			\item For all $n$ sufficiently large the distortion of $\FF_n$, $\Dis(\FF_n)<\epsilon$\label{A} \\
			
			\item For all $n$ sufficiently large $\FF_n(B_{\MM(Z)}(\rho_0,R))$ is an $\epsilon$-net of $B_{\MM(Z_n)}(\rho_n,R)$.\label{B}
		\end{enumerate}
		\medskip
		\noindent The proof will be complete by the \Cref{distortion} and \ref{epsilon net}, where we prove Statement \eqref{A} and \eqref{B} respectively.
	\end{proof}
	\medskip
	
	\noindent We shall use positive universal constants $c_4,c_5,c_6\cdots>0$ in the proof of the following proposition.
	
	\begin{prop}\label{distortion}
		Fix $R>0$. Given $\epsilon>0$, for all $n$ sufficiently large the distortion of $\FF_n$ (see \eqref{FF}), $\Dis(\FF_n)<\epsilon$, i.e.
		\begin{equation*}
			\big|d_{\MM(Z_n)}(\FF_n(\alpha),\FF_n(\beta))-d_{\MM(Z)}(\alpha,\beta)\big|<\epsilon
		\end{equation*}
		for all $\alpha,\beta\in B_{\MM(Z)}(\rho_0,R)$ (i.e. Statement \eqref{A} is true).
	\end{prop}
	
	\begin{proof}
		Let $\alpha,\beta\in B_{\MM(Z)}(\rho_0,R)$. We know $d_{\MM(Z)}(\alpha,\beta)=\|\tau_\alpha-\tau_\beta\|_\infty$ and by the triangle inequality we have 
		\begin{equation*}
			\begin{split}
				&\Big|\ d_{\MM(Z_n)}(\FF_n(\alpha),\FF_n(\beta))-d_{\MM(Z)}(\alpha,\beta)\ \Big|\\
				&=\ \Big|\ d_{\MM(Z)}(\FF_n(\alpha),\FF_n(\beta))-\|\tau_\alpha-\tau_\beta\|_\infty \Big|\\
				&\le\ \Big|\ d_{\MM(Z)}(\FF_n(\alpha),\FF_n(\beta))-\|\tau_\alpha\circ f_n-\tau_\beta\circ f_n\|_\infty \Big|+\Big|\|\tau_\alpha\circ f_n-\tau_\beta\circ f_n\|_\infty-\|\tau_\alpha-\tau_\beta\|_\infty\Big|.
			\end{split}
		\end{equation*}
		By \Cref{tauf_n}, the second term goes to zero as $n\to \infty$. So, it is enough to show that the first term is small for all $n$ and sufficiently large. The proof relies on straightforward repeated use of the triangle inequality and referring to lemmas, which has been proved so far.
		
		For any $\rho\in B_{\MM(Z)}(\rho_0,R)$, let us denote (see \eqref{FF}) \begin{equation*}
			A_n(\rho)\coloneqq \PP_\infty(C_n(\tau_\rho\circ f_n),\rho_n)
		\end{equation*}
		for notational convenience, then $\FF_n(\rho)=\pi^R_n(A_n(\rho))$. From inequality \eqref{expconv} 
		we have 
		\begin{equation*}
			\|\tau_{A_n(\rho)}-C_n(\tau_\rho\circ f_n)\|_\infty\le 4\|D_{\rho_n}(C_n(\tau_\rho\circ f_n))\|_\infty.
		\end{equation*}
		By \Cref{smoothing of taufn} for $n$ sufficiently large we have 
		\begin{equation}\label{anotherineq}
			\|C_n(\tau_\rho\circ f_n)-\tau_\rho\circ f_n\|_\infty<\epsilon.
		\end{equation} 
		Hence by \eqref{discrepancy 2 lipschitz ineq} and from \cref{discrepancy 2 lipschitz} we get,
		\begin{equation*}
			\begin{split}
				\|D_{\rho_n}(C_n(\tau_\rho\circ f_n))&-D_{\rho_n}(\tau_\rho\circ f_n)\|_\infty\le 2\epsilon\\
				\implies\|D_{\rho_n}(C_n(\tau_\rho\circ f_n))\|_\infty&\le \|D_{\rho_n}(\tau_\rho\circ f_n)\|_\infty+ 2\epsilon
			\end{split}
		\end{equation*} 
		and therefore 
		\begin{equation}\label{crucialineq}
			\|\tau_{A_n(\rho)}-C_n(\tau_\rho\circ f_n)\|_\infty\le c_4(\|D_{\rho_n}(\tau_\rho\circ f_n)\|_\infty+\epsilon).
		\end{equation}
		Now by \eqref{crucialineq} and since $\|C_n(\tau_\rho\circ f_n)\|_\infty\le\|\tau_\rho\circ f_n\|_\infty\le \|\tau_\rho\|_\infty$ we have,
		\begin{equation}\label{effect of retraction 1}
			\begin{split}
				d_{\MM(Z_n)}(A_n(\rho),\rho_n)=\|\tau_{A_n(\rho)}\|_\infty&\le\|C_n(\tau_\rho\circ f_n)\|_\infty+c_4(\|D_{\rho_n}(\tau_\rho\circ f_n)\|_\infty+\epsilon)\\
				&\le\|\tau_\rho\|_\infty+c_4(\|D_{\rho_n}(\tau_\rho\circ f_n)\|_\infty+\epsilon)\\
				&\le R+c_4(\|D_{\rho_n}(\tau_\rho\circ f_n)\|_\infty+\epsilon).
			\end{split}
		\end{equation}
		Again by property of $\pi^R_n$ as in \Cref{retraction lemma} if $d_\MM(A_n(\rho),\rho_n)\le R$ then $d_{\MM(Z_n)}(\FF_n(\rho),A_n(\rho))=0$ and for $d_{\MM(Z_n)}(A_n(\rho),\rho_n)>R$ we have,
		\begin{equation}\label{effect of retraction}
			\begin{split}
				d_{\MM(Z_n)}(\FF_n(\rho),A_n(\rho))=d_{\MM(Z_n)}(\pi^R_n(A_n(\rho)),A_n(\rho))&=d_{\MM(Z_n)}(A_n(\rho),\rho_n)-R\\
				&\le c_4(\|D_{\rho_n}(\tau_\rho\circ f_n)\|_\infty+\epsilon).
			\end{split}
		\end{equation}
		for all $n$ sufficiently large, for all $\rho\in B_{\MM(Z)}(\rho_0,R)$. All the inequalities we shall encounter next are true for all $n$ sufficiently large. For $\alpha,\beta \in B_{\MM(Z)}(\rho_0,R)$ using the triangle inequality
		\begin{equation}\label{crucialineq1}
			\begin{split}
				\Big|\ d_{\MM(Z_n)}(\FF_n(\alpha),\FF_n(\beta))&-d_{\MM(Z_n)}(A_n(\alpha),A_n(\beta))\Big|\\
				&\le d_{\MM(Z_n)}(\FF_n(\alpha),A_n(\alpha))+d_{\MM(Z_n)}(A_n(\beta),\FF_n(\beta))\\
				&\le c_5(\|D_{\rho_n}(\tau_\alpha\circ f_n)\|_\infty+\|D_{\rho_n}(\tau_\beta\circ f_n)\|_\infty+\epsilon).
			\end{split}
		\end{equation}
		Since $d_{\MM(Z_n)}(A_n(\alpha),A_n(\beta))=\|\tau_{A_n(\alpha)}-\tau_{A_n(\beta)}\|_\infty$. By the use of triangle inequality, and from \eqref{crucialineq}, we can write,
		\begin{equation}\label{crucialineq2}
			\begin{split}
				\Big|\ d_{\MM(Z_n)}(A_n(\alpha),A_n(\beta))-&\|C_n(\tau_\alpha\circ 	f_n)-C_n(\tau_\beta\circ f_n)\|_\infty\Big|\\
				\le& \|\tau_{A_n(\alpha)}-C_n(\tau_\alpha\circ f_n)\|_\infty+\|C_n(\tau_\beta\circ f_n)-\tau_{A_n(\beta)}\|_\infty\\
				\le& c_6(\|D_{\rho_n}(\tau_\alpha\circ f_n)\|_\infty+\|D_{\rho_n}(\tau_\beta\circ f_n)\|_\infty+\epsilon).
			\end{split}
		\end{equation}
		Also from \eqref{anotherineq} 
		\begin{equation}\label{crucialineq3}
			\begin{split}
				\Big|\|C_n(\tau_\alpha\circ f_n)&-C_n(\tau_\beta\circ f_n)\|_\infty - \|\tau_\alpha\circ f_n - \tau_\beta\circ f_n\|_\infty\Big|\\
				&\le\|C_n(\tau_\alpha\circ f_n)-\tau_\alpha\circ f_n\|_\infty+\|\tau_\beta\circ f_n-C_n(\tau_\beta\circ f_n)\|_\infty\\
				&\le 2\epsilon.
			\end{split}
		\end{equation}
		Then finally by the above three inequalities \eqref{crucialineq1},\eqref{crucialineq2} and \eqref{crucialineq3} we get,
		\begin{equation*}
			\begin{split}
				\Big|\ d_{\MM(Z_n)}(\FF_n(\alpha),\FF_n(\beta))-&\|\tau_\alpha\circ f_n-\tau_\beta\circ f_n\|_\infty\Big|\\
				\le& c_7(\|D_{\rho_n}(\tau_\alpha\circ f_n)\|_\infty+\|D_{\rho_n}(\tau_\beta\circ f_n)\|_\infty+\epsilon).
			\end{split}
		\end{equation*}
		Now by lemma \ref{discrepancy tau composed f_n} we have for all $n$ sufficiently large
		\begin{equation*}
			\sup_{\rho\ \in\  B_{\MM(Z)}(\rho_0,R)} \big\|D_{\rho_n}(\tau_\rho\circ f_n)\big\|_\infty\ <\ \epsilon.
		\end{equation*}
		Hence for any $\alpha,\beta \in B_{\MM(Z)}(\rho_0,R)$ then for all $n$ sufficiently large we have 
		\begin{equation*}
			\Big|\ d_{\MM(Z_n)}(\FF_n(\alpha),\FF_n(\beta))-\|\tau_\alpha\circ f_n-\tau_\beta\circ f_n\|_\infty\Big|\ <\ c_8\ \epsilon
		\end{equation*}
		where $c_8$ is some positive universal constant (independent of $R$ and $\epsilon $). We could have started with $\epsilon/2c_8$ instead of $\epsilon$ to eventually get this last term smaller than $\epsilon/2$. This completes the proof.
	\end{proof}
	
	\medskip
	
	We need the following lemmas before proving Statement \eqref{B}. The lemma below demonstrates that any AI-convergent sequence of compact semi-metric spaces forms an equicontinuous family. 
	
	\begin{lemma}\label{rho_n equicontinuity}
		Let $\{(Z_n,\rho_n)\}_{n\ge 1}$ be a sequence of compact semi-metric 
		spaces such that $(Z_n,\rho_n)\xrightarrow{AI\ conv.}(Z,\rho_0)$. Then $\{(Z_n,\rho_n)\}_{n\ge 1}$ is an equicontinuous family of 
		compact semi-metric spaces, i.e., given $\epsilon>0$ there exists $\delta=\delta(\epsilon)>0$ such that for all $n$  if $\xi,\eta\in Z_n$ and $\rho_n(\xi,\eta)<\delta$ then 
		\begin{equation*}
			\sup_{\zeta\in Z}\ \big|\rho_n(\xi,\zeta)-\rho_n(\eta,\zeta)\big|<\epsilon.
		\end{equation*}
	\end{lemma}
	\begin{proof}
		From the given hypothesis (by \Cref{AIconv}) we know, there exist $\epsilon_n$-isometries $f_n\colon (Z_n,\rho_n)\to (Z,\rho_0)$ with $\epsilon_n\to 0+$.
		Let $\xi,\eta,\zeta\in Z_n$ for some $n$, then 
		\begin{equation}\label{triangle ineq}
			\begin{split}
				\big|\rho_n(\xi,\zeta)-\rho_n(\eta,\zeta)\big|\le\big|\rho_n(\xi,\zeta)-\rho_0(f_n(\xi),f_n(\zeta))\big|&+\big|\rho_0(f_n(\xi),f_n(\zeta))-\rho_0(f_n(\eta),f_n(\zeta))\big|\\
				&+\big|\rho_0(f_n(\eta),f_n(\zeta))-\rho_n(\eta,\zeta)\big|\\
				<\ \epsilon_n+\big|\rho_0(f_n(\xi),f_n(\zeta))-\rho_0&(f_n(\eta),f_n(\zeta))\big|+\epsilon_n
			\end{split}
		\end{equation}
		since $f_n$'s are $\epsilon_n$-isometries.
		
		\noindent By \Cref{Coro 3} we know that there exists $\delta_1=\delta_1(\epsilon)>0$ such that if $\rho_0(f_n(\xi),f_n(\eta))<2\delta_1$ then 
		\begin{equation*}
			\big|\rho_0(f_n(\xi),f_n(\zeta))-\rho_0(f_n(\eta),f_n(\zeta))\big|<\frac{\epsilon}{2}.
		\end{equation*}
		
		\noindent If $\rho_n(\xi,\eta)<\delta_1$ then for all $n$ sufficiently large, (i.e. for all $n$ greater than some $N_0\in \mathbb N$) we shall have $\rho_0(f_n(\xi),f_n(\eta))<\rho_n(\xi,\eta)+\epsilon_n<2\delta_1$ ($f_n$ being $\epsilon_n$-isometry with $\epsilon_n \to 0+$). Consequently, from \eqref{triangle ineq} we have
		\begin{equation*}
			\big|\rho_n(\xi,\zeta)-\rho_n(\eta,\zeta)\big|<\epsilon
		\end{equation*}
		
		Therefore, we can find a $\delta = \delta(\epsilon) > 0$ such that for all $n$, the claimed result holds by applying \Cref{Coro 3} to the antipodal spaces $(Z_n, \rho_n)$ for $n \le N_0$.
	\end{proof}
	
	\medskip
	
	Observe the following \Cref{equicontinuity of tau 1} is an extension of \Cref{Coro 1} proved earlier. Here, instead of a single antipodal space, we are considering a sequence of antipodal spaces, which AI-converges to another antipodal space and consequently satisfies an equicontinuity condition by \Cref{rho_n equicontinuity}.
	\begin{lemma}
		\label{equicontinuity of tau 1}
		Let $\{(Z_n,\rho_n)\}_{n\ge 1}$ be a sequence of antipodal spaces such that $(Z_n,\rho_n)\xrightarrow{AI\ conv.}(Z,\rho_0)$. Fix $R>0$, given $\epsilon>0$ there exists  $\delta=\delta(R,\epsilon)>0$ such that for any $n$, if $\xi,\eta\in Z_n$ and $\rho_n(\xi,\eta)<\delta$ then 
		\begin{equation*}
			\big|\tau_\rho(\xi)-\tau_\rho(\eta)\big|<\epsilon
		\end{equation*}
		for all $\rho\in{B}_{\MM(Z_n)}(\rho_n,R)$.
	\end{lemma}
	\begin{proof}
		By the  above \Cref{rho_n equicontinuity}, $\{(Z_n,\rho_n)\}_{n\ge 1}$ is an equicontinuous family of antipodal spaces. Then the claim follows from \Cref{equicontinuity of tau main}.
	\end{proof}
	
	\medskip
	
	The following Lemma is similar to \Cref{discrepancy tau composed f_n}.
	
	\begin{lemma}
		\label{discrepancy tau composed g_n}
		Let $\{(Z_n,\rho_n)\}_{n\ge 1}$ be a sequence of antipodal spaces such that $(Z_n,\rho_n)\xrightarrow{AI\ conv.}(Z,\rho_0)$. Fix $R>0$, given $\epsilon>0$, there exists $\delta=\delta(R,\epsilon)>0$ such that$\colon$  for any $n$, and for any $\delta$-isometry $g_n\colon (Z,\rho_0)\to(Z_n,\rho_n)$, we have 
		\begin{equation}
			\sup_{\ \rho\ \in {B}_{\MM(Z_n)}(\rho_n,R)}\big\|D_{\rho_0}(\tau_\beta\circ g_n)\big\|_\infty\ <\ \epsilon.
		\end{equation}
	\end{lemma}
	\begin{proof}
		Since $\{Z_n,\rho_n\}_{n\ge 1}$ forms an equicontinuous family of antipodal spaces by \Cref{rho_n equicontinuity}, we have the lemma as a straightforward consequence of \Cref{discrepancy main}, 
	\end{proof}
	
	\medskip
	
	Now we prove Statement \eqref{B} and hence the last step in the proof of \Cref{forward}. Again, we shall use positive universal constants $c_9,c_{10},\cdots>0$ in the following proof.
	
	\medskip
	
	\begin{prop}\label{epsilon net}
		Fix $R>0$. Given $\epsilon>0$, for all $n$ sufficiently large the image of the function $\FF_n$, as defined in \eqref{FF}, is an $\epsilon$-net in ${B}_{\MM(Z_n)}(\rho_n,R)\subset \MM(Z_n)$.
	\end{prop}
	\begin{proof}
		We will show that, for all $n$ sufficiently large, if $\beta\in {B}_{\MM(Z_n)}(\rho_n,R)$ then we can construct $\hat{\beta_n}\in {B}_{\MM(Z)}(\rho_0,R)$ such that $d_{\MM(Z_n)}(\beta,\FF_n(\hat{\beta_n}))<\epsilon$.
		
		As in proof of \Cref{inverse}, for every $n$ consider $F_n$ a finite $(2\epsilon_n)$-net of $Z_n$ with respect to $\rho_n$ such that $\rho_0(\xi,\eta)\ge 2\epsilon_n$ for all distinct $\xi,\eta\in F_n$. Moreover, $f_n$ maps $F_n$ injectively onto a $G_n\coloneqq f_n(F_n)\subset Z$ for all $n$. Given any $\delta>0$, for all $n$ sufficiently large $G_n$ becomes a $\delta$-net of $Z$ with respect to $\rho_0$ (cf. proof of \Cref{inverse}).
		
		Let us choose $\delta>0$ such that for all $n$ sufficiently large it satisfies the claims of \Cref{rho_n equicontinuity} and \Cref{discrepancy tau composed g_n}, that is for $n$ sufficiently large, if $\xi,\eta \in Z_n$ with $\rho_n(\xi,\eta)<\delta$ then
		\begin{equation}\label{recall 1}
			\big|\tau_\rho(\xi)-\tau_\rho(\eta)\big|<\epsilon
		\end{equation}
		for all $\rho\in B_{\MM(Z_n)}(\rho_n,R)$. 
		
		Also choose and fix $\delta_0<\delta/2$ such that for any $n$, if we have a $\delta_0$-isometry $g_n\colon (Z,\rho_0)\to (Z_n,\rho_n)$ then,
		\begin{equation}\label{recall 2}
			\sup_{\ \rho\ \in {B}_{\MM(Z_n)}(\rho_n,R)}\big\|D_{\rho_0}(\tau_\beta\circ g_n)\big\|_\infty\ <\ \epsilon.
		\end{equation}

		By \Cref{inverse}, for every $n$ sufficiently large, there exist $\delta_0$-isometries  $g_n\colon (Z,\rho_0)\to (Z_n,\rho_n)$, which satisfies the injectivity relation \eqref{injectivity relation}, 
		\begin{equation*}
			(g_n\circ f_n)|_{F_n}=\text{id}_{F_n}\quad \hbox{and} \quad g_n(G_n)=F_n.
		\end{equation*}
		
		Note $G_n=f_n(F_n)$'s are $\delta_0$ net of $Z$ with respect to $\rho_0$ for all $n$ sufficiently large. Let $\{Q^n_{f_n(\zeta)}\}_{\zeta\in F_n}$ be the partition of unity on $(Z,\rho_0)$ 
		associated with the $\delta_0$-net $f_n(F_n)=G_n$ (refer to \Cref{PoU}). For ease of notation we suppress `$f_n$' in the subscript of $Q^n_{f_n(\zeta)}$ and just write `$Q^n_\zeta$' instead. 
		
		Consider, $\hat{C_n}: (B(Z),\|\cdot\|_\infty)\to (C(Z),\|\cdot\|_{\infty})$, the smoothing operator corresponding to the partition of unity $\{Q^n_{\zeta}\}_{\zeta\in F_n}$ on $(Z,\rho_0)$ associated with the $\delta_0$-net $f_n(F_n)=G_n$ (see \Cref{PoU}), where 
		\begin{equation*}
			\hat{C_n}(\tau)(\cdot)=\sum_{\zeta\in F_n}\tau(f_n(\zeta))\cdot Q^n_\zeta(\cdot) \quad \hbox{for}\quad \tau\in (B(Z),\|\cdot\|_\infty).
		\end{equation*} Then for all $n$ sufficiently large and $\beta\in {B}_{\MM(Z_n)}(\rho_n,R)$ we can define the continuous function\hfill\\ $\sigma_n\coloneqq \hat{C_n}(\tau_\beta\circ g_n)\colon Z\to \mathbb{R}$ such that,
		\begin{equation*}
			\sigma_n(\xi)\coloneqq \hat{C_n}(\tau_\beta\circ g_n)(\xi)=\sum_{\zeta\in F_n}\tau_\beta\circ(f_n(\zeta))\cdot Q^n_\zeta(\xi)=\sum_{\zeta\in F_n}\tau_\beta(\zeta)\cdot Q^n_\zeta(\xi)
		\end{equation*}
		(the above equality is justified because $g_n$ satisfies the injectivity relation \eqref{injectivity relation}). Observe that $\|\sigma_n\|\le\|\tau_\beta\|\le R$.
		
		Next, we define, 
		\begin{equation*}
			\hat{\beta_n}=\pi^R_0(\mathcal{P}_\infty(\sigma_n,\rho_0)),
		\end{equation*} where $\pi^R_0\colon \MM(Z)\to {B}_{\MM(Z)}(\rho_0,R)$ is the same retraction as defined in \ref{ret}. So by definition $\hat{\beta_n}\in{B}_{\MM(Z)}(\rho_0,R)$. We will show, $d_{\MM(Z_n)}(\beta,\FF_n(\hat{\beta_n}))<\epsilon$.
		
		\medskip
		
		For any $\xi\in Z$, we know $Q^n_\zeta(\xi)>0$ if and only if $\zeta\in S_n(\xi)\coloneqq \{\zeta\in F_n\ |\ \rho_0(\xi,f_n(\zeta))<\delta_0\}\subset F_n$. Consequently, we have (cf. \Cref{smoothing of tau lemma}),
		\begin{equation*}
			\sum_{\zeta\in S_n(\xi)}Q^n_\zeta(\xi)=1 \quad \text{ and } \quad \sigma_n(\xi)=\sum_{\zeta\in S_n(\xi)}\tau_\beta(\zeta)Q^n_\zeta(\xi).
		\end{equation*}
		Since we have $\epsilon\to 0+$, for all $n$ sufficiently large, if $\eta\in Z_n$ (so, $f_n(\eta)\in Z$) and $\zeta\in S_n(f_n(\eta))$ then, 
		\begin{equation*}
			\rho_n(\eta,\zeta)<\rho_0(f_n(\eta),f_n(\zeta))+\epsilon_n<\delta_0+\epsilon_n<\delta.
		\end{equation*}
		Therefore, by the choice of $\delta$ above \eqref{recall 1} observe that for $\eta\in Z_n$,
		\begin{equation*}
			\Big|\sigma_n\circ f_n(\eta)-\tau_\beta(\eta)\Big|\le\sum_{\zeta\in S_n(f_n(\eta))}\Big|\tau_\beta(\zeta)-\tau_\beta(\eta)\Big|\cdot Q^n_\zeta(f_n(\eta))\ <\ \epsilon.
		\end{equation*}
		Hence we have 
		\begin{equation*}
			\big\|\sigma_n\circ f_n-\tau_\beta\big\|_\infty<\epsilon.
		\end{equation*}
		
		For all $n$ sufficiently large, for $\xi\in Z$, by definition of $S_n(\xi)$, we have $\rho_0(\xi,f_n(\zeta))<\delta_0$ for each $\zeta\in S_n(\xi)$. Now $g_n\colon (Z,\rho_0)\to (Z_n,\rho_n)$ being $\delta_0$-isometry, we have for $\zeta \in S_n(\xi)$
		\begin{equation*}
			\rho_n(g_n(\xi),\zeta)=\rho_n(g_n(\xi),g_n\circ f_n(\zeta))<\rho_0(\xi,f_n(\zeta))<\delta_0<\delta,
		\end{equation*} (since $g_n\circ f_n(\zeta)=\zeta$). 
		So we have for all $n$ sufficiently large, for $\beta \in{B}_{\MM(Z_n)}(\rho_n,R)$ and $\xi\in Z$ we get,
		\begin{equation*}
			\Big|\sigma_n(\xi)-\tau_\beta\circ g_n(\xi)\Big|\le\sum_{\zeta\in S_n(\xi)}\Big|\tau_\beta(\zeta)-\tau_\beta(g_n(\xi))\Big|\cdot Q^n_\zeta(\xi)<\ \epsilon
		\end{equation*}
		similarly as above. Hence (cf. \Cref{smoothing of taufn})
		\begin{equation*}
			\big\|\sigma_n-\tau_\beta\circ g_n\big\|_\infty<\epsilon.
		\end{equation*}
		Let us denote $A_0(\sigma_n)\coloneqq \mathcal{P}_\infty(\sigma_n,\rho_0)$, then by property of antipodalization map \eqref{expconv} we have, 
		\begin{equation*}
			\big\|\tau_{A_0(\sigma_n)}-\sigma_n\big\|_\infty\le 4\big\|D_{\rho_0}(\sigma_n)\big\|_\infty<c_{10}(\big\|D_{\rho_0}(\tau_\beta\circ g_n)\big\|_\infty+\epsilon).
		\end{equation*}
		Therefore, by similar application of \Cref{retraction lemma}, and by the same argument as in \eqref{crucialineq},\eqref{effect of retraction 1} and \eqref{effect of retraction} give us,	
		\begin{align*}
			\big\|\tau_{\hat{\beta_n}}-\tau_{A_0(\sigma_n)}\big\|_\infty=d_{\MM(Z)}(\hat{\beta_n},A_0(\sigma_n))&=d_{\MM(Z)}(\pi^R_0(A_0(\sigma_n)),A_0(\sigma_n))\\
			&=d_{\MM(Z)}(A_0(\sigma_n),\rho_0)-R\\
			&< c_{10}(\big\|D_{\rho_0}(\tau_\beta\circ g_n)\big\|_\infty+\epsilon).
		\end{align*}
		
		For $\beta\in {B}_{\MM(Z_n)}(\rho_n,R)$  we have,
		\begin{equation*}
			\begin{split}
				d_{\MM(Z_n)}(\FF_n(\hat{\beta_n}),\beta)&=\big\|\tau_{\FF_n(\hat{\beta_n})}-\tau_\beta\big\|_\infty\\
				&\le \big\|\tau_{\FF_n(\hat{\beta_n})}-\tau_{\hat{\beta_n}}\circ f_n\big\|_\infty+\big\|\tau_{\hat{\beta_n}}\circ f_n-\tau_\beta\big\|_\infty.\\
			\end{split}
		\end{equation*}
		From the proof of \Cref{distortion}, arguing as in \eqref{anotherineq},\eqref{crucialineq} and \eqref{effect of retraction},
		\begin{equation*}
			\big\|\tau_{\FF_n(\hat{\beta_n})}-\tau_{\hat{\beta_n}}\circ f_n\big\|_\infty< c_{11}(\big\|D_{\rho_n}(\tau_{\hat{\beta_n}}\circ f_n)\big\|_\infty+\epsilon).
		\end{equation*}
		For the other term 
		\begin{equation*}
			\big\|\tau_{\hat{\beta_n}}\circ f_n-\tau_\beta\big\|_\infty\le \big\|\tau_{\hat{\beta_n}}\circ f_n-\sigma_n\circ f_n\big\|_\infty+\big\|\sigma_n\circ 	f_n-\tau_\beta\big\|_\infty.
		\end{equation*}
		From above we have for all $n$ sufficiently large $\|\sigma_n\circ 	f_n-\tau_\beta\|_\infty<\epsilon$, and 
		\begin{equation*}
			\begin{split}
				\big\|\tau_{\hat{\beta_n}}\circ f_n-\sigma_n\circ f_n\big\|_\infty&\le \big\|\tau_{\hat{\beta_n}}-\sigma_n\big\|_\infty\\
				& \le\big\|\tau_{\hat{\beta_n}}-\tau_{A_0(\sigma_n)}\big\|_\infty+\big\|\tau_{A_0(\sigma_n)}-\sigma_n\textbf{}\big\|_\infty\\
				&< c_{12}(\big\|D_{\rho_0}(\tau_\beta\circ g_n)\big\|_\infty+\epsilon).
			\end{split}
		\end{equation*}
		Therefore, we finally have for all $n$ sufficiently large
		\begin{equation*}
			d_{\MM(Z_n)}(\FF_n(\hat{\beta_n}),\beta)< c_{13}(\|D_{\rho_n}(\tau_{\hat{\beta_n}}\circ f_n)\|_\infty+\|D_{\rho_0}(\tau_\beta\circ g_n)\|_\infty+\epsilon).
		\end{equation*}
		Now by \Cref{discrepancy tau composed f_n} and \Cref{discrepancy tau composed g_n} we have for all $n$ sufficiently large
		\begin{equation*}
			\sup_{\rho\ \in\  B_{\MM(Z)}(\rho_0,R)} \|D_{\rho_n}(\tau_\rho\circ f_n)\|_\infty\ <\ \epsilon
			\quad \text{and} \quad 
			\sup_{\ \rho\ \in\  B_{\MM(Z_n)}(\rho_n,R)}\|D_{\rho_0}(\tau_\beta\circ g_n)\|_\infty\ <\ \epsilon,
		\end{equation*}
		respectively, and therefore 
		\begin{equation*}
			d_{\MM(Z_n)}(\FF_n(\hat{\beta_n}),\beta)\le \ c_{14}\ \epsilon,
		\end{equation*}
		where $c_{14}$ is some positive universal constant. Again, starting with $\epsilon/2c_{14}$ instead of $\epsilon$, we would have eventually got this last term less or equal to $\epsilon/2$. This finishes the proof.
	\end{proof}
	
	\bigskip

	\section{Proof of \Cref{backward}}
	Let $\{(X_n,x_n)\}_{n\ge 1}$ be a sequence of Gromov product space, such that $(X_n,x_n)\xrightarrow{GH conv.}(X,x)$, where $X$ is another Gromov product space. 
	
	\noindent We will demonstrate that, for any given $\delta > 0$, one can construct $\delta$-isometries $g_n\colon (\partial_P X, \rho_x) \to (\partial_P X_n, \rho_{x_n})$ for all $n$ sufficiently large. By hypothesis $\{(\partial_P X_n, \rho_{x_n})\}_{n \ge 1}$, is an equicontinuous family of antipodal functions. Thus, by \Cref{conditional inverse}, this will be sufficient to conclude that $(\partial_P X_n, \rho_{x_n}) \xrightarrow{AI\ conv.} (\partial_P X, \rho_x)$ as claimed.
	
	Any Gromov product space $Y$ is isometrically embedded in $\MM(\partial_P Y)$ via the visual embedding $i_Y\colon y\mapsto \rho_y$. For any $y_1,y_2\in Y$ 
	\begin{equation*}
		\log \frac{d\rho_{y_2}}{d\rho_{y_1}}= B(y_1,y_2,\ \cdot\ )
	\end{equation*} where $B\colon Y\times Y\times \partial_P Y \to \R$, is the Busemann cocycle. Hence $$\argmx \frac{d\rho_{y_2}}{d\rho_{y_1}}= \argmx\  B(y_1,y_2,\ \cdot\ )$$ (see \cite[Section 5, and 6]{biswas2024quasi},\cite{biswas2015moebius}). 
	
	We shall treat $X_n$'s and $X$ as isometrically embedded subspaces of $\MM(\partial_P X_n)$'s and $\MM(\partial_P X)$, respectively. 
	
	Since $(X_n,\rho_{x_n})\xrightarrow{GH\ conv.}(X,\rho_x)$, for every $R>0$, there exists $\epsilon_n$-isometries 
	\begin{equation*}
		\FF_n\colon B_{\MM(\partial_P X)}(\rho_x,R)\cap X\cong B_X(x,R)\to B_{X_n}(x_n,R)\cong B_{\MM(\partial_P X_n)}(\rho_{x_n},R)\cap X_n
	\end{equation*} such that $\epsilon_n \to 0+$ and $\FF_n(x)=x_n$. We shall sometimes abuse notation and to think of $\FF_n$ as a function from $B_{\MM(\partial_P X)}(\rho_x,R)\cap X$ to $B_{\MM(\partial_P X_n)}(\rho_{x_n},R)\cap X_n$.
	
	Let $\delta>0$ be given. Now, $\{(\partial_P X_n,\rho_x)\}_{n\ge 1}$ being equicontinuous, we can choose $\epsilon=\epsilon(\delta)<\delta$ positive such that if $\xi,\xi',\eta,\eta'\in \partial_P X$ (respectively $\in \partial_P X_n$) with $\max\{\rho_x(\xi,\xi'),\rho_x(\eta,\eta')\}<\epsilon$ (respectively $\max\{\rho_{x_n}(\xi,\xi'),\rho_{x_n}(\eta,\eta')\}<\epsilon$) then
	\begin{equation}\label{equicontinuity used}
		|\rho_x(\xi,\eta)-\rho_x(\xi',\eta')|<\delta \quad \hbox{(respectively } |\rho_{x_n}(\xi,\eta)-\rho_{x_n}(\xi',\eta')|<\delta \hbox{)} 
	\end{equation} 
	from \Cref{Coro 1} and equicontinuity of $\rho_{x_n}$'s.
	
	\bigskip
	
	\subsection{Construction of the candidate for almost isometry $g_n$}\label{construction of gn}\hfill\\
	We can define a map 
	\begin{equation*}
		z\ \colon\ \partial_P X\to S_X(x,R)\coloneqq \{\ y\in X\ |\ d_X(x,y)=R\ \},
	\end{equation*}
	by choosing for each $\zeta\in \partial_P X$, a point $z(\zeta)$ on the sphere of radius $R$ centered at $x$, such that $z(\zeta)$ is on the geodesic ray joining $x$ to $\zeta$. We denote by $\alpha_\zeta$ the visual antipodal function $\rho_{z(\zeta)}\in \MM(\partial_P X)$. Note that
	\begin{equation*}
		\zeta \in \argmx \frac{d\alpha_\zeta}{d\rho_x}=\argmx\  B(x,z(\zeta),\ \cdot\ )\quad \hbox{ and }\quad  d_X(x,z(\zeta))=\log \frac{d\alpha_\zeta}{d\rho_x}(\zeta)
	\end{equation*} (see \cite[Section 6]{biswas2024quasi}).
	
	Fix $R>0$ large, so that 
	\begin{equation}\label{define epsilon_0}
		\epsilon_0\coloneqq \frac{2e^{-R}}{\epsilon}<\epsilon.
	\end{equation} 
	Let $F$ be a finite $\epsilon_0$-net of $\partial_P X$ with respect to $\rho_x$, such that $\rho_x(\zeta, \zeta') \ge \epsilon_0$, for all distinct $\zeta, \zeta' \in F$. By \Cref{gromov ip upper bound}, for distinct $\zeta,\zeta'\in F$ we get,
	\begin{equation}\label{gromov ip upper bound1}
		\exp((z(\zeta)|z(\zeta'))_x)=\exp((\alpha_{\zeta}|\alpha_{\zeta'})_{\rho_x})\le \frac{1}{\rho_x(\zeta,\zeta')}\le\frac{1}{\epsilon_0}.
	\end{equation} Hence by the choice of $\epsilon_0$ and $R$ (see \eqref{define epsilon_0}), we have, 
	\begin{equation*}
		d_X(z(\zeta),z(\zeta'))=d_{\MM(\partial_P X)}(\alpha_\zeta,\alpha_{\zeta'})> \log \left(\frac{1}{\epsilon}\right)
	\end{equation*}
	(we can choose $\epsilon$ small enough to get a large lower bound here). Thus, the points $z(\zeta)$'s and hence the corresponding $\alpha_\zeta$'s are distinct for $\zeta\in F$.
	
	By the axiom of choice, we define $p \colon \partial_P X \to F$, a retraction map, such that $\rho_x(\xi, p(\xi)) < \epsilon_0$ for all $\xi \in \partial_P X$ and $p|_{F} = \text{id}_{F}$. 
	
	We define a map $\tilde{g_n}\colon F \to \partial_P X_n$ as follows: for each $\zeta\in F$ we choose 
	\begin{equation*}
		\zeta_n\in \argmx\ B(x_n,\FF_n(z(\zeta)),\ \cdot\ )=\argmx\  \frac{d\FF_n(\alpha_{\zeta})}{d\rho_{x_n}}\subset \partial_P X_n.
	\end{equation*}  and define $\tilde{g_n}(\zeta)\coloneqq \zeta_n$ ( the point $\tilde{g_n}(\zeta)\in \partial_P X$ is the end point of a geodesic ray we get by extending the geodesic segment joining $x_n$ to $\FF_n(z(\zeta))$, i.e. $\FF_n(z(\zeta))$ is on a geodesic ray $[x_n,\zeta_n)$).
	
	We define the map 
	\begin{equation}\label{candidate}
		\begin{split}
			g_n\colon \partial_P X&\to \partial_P X_n\\
			\xi&\mapsto \tilde{g_n}\circ p(\xi)
		\end{split}
	\end{equation} i.e. composition of the retraction map $p$ with $\tilde{g_n}$. 
	
	We will show that this map $g_n$ is a $3\delta$-isometry for all $n$ sufficiently large.
	
	\subsection{The constructed candidate $g_n$ has small distortion for all $n$ sufficiently large}\hfill\\
	
	\noindent Note that the map $g_n$ depends on where it maps the finite set $F$.  
	Suppose $\xi,\eta\in \partial_P X$, let $p(\xi)=\zeta$, $p(\eta)=\zeta'\in F$, where $p$ is the retraction defined above. Then $\max\{\rho_x(\xi,\zeta),$ $\rho_x(\eta,\zeta')\}<\epsilon_0<\epsilon$, which implies $|\rho_x(\zeta,\zeta')-\rho_x(\xi,\eta)|<\delta$ (from \eqref{equicontinuity used}). Note, $g_n(\xi)=g_n(\zeta)$ and $g_n(\eta)=g_n(\zeta')$, thus,
	\begin{equation*}
		\begin{split}
			|\rho_{x_n}(g_n(\xi),g_n(\eta))-\rho_x(\xi,\eta)|&=|\rho_{x_n}(g_n(\zeta),g_n(\zeta'))-\rho_x(\zeta,\zeta')|+|\rho_x(\zeta,\zeta')-\rho_x(\xi,\eta)|\\
			&\le |\rho_{x_n}(g_n(\zeta),g_n(\zeta'))-\rho_x(\zeta,\zeta')|+\delta.
		\end{split}
	\end{equation*} Eventually, by the \Cref{distortion on finite set} below we have, distortion of $g_n$, 
	\begin{equation*}
		\Dis(g_n)= \sup_{\xi,\eta\in \partial_P X}|\rho_{x_n}(g_n(\xi),g_n(\eta))-\rho_x(\xi,\eta)|\le \delta + 4\epsilon_n+\delta<3\delta,
	\end{equation*}for all $n$ sufficiently large (since $\epsilon_n\to 0+$).
	
	\begin{sublemma}\label{distortion on finite set}
		For $\zeta,\zeta'\in F$ the map $g_n$ defined above in \eqref{candidate} (for all $n$ sufficiently large) satisfies,\\
		\begin{enumerate}[(i)]
			\item $\rho_{x_n}(g_n(\zeta),g_n(\zeta'))
			\le \rho_x(\zeta,\zeta')+\delta+4\epsilon_n$,\\
			\item $\rho_x(\zeta,\zeta')\le \rho_{x_n}(g_n(\zeta),g_n(\zeta'))+\delta+4\epsilon_n$.\\
		\end{enumerate}Thus for all $\zeta,\zeta,\in F$,
		\begin{equation*} 
			|\rho_{x_n}(g_n(\zeta),g_n(\zeta'))-\rho_x(\zeta,\zeta')|\le \delta+ 4\epsilon_n.
		\end{equation*} 
	\end{sublemma}

	\begin{proof}
		In both (i) and (ii), we use similar ideas. In the former, we use the continuity of $\rho_x$, while in the latter, we use the equicontinuity of $\rho_{x_n}$'s to conclude the claims.  
		
		Let $\zeta,\zeta'\in F$ and for ease of notation define $u=z(\zeta)$, $v=z(\zeta')$, $u_n=\FF_n(z(\zeta))$ and $v_n=\FF_n(z(\zeta'))$. Define the corresponding visual antipodal functions $\alpha=\alpha_{\zeta}$, and  $\beta=\alpha_{\zeta'}$. 
		Denote \begin{equation*}
			g_n(\zeta)=\zeta_n\in \argmx \frac{d\FF_n(\alpha)}{d\rho_{x_n}}\quad \hbox{ and }\quad g_n(\zeta')=\zeta'_n\in \argmx\frac{d \FF_n(\beta)}{d\rho_{x_n}}
		\end{equation*} as above. 
		
		\underline{\it Proof of (i).} By \Cref{gromov ip upper bound}
		\begin{equation*}
			\rho_{x_n}(\zeta_n,\zeta'_n)\le\exp(-(\FF_n(\alpha)|\FF_n(\beta))_{\rho_{x_n}})=\exp(-(u_n|v_n)_{x_n}).
		\end{equation*} 
		Now, $\FF_n$ being $\epsilon_n$-isometry with $\FF_n(x)=x_n$ we have 
		\begin{equation}\label{2 epsilon difference}
			\begin{split}
				\big|(\FF_n(\alpha)|\FF_n(\beta))_{\rho_{x_n}}-(\alpha|\beta)_{\rho_x}\big|&=\big|(u_n|v_n)_{x_n}-(u|v)_x\big|
				<2\epsilon_n\\
				\hbox{ which implies }\quad \big|\exp(-(u_n|v_n)_{x_n})- \exp(-(u|v)_x)\big|&\le \big|(u_n|v_n)_{x_n}-(u|v)_x\big| <2\epsilon_n
			\end{split}
		\end{equation}since $(t\mapsto e^{-t})$ is $1$-Lipschitz for $t\ge 0$. Consequently,
		\begin{equation*}
			\begin{split}
				\rho_{x_n}(\zeta_n,\zeta'_n)\le \exp(-(u_n|v_n)_{x_n})&\le \exp(-(u|v)_x)+2\epsilon_n.\\
			\end{split}
		\end{equation*}
		
		We know $X$ is a geodesically complete space. Any geodesic segment $[u,v]$ can be extended bi-infinitely on either side. On one side beyond $u$ it runs to $\tilde{\zeta}\in \partial_P X$ and on the other side beyond $v$ it runs to $\tilde{\zeta'}\in \partial_P X$.
		Then 
		\begin{equation*}
			\begin{split}
				&\tilde{\zeta}\in \argmx\ B(v,u,\ \cdot\ )= \argmx\ \frac{d\alpha}{d\beta}\implies B(v,u,\tilde{\zeta})=d_X(u,v)=\log\frac{d\alpha}{d\beta}(\tilde{\zeta}),\\
				\hbox{ and } &\tilde{\zeta'}\in \argmx\ B(u,v,\ \cdot\ )= \argmx\ \frac{d\beta}{d\alpha}\implies B(u,v,\tilde{\zeta'})=d_X(u,v)=\log\frac{d\beta}{d\alpha}(\tilde{\zeta'}).
			\end{split}
		\end{equation*} Moreover, $\alpha(\tilde{\zeta},\tilde{\zeta'})=\beta(\tilde{\zeta},\tilde{\zeta'})=1$. By GMVT (\Cref{GMVT}) we have,
		\begin{equation}\label{getting epsilon bound}
			\begin{split}
				1\ge \alpha(\zeta,\tilde{\zeta})^2=\rho_x(\zeta,\tilde{\zeta})^2\frac{d\alpha}{d\rho_x}(\zeta)\frac{d\alpha}{d\rho_x}(\tilde{\zeta})
				&=\rho_x(\zeta,\tilde{\zeta})^2\cdot \frac{d\alpha}{d\rho_x}(\zeta) \cdot\frac{d\alpha}{d\beta}(\tilde{\zeta})\cdot\frac{d\beta}{d\rho_x}(\tilde{\zeta})\\
				&\ge \rho_x(\zeta,\tilde{\zeta})^2\cdot \exp(d_X(u,x)+d_X(u,v)-d_X(v,x))\\
				&= \rho_x(\zeta,\tilde{\zeta})^2\cdot\exp(-2(u|v)_x+2d_X(u,x))\\
				&= \rho_x(\zeta,\tilde{\zeta})^2\cdot\exp(-2(u|v)_x) \cdot e^{2R}
			\end{split}
		\end{equation}
		Since $\rho_x(\zeta,\zeta')\ge \epsilon_0$, from \eqref{gromov ip upper bound1}, and by the choice of $\epsilon_0$ and $R$ (see \eqref{define epsilon_0}), we get,
		\begin{equation}\label{easy direction}
			\begin{split}
				\rho_x(\zeta,\tilde{\zeta})\le \exp((u|v)_x)\cdot e^{-R}
				=\exp((\alpha_\zeta|\alpha_{\zeta'})_{\rho_x})\cdot e^{-R}
				\le \frac{e^{-R}}{\epsilon_0}
				<\frac{\epsilon}{2}.
			\end{split}
		\end{equation}
		Similarly, we also have $\rho_x(\zeta',\tilde{\zeta'})<\epsilon/2$. 
		
		Hence by continuity of $\rho_x$ \eqref{equicontinuity used} we have $|\rho_x(\zeta,\zeta')-\rho_x(\tilde{\zeta},\tilde{\zeta'})|<\delta$. Moreover, we have
		\begin{equation}\label{gromov ip lower bound}
			\begin{split}
				\rho_x(\tilde{\zeta},\tilde{\zeta'})=\exp(-(\tilde{\zeta}|\tilde{\zeta'})_x)
				\ge\exp(-(u|v)_x).
			\end{split}
		\end{equation} since the Gromov product decreases as the two points $u$ and $v$ go in opposite directions towards $\tilde{\zeta}$ and $\tilde{\zeta'}$ respectively along the bi-infinite geodesic joining them. The inequality \eqref{gromov ip lower bound} is true for any Gromov product space in general by geodesic completeness.
		
		Therefore,
		\begin{equation*}
			\begin{split}
				\rho_{x_n}(\zeta_n,\zeta'_n)&\le\exp(-(u|v)_x)+2\epsilon_n\\
				&\le \rho_x(\tilde{\zeta},\tilde{\zeta'})+2\epsilon_n\\
				&\le \rho_x(\zeta,\zeta')+\delta+2\epsilon_n. 
			\end{split}
		\end{equation*}
		
		\medskip
		
		\underline{\it Proof of (ii).} From \eqref{gromov ip upper bound1} and \eqref{2 epsilon difference}  we have, for all $n$ sufficiently large,
		\begin{equation*}
			\begin{split}
				\rho_x(\zeta,\zeta')&\le\exp(-(u|v)_x)
				\le \exp(-(u_n|v_n)_{x_n})+2\epsilon_n.\\
			\end{split}
		\end{equation*}
		Let us now denote $\alpha_n=\FF_n(\alpha)$ and $\beta_n=\FF_n(\beta)$. Again $X_n$ being geodesically complete we can find $\tilde{\zeta_n},\tilde{\zeta'_n}\in \partial_P X_n$ such that
		\begin{equation*}
			\begin{split}
				&\tilde{\zeta_n}\in \argmx\ B(v_n,u_n,\ \cdot\ )= \argmx\ \frac{d\alpha_n}{d\beta_n}\implies B(v_n,u_n,\tilde{\zeta_n})=d_X(u_n,v_n)=\log\frac{d\alpha_n}{d\beta_n}(\tilde{\zeta_n}),\\
				\hbox{ and } &\tilde{\zeta'_n}\in \argmx\ B(u_n,v_n,\ \cdot\ )= \argmx\ \frac{d\beta_n}{d\alpha_n}\implies B(u_n,v_n,\tilde{\zeta'_n})=d_X(u_n,v_n)=\log\frac{d\beta_n}{d\alpha_n}(\tilde{\zeta'_n}),
			\end{split}
		\end{equation*} 
		as in the previous paragraph. Also $\tilde{\zeta'_n}\in \argmn \frac{d\alpha_n}{d\beta_n}$ and $\alpha_n(\tilde{\zeta_n},\tilde{\zeta'_n})=\beta_n(\tilde{\zeta_n},\tilde{\zeta'_n})=1$. Note that we are doing everything again at the $n$th level. Similarly, as in \eqref{getting epsilon bound}, we have 
		\begin{equation*}
			1\ge\alpha_n(\zeta_n,\tilde{\zeta_n})^2\ge\rho_{x_n}(\zeta_n,\tilde{\zeta'_n})^2\cdot \exp(-2(u_n|v_n)_{x_n}+2d_{X_n}(u_n,x_n)).
		\end{equation*} 
		Therefore, using \eqref{gromov ip upper bound1} and $\FF_n$ being $\epsilon_n$-isometry, for $n$ sufficiently large, (since $\epsilon_n \to 0+$) similarly as in \eqref{easy direction} we get, 
		\begin{equation*}
			\begin{split}
				\rho_{x_n}(\zeta_n,\tilde{\zeta_n})\le \exp((u_n|v_n)_{x_n}-d_{X_n}(u_n,x_n))
				&\le \exp((u|v)_x-d_{ X}(u,x))+3\epsilon_n\\
				&\le \exp((u|v)_x)\cdot e^{-R}+3\epsilon_n\\
				&\le \frac{\epsilon}{2}+3\epsilon_n\\
				&< \epsilon.
			\end{split}
		\end{equation*} Similarly, we have for all $n$ sufficiently large $\rho_{x_n}(\zeta'_n,\tilde{\zeta'_n})<\epsilon$. Again by equicontinuity property of $\rho_{x_n}$'s \eqref{equicontinuity used} for all $n$ we have $|\rho_{x_n}(\zeta_n,\zeta'_n)-\rho_{x_n}(\tilde{\zeta_n},\tilde{\zeta'_n})|<\delta$ and also as in \eqref{gromov ip lower bound} we have,
		$\rho_{x_n}(\tilde{\zeta_n},\tilde{\zeta'_n})\ge \exp(-(\alpha_n|\beta_n)_{\rho_{x_n}})$, therefore from what we got above,
		\begin{equation*}
			\begin{split}
				\rho_x(\zeta,\zeta')< \exp(-(u_n|v_n)_{x_n})+2\epsilon_n
				&=\exp(-(\alpha_n|\beta_n)_{\rho_{x_n}})+2\epsilon_n\\
				&\le\rho_{x_n}(\tilde{\zeta_n},\tilde{\zeta'_n})+2\epsilon_n\\
				&\le\rho_{x_n}(\zeta_n,\zeta'_n)+\delta+2\epsilon_n.
			\end{split}
		\end{equation*}
	\end{proof}
	
	\medskip
	
	\subsection{$(\partial_P X_n,\rho_{x_n})$ is contained in a small neighborhood of the image of $g_n$ for all $n$ sufficiently large}\hfill\\ 
	
	\noindent Let $\xi_n\in \partial_P X_n$ for some $n\ge 1$, we shall show for all $n$ sufficiently large there exists $\zeta\coloneqq \zeta(\xi_n)\in Z$ such that $\rho_n(\xi_n,g_n(\zeta))<3\delta$. 
	
	Choose $w_n\in S_{X_n}(x_n,R)\coloneqq\{ w\in X_n\ |\ d_X(x_n,w)=R\ \}\subset X_n$ 
	such that $\xi_n\in \argmx\ B(x_n,w_n,\ \cdot\ )$
	, i.e. point $w_n\in X_n$ is chosen at a distance $R$ from $x_n$, on a geodesic ray $[x_n,\xi_n)$ joining $x_n$ to $\xi_n$. There exists  $w\in  S_X(x,R)\subset B_X(x,R)$, such that $d_{ X_n}(\FF_n(w),w_n)<4\epsilon_n$.
	
	\noindent Further choose $\xi\in \argmx\ B(x,w,\ \cdot\ )$ and $\zeta=p(\xi)\in F$ as defined above, then $\rho_x(\xi,\zeta)<\epsilon_0$. Note that the choice of $\zeta$ is dependent on $\xi_n$. We will show, this $\zeta$ satisfies $\rho_n(\xi_n,g_n(\zeta))<3\delta$ for all $n$ sufficiently large.
	
	Let $s=d_{ X}(z(\zeta),w)$ (where $z(\zeta)\in S_X(o,R)$ as defined at the beginning of the section) with $d_{X}(z(\zeta),x)=R$ and $\zeta\in\argmx\  B(x,z(\zeta),\ \cdot\ )=\argmx \frac{d\alpha_{\zeta}}{d\rho_x}$. For ease of notation, denote $z(\zeta)$ by $u$.
	
	Cases$\colon$ 
	\begin{enumerate}[(i)]
		\item $s\le R$ $\colon$  Then $2(w|u)_x=2(w|z(\zeta))_x=R+R-s\ge R$, hence using \eqref{epsiloniso}, for all $n$ sufficiently large (since $\epsilon_n\to 0+$) we have
		\begin{equation*}
			\begin{split}
				\rho_{x_n}(\xi_n,g_n(\zeta))^2&\le\exp(-2(w_n|\FF_n(u))_{x_n})\\
				&\le\exp(-2(\FF_n(w)|\FF_n(u))_{x_n})+8\epsilon_n\quad \quad \hbox{(since $t\mapsto e^{-t}$ is $1$-Lipschitz for $t\ge0$)}\\
				&\le \exp(-2(w|u)_x)+12\epsilon_n \quad \quad \quad \quad \quad \quad\hbox{(from \eqref{2 epsilon difference})}\\
				&\le \exp(-R)+12\epsilon_n\\
				&\le \frac{\epsilon^2}{2}+10\epsilon_n\\
				&< \delta^2
			\end{split}
		\end{equation*}
		
		\item $s>R$ $\colon$ Similarly as in the proof of \Cref{distortion on finite set} we can find $\tilde{\zeta},\tilde{\xi}\in \partial_P X$ with $\tilde{\zeta}\in \argmx\frac{d\alpha_{\zeta}}{d\rho_w}$ and $\tilde{\xi}$ is a $\rho_w$-antipode, 
		of $\tilde{\zeta}$, i.e. $\rho_w(\tilde{\zeta},\tilde{\xi})=1$ with $\tilde{\xi}\in \argmx \frac{d\rho_w}{d\alpha_{\zeta}}$. Same calculation as in \eqref{getting epsilon bound} gives,
		\begin{equation*}
			\begin{split}
				1\ge \alpha_{\zeta}(\zeta,\tilde{\zeta})^2
				&\ge \rho_x(\zeta,\tilde{\zeta})^2\exp(d_X(u,x)+d_X(u,w)-d_X(w,x))\\
				&\ge\rho_x(\zeta,\tilde{\zeta})^2\ e^{R}\cdot e^s\cdot e^{-R}\\
				&=\rho_x(\zeta,\tilde{\zeta})^2\cdot e^s\\
				\implies \rho_x(\zeta,\tilde{\zeta})&\le e^{-s/2}<e^{-R/2}<\epsilon.
			\end{split}
		\end{equation*}
		Similarly, we have
		\begin{equation*}
			\begin{split}
				1\ge \rho_w(\xi,\tilde{\xi})^2
				&\ge \rho_x(\xi,\tilde{\xi})^2\exp(d_X(w,x)+d_X(u,w)-d_(u,x))\\
				&\ge \rho_x(\xi,\tilde{\xi})^2\cdot e^{R}\cdot e^s\cdot e^{-R}\\
				\implies \rho_x(\xi,\tilde{\xi})&\le e^{-s/2} <e^{-R/2}<\epsilon.
			\end{split}
		\end{equation*} Therefore, by continuity of $\rho_x$, (see \eqref{equicontinuity used}) $|\rho_x(\xi,\zeta)-\rho_x(\tilde{\xi},\tilde{\zeta})|<\delta$. So we have for all $n$ sufficiently large,
		\begin{equation*}
			\begin{split}
				\rho_{x_n}(\xi_n,g_n(\zeta))&\le \exp(-(w|u)_x+2\epsilon_n)\\
				& \le \rho_x(\tilde{\xi},\tilde{\zeta}) \cdot e^{\ 4\epsilon_n}\\
				&\le (\rho_x(\xi,\zeta)+\delta)\cdot e^{\ 4\epsilon_n}\\
				&\le (\epsilon_0+\delta) \cdot e^{\ 4\epsilon_n}\\
				&\le 3\delta.
			\end{split}
		\end{equation*}
	\end{enumerate}
	Therefore, for all $n$ large $g_n(\partial_P X)$ is a $3\delta$-net of $\partial_P X_n$ with respect to $\rho_{x_n}$.
	\medskip
	
	Thus, we have shown that $g_n\colon (\partial_P X,\rho_x)\to (\partial_P X_n,\rho_{x_n})$ constructed in \Cref{construction of gn} is a $3\delta$-isometry for all $n$ sufficiently large (could have started with $\delta/3$ to eventually get $\delta$ here). 
	This finishes the proof.
	
	\bigskip
	
	\section{GH convergence to maximal Gromov hyperbolic spaces with finite boundary}
	
	This section focuses on developing the technical tools required to prove \Cref{backward finite}. In the later part of the section, we provide the proof of the theorem.
	
	Before proceeding with the discussion, we define a few notions. Let $X$ be a good Gromov product space. Fix a base point $o\in X$. 
	For $R>0$, let us denote the closed ball and sphere of radius $R$ centered at $o$  by 
	\begin{equation*}
		B_X(o,R)\coloneqq\{\ x\in X\ |\ d_X(o,x)\le R\ \},\quad \text{ and }\quad S_X(o,R)\coloneqq \{x \in X \mid d_X(o, x) = R\} \subset X
	\end{equation*} respectively. For a closed subset $V\subset S_X(o,R)$ define the cone of $V$ based at $o$,
	\begin{equation}\label{definition of cone}
		\hbox{Cone}_o(V)\coloneqq\{\ x \in X\ | \hbox{ there exists a geodesic segment } [o,x] \hbox{ such that } [o,x]\cap V\neq \emptyset\ \}\ \subseteq X
	\end{equation}
	and the shadow of $V$ on the Gromov product boundary $\partial_P X$,
	\begin{equation}
		\hbox{Shad}_o(V)\coloneqq\{\ \xi \in \partial X\ | \hbox{ there exists a geodesic ray } [o,\xi) \hbox { such that } [o,\xi)\cap V\neq \emptyset\ \}\ \subseteq \partial_P X.
	\end{equation}
	It is immediate to observe that $V\subset \cone_o(V)$. Also for $\xi\in \shad_o(V)$ there exists a ray $[o,\xi)$ such that $[o,\xi)\cap V\neq \emptyset$, thus we can find a point $x_0$ on the ray such that for any point $x$ after $x_0$ is in $\cone_o(V)$, i.e. for all points $x$ on the ray $[o,\xi)$ with $d_X(o,x)\ge d_X(o,x_0)$ we have $x\in \cone_o(V)$. Also, observe that,
	\footnote{Since we know that $B_X(o, R)\coloneqq\{\ x\in X\ |\ d_X(o,x)\le R\ \}$ is the closed ball, the complement of the open ball is equal to the closure of the set $X\setminus B_X(o, R)$.}
	\begin{equation}\label{immediate property of cone}
		\begin{split}
			&\cone_o(V)\subseteq \overline {X\setminus B_X(o, R)}\ =\{\ x\in X\ |\ d_X(o,x)\ge R\ \},\\ 
			\text{ and }\quad &\cone_o(U)\cap B_X(o,R)=\cone_o(V)\cap S_X(o,R)=V.
		\end{split}
	\end{equation}

	In \Cref{disjoint cone and shadow} and \Cref{path component}, we observe that in a maximal Gromov product space $X$, the connected components of the sphere $S_X(o,R)$ 
	provide information about the structure of $X$ outside the ball $B_X(o, R)$ and the Gromov product boundary $\partial_P X$.

	\begin{lemma}\label{disjoint cone and shadow}
		Let $X$ be a maximal Gromov product space. For $o\in X$ and $R>0$,  let $S_X(o,R)$ be the sphere of radius $R>0$ centered at $o$. Let $U, V\subset S_X(o,R)$ be disjoint connected components in $S_X(o,R)$ (i.e. $U\cap V=\emptyset$). Then we have the following:\\
		\begin{enumerate}[(1)]
			\item For any $p\in \cone_o(U)$ and $q\in \cone_o(V)$ the Gromov product $(p|q)_o\le R$. Consequently, $\cone_o(U)\cap\cone_o(V)=\emptyset$.\\
			
			\item The shadow of $U$ and $V$ on $\partial_P X$ are disjoint, i.e. $\shad_o(U)\cap \shad_o(V)=\emptyset$. In fact $(\xi|\eta)_o\le R$ for all $\xi\in \shad_o(U)$ and $\eta\in \shad_o(V)$.\\
			\item Moreover, if $\CC$ is the  set of connected components of $S_X(o,R)$ then,
			\begin{equation}\label{disjoint union}
				\overline{X\setminus B_X(o,R)} = \bigsqcup_{U\in \CC}\cone_o(U)\footnote{Here `$\bigsqcup$' denotes disjoint union.} \quad \text{ and } \quad \partial_P X=\bigsqcup_{U\in \CC}\shad_o(U).
			\end{equation}
		\end{enumerate} 
	\end{lemma}
	\begin{proof}
		\underline{{\it Proof of (1).}} 
		First, we prove $\cone_o(U)\cap\cone_o(V)=\emptyset$ assuming $(p|q)_o\le R$ for all $p\in \cone_o(U)$ and $q\in \cone_o(V)$. Let us suppose there exists $p\in \cone_o(U)\cap \cone_o(V)\neq\emptyset$. Then from \eqref{immediate property of cone} we know,
		\begin{equation*}
			R\le d_X(o,p)\le (p|p)_o\le R, \text{ which implies } p\in S_X(o,R)\cap \cone_o(U)\cap \cone_o(V).
		\end{equation*}
		Again, from \eqref{immediate property of cone}, we have $p\in U\cap V$, which is impossible, since $U\cap V=\emptyset$. Thus the cones $\cone_o(U)$ and $\cone_o(V)$ must be disjoint.
		
		\medskip
		
		Next we prove $(p|q)_o\le R$ for all $p\in \cone_o(U)$ and $q\in \cone_o(V)$. Let us suppose there exists some $p\in \cone_o(U)$ and $q\in cone_o(V)$ such that $(p|q)_o>R$, and we shall arrive at a contradiction. 
		Let $\lambda\colon [0,1]\to [p,q]$ be a constant speed parametrization of a geodesic segment $[p,q]$ such that $\lambda(0)=p$, $\lambda(1)=q$. 
		We have, 
		\begin{equation*}
			d_X(o,\lambda([0,1]))=d_X(o,[p,q])\ge (p|q)_o>R.
		\end{equation*}
		Thus $\lambda([0,1])\subset {X\setminus B_X(o,R)}$. Consequently, $p\in \cone_o(U)\setminus U$ and $q\in \cone_o(V)\setminus V$. Moreover, there exists 
		geodesic segments $[o,p]$ and $[o,q]$ with $\{p'\}=[o,p]\cap U$ and $\{q'\}= [o,q]\cap V$ respectively. Let $\alpha_1\colon [0,1]\to [o,p']$ and $\alpha_2\colon [0,1]\to [o,q']$ be constant speed parametrizations of the geodesic segments $[o,p']$ and $[o,q']$, respectively. Consider the concatenation path $\alpha \coloneq \overline{\alpha_1}* \alpha_2\colon [0,1]\to X$, such that
		\begin{equation*}
			\alpha(t)\coloneqq 
			\begin{cases}
				\alpha_1(1-2t) \quad &\text{if } \, t\in\big[1,\frac{1}{2}\big]\\
				\alpha_2(2t-1) \quad &\text{if } \, t\in\big[\frac{1}{2},1\big]\\
			\end{cases},
		\end{equation*} joining $p'$ to $q'$. 
		
		Now, $X$ is a maximal Gromov product space; thus, from \Cref{geodesic bicombing subsection} we know it admits a convex geodesic bi-combing $\Gamma\colon X\times X\times [0,1]\to X$ (note that $\Gamma$ is continuous). Let us define 
		\begin{equation*}
			\begin{split}
				H\colon [0,1]\times [0,1]&\to X\\
				(s,t)&\mapsto \Gamma(\alpha(s),\lambda(s),t).
			\end{split}
		\end{equation*}
		Let $s\in [0,1]$, then the function $f_s\colon t\in[0,1]\mapsto d_X(o,H(s,t))\in[0,\infty)$ is convex. For $s\in (0,1)$ we have,
		\begin{equation}
			\begin{split}
				& f_s(0)=d_X(o,H(s,0))=d_X(o,\alpha(s))<R\\
				\hbox{ and }& f_s(1)=d_X(o,H(s,1))=d_X(o,\lambda(s))>R.
			\end{split}
		\end{equation}
		Let $A_s\coloneqq\{\ t\in [0,1]\ |\ f_s(t)=R\ \}\neq \emptyset$. If $t_1,t_2\in A_s$, with $t_1\le t_2$ then $f_s$ being convex; and $f_s(0)<R=f_s(t_2)$ implies $t_1=t_2$ (otherwise $f_s(t_1)<R$, a contradiction). Thus $A_s$ is singleton, i.e. $A_s=\{\ \hat{t}(s)\ \}$ for some $\hat{t}(s)\in (0,1)$.
		
		\medskip
		
		\underline{Claim:} The function $\hat{t}\colon (0,1)\to (0,1) (s\mapsto \hat{t}(s))$ is continuous.

		Proof of the claim: Let us consider a sequence $s_n\to s$ in $(0,1)$. Passing to a sub-sequence $\hat{t}(s_n)\to t$ for some $t\in [0,1]$. By continuity of $\Gamma$ we have 
		\begin{equation*}
			R=f_{s_n}(\hat{t}(s_n))\to f_s(t)\quad \text{ as $n\to \infty$ along the sub-sequence}.
		\end{equation*}
		As a consequence $f_s(t)=R$ and $\hat{t}(s)=t\in(0,1)$. Therefore, $\hat{t}(s)=t$ is the unique limit point of the sequence $\{\hat{t}(s_n)\}_{n\ge 1}$. Hence $\hat{t}(s_n)\to \hat{t}(s)$ as $n\to \infty$, so $\hat{t}$ is continuous.
		
		Observe that $H(0,\cdot)=\Gamma_{p'p}$ is a constant speed parametrization of a geodesic from $p'$ to $p$, and 
		\begin{equation*}
			d_X(o,p)=d_X(o,p')+d_X(p',p).
		\end{equation*} Thus, the concatenation $\alpha_1*\Gamma_{p'p}$ is a constant speed parametrization of a geodesic segment joining $o$ to $p$. Consequently, $f_0(t)=d_X(o,H(0,t))=d_X(o,\Gamma(p',p,t))>R$ for all $t\in (0,1]$. Moreover $f_0(t)=R$ if and only if $t=0$. Hence we have $\hat{t}(s)\to 0$ as $s\to 0+$. Arguing similarly we also have $\hat{t}(s)\to 1$ as $s\to 1-$. Therefore, $\hat{t}$ extends to a continuous function $\hat{t}\colon [0,1]\to [0,1]$ with $\hat{t}(0)=0$ and $\hat{t}(1)=1$.
		
		Define the curve in $S_X(o,R)$
		\begin{equation*}\label{theta curve}
			\begin{split}
				\theta\colon [0,1]\to &S_X(o,R)\\
				s\mapsto &H(s,\hat{t}(s))=\Gamma(\alpha(s),\lambda(s),\hat{t}(s)).
			\end{split}
		\end{equation*}
		Then $\theta(0)=\alpha(0)=p'\in U$ and $\theta(1)=\alpha(1)=q'\in V$. However, $\theta([0,1])$ is connected, so $\theta(0)=p'\in U$ implies $\theta([0,1])\subset U$ (as $U$ connected component of $S_X(o,R)$) and $\theta(1)=q'\in U$. So we have $q'\in U\cap V$ which is a contradiction to the hypothesis $U\cap V=\emptyset$.
		Thus we must have $(p|q)_o\le R$ for all $p\in \cone_o(U)$ and $q\in\cone_o(V)$.
		
		\medskip
		
		\underline{{\it Proof of (2).}} This is a consequence of statement (1). Let $\xi\in \shad_o(U)$ and $\eta\in \shad_o(V)$. Consider two geodesic rays $\gamma_1\colon[0,\infty)\to X$ and $\gamma_2\colon[0,\infty)\to X$ joining $o$ to $\xi$ and $\eta$ on $\partial_P X$, respectively. We know from the properties of Gromov product spaces that
		\begin{equation*}
			\lim_{t\to\infty} (\gamma_1(t)|\gamma_2(t))_o=(\xi|\eta)_o.
		\end{equation*} Now, we have for all $t$ large $\gamma_1(t)\in \cone_o(U)$ and $ \gamma_2(t)\in \cone_o(V)$. Thus, from above we have for all $t$ large $(\gamma_1(t)|\gamma_2(t))_o\le R$ and $(\xi|\eta)_o\le R$. This also implies $\shad_o(U)\cap\shad_o(V)=\emptyset$.
		
		\medskip
		
		\underline{{\it Proof of (3).}} Let $\CC$ be the collection of all connected components of $S_X(o,R)$. From statement (1) and (2) we have
		\begin{equation*}
			\bigsqcup_{U\in \CC}\cone_o(U)\subseteq \overline{X\setminus B_X(o,R)} \quad \text{ and } \quad \bigsqcup_{U\in \CC}\shad_o(U)\subseteq \partial_P X.
		\end{equation*}
		For $x\in \overline{X\setminus B_X(o,R)}$ any geodesic segment $[o,x]$ must intersect $S_X(o,R)$, and hence it intersect some connected component $V$ of $S_X(o,R)$. Thus $x\in \cone_o(V)$. Similarly, for $\xi \in \partial_P X$ any geodesic ray $[o,\xi)$ must intersect some connected component $V$ of $S_X(o,R)$, and therefore $\xi\in \shad_o(V)$. Finally, we can conclude
		\begin{equation*}
			\overline{X\setminus B_X(o,R)} = \bigsqcup_{U\in \CC}\cone_o(U) \quad \text{ and } \quad \partial_P X=\bigsqcup_{U\in \CC}\shad_o(U).	
		\end{equation*}
	\end{proof}
	
	We make the following remark:
	\begin{rmk}
		Let $X$ be a maximal Gromov product space. Let $U$ be a connected component of the sphere $S_X(o, R)$, as in \Cref{disjoint cone and shadow}. Suppose a point $x \in \cone_o(U)$, i.e., there exists a geodesic segment $[o, x]$ which intersects $U$ (see \eqref{definition of cone}). From \Cref{disjoint cone and shadow}(1), we can argue that every geodesic segment $[o, x]$ joining $o$ to $x$ must intersect $U$. Similarly, from \Cref{disjoint cone and shadow}(2), for $\xi \in \shad_o(U)$, every geodesic ray $[o,\xi)$ intersects $U$. 
	\end{rmk}
	\begin{lemma}\label{path component}
		Let $X$ be a maximal Gromov product space. For $o\in X$ and $R>0$,  let $S_X(o,R)$ denote the sphere of radius $R>0$ centered at $o$. Let $U, V\subset S_X(o,R)$ be disjoint connected components of $S_X(o,R)$. Then for any
		$p\in\cone_o(U)\setminus U$ and $q\in\cone_o(V)\setminus V$,the points $p$ and $q$ lie in the different path components in $X\setminus B_X(o,R)$, that is, for $p\in \cone_o(U)\setminus U$ and $q\in \cone_o(V)\setminus V$ there does not exist any path $\lambda\colon [0,1]\to X\setminus B_X(o,R)$ with $\lambda(0)=p$ and $\lambda(1)=q$.
	\end{lemma}
	\begin{proof}
		Let us suppose there exist $p\in \cone_o(U)\setminus U$ and $q\in \cone_o(V)\setminus V$ such that there exists a path \begin{equation*}
			\lambda\colon [0,1]\to X\setminus B_X(o,R)
		\end{equation*} with $\lambda(0)=p$ and $\lambda(1)=q$. Thus $d_X(o,\lambda([0,1]))>R$. Let us consider geodesic segments $[o,p]$ and $[o,q]$ joining $o$ to $p$ and $q$ respectively. Consider, $p'\in [o,p]\cap U$ and $q'\in [o,q]\cap V$. Same as in the proof of \Cref{disjoint cone and shadow}(1), using convex geodesic bi-combing on $X$ we can construct a path (see \eqref{theta curve}) $\theta\colon [0,1]\to S_X(o,R)$ with $\theta(0)=p'$ and $\theta(1)=q'$. This is not possible since $p'\in U$ and $q'\in V$, which are disjoint connected components of $S_X(o,R)$.
	\end{proof}
	\begin{prop}\label{gromov product close}
		Let $X$ be a maximal Gromov product space. For $o\in X$ and $R>0$, let $S_X(o,R)$ denote the sphere of radius $R>0$ centered at $o$. Let $U, V\subset S_X(o,R)$ be disjoint connected components of $S_X(o,R)$. For $\xi\in \shad_o(U)$ and $\eta\in \shad_o(V)$ we have,
		\begin{equation}\label{gromov product close 1}
			\big|(\xi|\eta)_o-(x|y)_o\big|\le \hbox{diam}(U)+\hbox{diam}(V)
		\end{equation} for all $x\in U$ and $y\in V$. Moreover, for a closed subset $U\subseteq S_X(o,R)$, if $\xi,\eta\in \shad_o(U)$ we have 
		\begin{equation}\label{gromov product close 2}
			(\xi|\eta)_o\ge R-\frac{\hbox{diam}(U)}{2}.
		\end{equation}
	\end{prop}
	\begin{proof}
		First, we prove \eqref{gromov product close 1}. Let $\xi\in \shad_o(U)$ and $\eta\in \shad_o(V)$. We consider two geodesic rays $\gamma_1\colon[0,\infty)\to X$ and $\gamma_2\colon[0,\infty)\to X$ running from $o$ to $\xi$ and $\eta$ on $\partial_P X$, respectively, with $\gamma_1(R)\in U$ and $\gamma_2(R)\in V$. Now for all $t$ large enough (say $t>2R$), we have, $\gamma_1(t)\in \cone_o(U)\setminus U$ and $\gamma_2(t)\in \cone_o(V)\setminus V$. $X$ being a geodesic metric space, there exists a geodesic segment, say $\gamma$ joining $\gamma_1(t)$ to $\gamma_2(t)$. By \Cref{path component} we know $\gamma$ must intersect $B_X(o,R)$.
		
		Define $t_0 \coloneq \inf \{t \ge 0 \mid d_X(o, \gamma(t)) \leq R\}$. Note that $d_X(o, \gamma(t)) > R$ for all $t < t_0$, and 
		$d_X(o, \gamma(t_0)) = R$, i.e. $\gamma(t_0)\in S_X(o,R)$.  
		
		For all $t < t_0$, we have $\gamma(t) \in \cone_o(U)$. Otherwise, (referring to \eqref{disjoint union}), for some $t<t_0$, we would have $\gamma(t) \in \cone_o(V) \setminus V$, where $V$ is some other connected component of the sphere $S_X(o,R)$ with $V \neq U$. Then the path $\gamma([0,t])$ joins $\gamma(0)\in \cone_o(U)\setminus U$ to $\gamma(t)\in \cone_o(V)\setminus V$. By \Cref{path component}, there would exist some $t_1$ with $0 < t_1 < t < t_0$ such that $d_X(o, \gamma(t_1)) \leq R$, contradicting the definition of $t_0$. Thus, $\gamma(t_0)$ is an accumulation point of $\cone_o(U)$ and $\gamma(t_0)\in S_X(o,R)$.
		
		Suppose $p$ is an accumulation point of $\cone_o(U)$ and $p\in S_X(o,R)$. Then there exists a sequence $p_n\in \cone_o(U)$ such that $p_n\to p$ as $n\to \infty$. Let $x_n\in [o,p_n]\cap U$ for some geodesic segment joining $o$ to $p_n$.
		\begin{equation*}
			\begin{split}
				d_X(x_n,p)&\le d_X(x_n,p_n)+d_X(p_n,p)\\
				&=d_X(o,p_n)-d_X(o,x_n)+ d_X(p_n,p)\\
				&=d_X(o,p_n)-R+d_X(p_n,p)\xrightarrow{n\to \infty} d_X(o,p)-R+0=0.
			\end{split}
		\end{equation*}
		Thus $x_n\to p$ as $n\to \infty$. Since $U$ is a connected component of $S$, it is closed, and we have $p\in U$. From this above discussion, we get $\gamma(t_0)\in U$. 
		
		Define $t_1\coloneqq\sup\{t\ge0\ |\ d_X(o,\gamma(t))\le R\}$, then similarly we can argue that $\gamma(t_1)\in V$. 
		
		Note that 
		\begin{equation*}
			\begin{split}
				d_X(\gamma_1(t),\gamma_2(t))&=d_X(\gamma_1(t),\gamma(t_0))+d_X(\gamma(t_0),\gamma(t_1))+d_X(\gamma(t_1),\gamma_2(t)),\\
				d_X(o,\gamma_1(t))&=d_X(o,\gamma_1(R))+d_X(\gamma_1(R),\gamma_1(t))=R+d_X(\gamma_1(R),\gamma_1(t)),\\	d_X(o,\gamma_2(t))&=d_X(o,\gamma_2(R))+d_X(\gamma_2(R),\gamma_2(t))=R+d_X(\gamma_2(R),\gamma_2(t)),
			\end{split}
		\end{equation*}which gives,
		\begin{equation}\label{gromov product equation}
			\begin{split}
				2(\gamma_1(t)|\gamma_2(t))_o=&2R-d_X(\gamma(t_0),\gamma(t_1))\\
				&+d_X(\gamma_1(t),\gamma(t_0))-d_X(\gamma_1(R),\gamma_1(t))\\
				&+d_X(\gamma(t_1),\gamma_2(t))-d_X(\gamma_2(R),\gamma_2(t)).
			\end{split}
		\end{equation}
		For $x\in U$ and $y\in V$ we have $2(x|y)_o=d_X(x,o)+d_X(y,o)-d_X(x,y)=2R-d_X(x,y).$ Thus from \eqref{gromov product equation} and application of the triangle inequality we get,
		\begin{equation}\label{gromov product approximation}
			\begin{split}
				2\big|(\gamma_1(t)|\gamma_2(t))_o-(x|y)_o\big|\le&\  \big|d_X(x,y)-d_X(\gamma(t_0),\gamma(t_1))\big|\\
				&+\big|d_X(\gamma_1(t),\gamma(t_0))-d_X(\gamma_1(R),\gamma_1(t))\big|\\
				&+\big|d_X(\gamma(t_1),\gamma_2(t))-d_X(\gamma_2(R),\gamma_2(t))\big|\\
				\le&\ d_X(x,\gamma(t_0))+d_X(\gamma(t_1),y)+d_X(\gamma(t_0),\gamma_1(R))+d_X(\gamma(t_1),\gamma_2(R))\\
				\le&\ 2 (\hbox{diam}(U)+\hbox{diam}(V))
			\end{split}
		\end{equation}
		We know from property of Gromov product space $$\lim_{t\to\infty} (\gamma_1(t)|\gamma_2(t))_o=(\xi|\eta)_o.$$ Therefore, from \eqref{gromov product approximation} we can conclude the desired claim \eqref{gromov product close 1} (by taking limit $t\to \infty$ on the left hand side).
		
		For proving \eqref{gromov product close 2}, consider geodesic rays $\gamma_1$ and $\gamma_2$ joining $o$ to $\xi$ and $\eta$, respectively, as above. Then $(\gamma_1(t)|\gamma_2(t))_o\uparrow (\xi|\eta)_o$ as $t\to \infty$. Note $\gamma_1(R),\gamma_2(R)\in U$. Thus,
		\begin{equation*}
			(\xi|\eta)_o\ge(\gamma_1(R)|\gamma_2(R))_o= \frac{1}{2}\bigg(d_X(o,\gamma_1(R))+d_x(o,\gamma_2(R))-d_X(\gamma_1(R),\gamma_2(R))\bigg)\ge R-\frac{\hbox{diam}(U)}{2}.
		\end{equation*}
	\end{proof}
	
	Now we are well equipped to give the proof of \Cref{backward finite}.
	
	\begin{proof}[{\bf Proof of \Cref{backward finite}}]
		Since $(X_n,x_n)$ is a maximal Gromov product space, it is isometric to $\MM(\partial_P X_n,\rho_{x_n})$, for each $n$, via the visual embedding (see \Cref{structure theorem}). Similarly, $(X,x)$ is isometric to $\MM(\partial_P X,\rho_{x})$. Define the sequence of antipodal spaces $(Z_n,\rho_n)=(\partial_P X_n,\rho_{x_n})$ for all $n$ and $(Z,\rho_0)=(\partial_P X,\rho_x)$, where $Z$ is a finite set of cardinality $m<\infty$.
		
		By the given hypothesis the Moebius spaces $(\MM(Z_n),\rho_n)\xrightarrow{GH\ conv.} (\MM(Z),\rho_0)$.
		We will demonstrate that for any given $\delta > 0$, it is possible to construct $\delta$-isometries $f_n\colon (Z_n, \rho_n) \to (Z, \rho_0)$ for all sufficiently large $n$. Thus, by \Cref{AIconv}, this will be sufficient to conclude $(Z_n,\rho_n)\xrightarrow{AI\ conv.} (Z,\rho_0)$, i.e. $(\partial_P X_n,\rho_{x_n})\xrightarrow{AI conv.}(\partial_P X,\rho_x)$.
		
		Since $(\MM(Z_n),\rho_n)\xrightarrow{GH\ conv.}(\MM(Z),\rho_0)$, for every $R>0$, there exists $\epsilon_n$-isometries $$\FF_n\colon B_{\MM(Z)}(\rho_0,R)\to B_{\MM(Z_n)}(\rho_n,R)$$ such that $\epsilon_n \to 0+$ and $\FF_n(\rho_0)=\rho_n$. 
		
		The antipodal space $(Z,\rho_0)$ is of finite cardinality, let us label $Z=\{1,2,\cdots,m\}$. By \cite[Theorem 1.1]{biswas2024polyhedral}, the maximal Gromov hyperbolic spaces $\MM(Z,\rho_0)$ is isometric to a polyhedral complex such that outside a large enough ball, it is the union of $m$ geodesic rays, each ray going to one among the $m$ points on the Gromov boundary. So, for all $R>0$ large enough, the sphere of radius $R$ centered at $\rho_0$ in $\MM(Z)$,
		\begin{equation*}
			S_0\coloneqq S_{\MM(Z)}(\rho_0,R)=\{\alpha_1,\alpha_2,\cdots,\alpha_m\},
		\end{equation*} (has cardinality $m$) where each point $\alpha_i$ is on the geodesic ray going to point $i\in \partial \MM(Z)=Z$. Choose and fix such an $R>0$ with  
		\begin{equation}\label{choice of R}
			R\ge\log\frac{2}{\delta}.
		\end{equation} The Gromov product of two distinct points on the Gromov boundary $i,j\in \partial \MM(Z)=Z$ is attained at $\alpha_i,\alpha_j$, that is $(i|j)_{\rho_0}=(\alpha_i|\alpha_j)_{\rho_0}$ (cf. \cite[(13)]{biswas2024polyhedral}). Thus $\rho_0(i,j)=\exp(-(\alpha_i|\alpha_j)_{\rho_0})$.
		
		Consider the spheres of radius $R$ centered at $\rho_n$, $S_n\coloneqq S_{\MM(Z_n)}(\rho_n,R)\subseteq \MM(Z_n)$. There is a universal constant $k_0>0$ such that using $\epsilon_n$-isometries $\FF_n$ we can construct $(k_0\cdot\epsilon_n)$-isometries $\hat{\FF_n}:S_0\to S_n$. Now $S_0$ being finite for all $n$ large $S_n$ is a finite disjoint union of $m$ closed balls $B_n^i$ in $S_n$\footnote{ These are balls in the metric space $S_n$ equipped the subspace metric }, centered at $\hat{\FF_n}(\alpha_i)$ of radius $(k_0\cdot\epsilon_n)$, 
		\begin{equation*}
			S_n=\bigsqcup_{i=1}^{m}B_n^i.
		\end{equation*} and diameter $\hbox{diam}(B_n^i)\le 2k_0\cdot\epsilon_n$. Thus for all $n$ large enough, from \Cref{disjoint cone and shadow} and \eqref{disjoint union} we have,
		\begin{equation}\label{shadow decomposition}
			\partial_P\MM(Z_n)=Z_n=\bigsqcup_{U\in \CC_n}\shad_{\rho_n}(U)=\bigsqcup_{i=1}^m\Biggr(\bigsqcup_{\substack{U\in \CC_n\\ U\subset B_n^i}} \shad_{\rho_n}(U)\Biggr)=\bigsqcup_{i=1}^m \shad_{\rho_n}(B_n^i).
		\end{equation} where $\CC_n$ is the set of connected components in $S_n$.
		
		For $n$ sufficiently large, we define $g_n:(Z_n,\rho_n)\to (Z,\rho_0)$ as follows,
		\begin{equation*}
			g_n(\xi)=i,\quad \hbox{for} \ \xi\in \shad_{\rho_n}(B_n^i).
		\end{equation*}Clearly $g_n$ is surjective. We will show $g_n$ is a $\delta$-isometry for all $n$ large, i.e. we just have to show distortion of $g_n$, $\Dis(g_n)<\delta$. 
		
		\medskip
		
		Suppose $\xi \in \shad_{\rho_n}(B_n^i)$ and $\eta \in \shad_{\rho_n}(B_n^j)$, for $i \neq j$. Then there exist connected components $U$ and $V$ of $S_n$ such that $U \subseteq B_n^i$ and $V \subseteq B_n^j$, with $\xi \in \shad_{\rho_n}(U)$ and $\eta \in \shad_{\rho_n}(V)$ (from \eqref{shadow decomposition}). Also $g_n(\xi)=i$ and $g_n(\eta)=j$.
		For $\beta\in U$ and $\beta'\in V$, by $\eqref{gromov product close 1}$ in \Cref{gromov product close} we have
		\begin{equation*}
			\big|(\xi|\eta)_{\rho_n}-(\beta|\beta')_{\rho_n}\big|\le\hbox{diam}(U)+\hbox{diam}(V)\le \hbox{diam}(B_n^i)+\hbox{diam}(B_n^j)\le 4k_0\cdot\epsilon_n.
		\end{equation*} Also we have,
		\begin{equation*}
			\big|(\beta|\beta')_{\rho_n}-(\hat{\FF_n}(\alpha_i)\big|\hat{\FF_n}(\alpha_j))_{\rho_n}|\le \hbox{diam}(B_n^i)+\hbox{diam}(B_n^j)\le 4k_0\cdot\epsilon_n.
		\end{equation*}
		Since $\hat{\FF_n}$'s are $(k_0\cdot\epsilon_n)$-isometries
		\begin{equation*}
			\big|(i|j)_{\rho_0}-(\hat{\FF_n}(\alpha_i)|\hat{\FF_n}(\alpha_j))_{\rho_n}\big|=\big|(\alpha_i|\alpha_j)_{\rho_0}-(\hat{\FF_n}(\alpha_i)|\hat{\FF_n}(\alpha_j))_{\rho_n}\big|\le k_0\cdot \epsilon_n.
		\end{equation*} From the above three inequalities, by straightforward application of the triangle inequality, we have,
		\begin{equation*}
			\big|(i|j)_{\rho_n}-(\xi|\eta)_{\rho_n}\big|\le k_1\cdot \epsilon_n
		\end{equation*} (for some universal constant $k_1>0$). Now, for $t\ge 0$ the function $(t\mapsto \exp{(-t)})$ is $1$-Lipschitz, hence we get,
		\begin{equation*}
			\begin{split}
				\big|\rho_0(g_n(\xi),g_n(\eta))-\rho_n(\xi,\eta)\big|&=\big|\rho_0(g_n(i),g_n(j))-\rho_n(\xi,\eta)\big|\\
				&=\big|\exp(-(i|j)_{\rho_0})-\exp(-(\xi|\eta)_{\rho_n})\big|\\
				&\le k_1\cdot \epsilon_n\\
				&<\delta
			\end{split}
		\end{equation*} for all $n$ large (as $\epsilon_n\to 0+$).
		
		\medskip
		
		Let $\xi,\eta\in \shad_{\rho_n}(B_n^i)$ then by \eqref{gromov product close 2} in \Cref{gromov product close} we know
		$2(\xi|\eta)_{\rho_n}\ge 2R-\hbox{diam}(B_n^i)$. Thus by the choice of $R$ in \eqref{choice of R},
		\begin{equation*}
			\begin{split}
				\big|\rho_0(g_n(\xi),g_n(\eta))-\rho_n(\xi,\eta)\big|&=\rho_n(\xi,\eta)=\exp{(-(\xi|\eta)_{\rho_n})}\\
				&\le \exp{(-R+k_0\cdot\epsilon_n)}\\
				&\le \frac{\delta}{2}\cdot (e^{k_0\cdot\epsilon_n})<\delta
			\end{split}
		\end{equation*}for all $n$ large (since $\epsilon_n\to 0+$). 
		
		\medskip
		
		Finally, we have for all $n$ sufficiently large $g_n:(Z_n,\rho_n)\to(Z,\rho_0)$, constructed above is a $\delta$-isometry. This finishes the proof.
	\end{proof}
	\bigskip
	
	\section{Some Applications}
	\subsection{Moebius equivalence and AI-convergence}\hfill\\
	
	\noindent We will show that under certain conditions, the Moebius equivalence of antipodal spaces is preserved under AI-convergence.
	\begin{subprop}\label{Mobius equivalence prop}
		Let $\{(Z_{n},\rho_{n})\}_{n\ge 1}$ and $\{(W_{n},\beta_{n})\}_{n\ge 1}$ be two sequences of antipodal spaces such that 
		\begin{equation*}
			(Z_{n},\rho_{n})\xrightarrow{AI\ conv.}(Z,\rho_0)\quad \text{and} \quad (W_{n},\beta_{n})\xrightarrow{AI\ conv.}(W,\beta_0),
		\end{equation*} 
		where $(Z,\rho_0)$ and $(W,\beta_0)$ are antipodal spaces as well. Suppose there exists  Moebius homeomorphism $\varphi_n\colon(W_{n},\beta_{n})\to(Z_{n},\rho_{n})$, with
		\begin{equation}\label{boundedness condition}
			d_{\MM(W_{n})}({{\varphi_n}^*}\rho_{n},\beta_{n})=\left\Vert \log \frac{d ({\varphi_n}^*\rho_{n})}{d\beta_{n}}\right\Vert_\infty<c<\infty,
		\end{equation} for every $n$ (where ${\varphi_n}^*$ is the pull-back, see \Cref{induced isometry}). Then $(Z,\rho_0)$ is Moebius equivalent to $(W,\beta_0)$.
	\end{subprop}
	\begin{proof}
		By \Cref{forward} we have,
		\[\begin{tikzcd}
			{(\MM(Z_{n}),\rho_{n})} && {(\MM(Z),\rho_0)} \\
			{} \\
			{(\MM(W_{n}),\beta_{n})} && {(\MM(W),\beta_0).}
			\arrow["GH\ conv.", from=1-1, to=1-3]
			\arrow["{{{\varphi_n}^*}}"', from=1-1, to=3-1]
			\arrow["\Phi", dashed, from=1-3, to=3-3]
			\arrow["GH\ conv.", from=3-1, to=3-3]
		\end{tikzcd}\]
		Now fix a countable dense subset  $\DD\subset\MM(Z)$, with $\rho_0\in \DD$. 
		
		\noindent Fix some $R>2c$. Define $\DD_R\coloneqq \DD\cap B_{\MM(Z)}(\rho_0,R-c)$, and it is dense in $B_{\MM(Z)}(\rho_0,R-c)$. Note from above, there exists $\epsilon_n$-isometries, 
		\begin{equation*}
			\begin{split}
				&\FF_n\colon B_{\MM(Z)}(\rho_0,R-c)\to B_{\MM(Z_n)}(\rho_n,R-c) \quad \text{and} \\
				&\GG_n\colon B_{\MM(W_n)}(\beta_n,R)\to B_{\MM(W)}(\beta_0,R)
			\end{split}
		\end{equation*}
		with $\FF_n(\rho_0)=\rho_n$, $\ \GG_n(\beta_n)=\beta_0$ and $\epsilon_n\to 0+$. Observe, for all $n$ we have (from \eqref{boundedness condition})
		\begin{equation*}
			{\varphi_n}^*(B_{\MM(Z_n)}(\rho_n,R-c))\subset B_{\MM(W_n)}(\beta_n,R)\subset {\varphi_n}^*(B_{\MM(Z_n)}(\rho_n,R+c)).
		\end{equation*} 
		Therefore, the map
		\begin{equation*}
			\HH_n\coloneqq \GG_n\circ{\varphi_n}^*\circ \FF_n \colon B_{\MM(Z)}(\rho_0,R-c)\to B_{\MM(W)}(\beta_0,R)
		\end{equation*} has distortion, $\text{Dis}(\HH_n)\le2\epsilon_n$ and the image is $2(c+2\epsilon_n)$-dense in $B_{\MM(W)}(\beta_0,R)$. Employing the same Cantor's diagonal procedure, with dense subset $\DD_R$, as in in the proof of \cite[Theorem 7.3.30]{burago2022course}, we get an isometric embedding $\HH_R\colon\DD_R\to B_{\MM(W)}(\beta_0,R)$, which we can extend to an isometric embedding 
		\begin{equation*}
			\HH_R\colon B_{\MM(Z)}(\rho_0,R-c)\to B_{\MM(W)}(\beta_0,R)
		\end{equation*}
		such that the image is $2c$-dense in  $B_{\MM(W)}(\beta_0,R)$. Moreover, $d_{\MM(W)}(\HH_R(\rho_0),\beta_0)\le c$. 
		
		Consider the family of such isometric embeddings $\HH_{R_m}$, where $R_m=(m+2)\cdot c$ and $m\ge 1$ is an integer. Again, employ Cantor's diagonal procedure, now with the dense subset $\DD$ in $\MM(Z)$, to get an embedding $\Phi\colon\DD\to \MM(W)$. Thus we can extend this to an isometric embedding $\Phi\colon\MM(Z)\to\MM(W)$, with image being $2c$-dense in $\MM(W)$.
		
		The isometric embedding $\Phi$ extends to a Moebius embedding $\phi\colon Z\to W$ (see \Cref{Moebius homeomorphism}), Since the Image$(\Phi)$ is $2c$-dense in $\MM(W)$, we have, $\phi$ is surjective, hence a Moebius homeomorphism. This finishes the proof. 
	\end{proof}
	
	\medskip
	
	\subsection{$\MM(Z)$ as a Gromov-Hausdorff limit of polyhedral complexes}\label{PolyComplex subsection}\hfill\\
	
	\noindent Given an infinite antipodal space say $(Z,\rho_0)$ we have finite antipodal spaces $\{(Z_n,\rho_n)\}_{n\ge 1}$ such that $(Z_n,\rho_n)\xrightarrow{AI\ conv.}(Z,\rho_0)$. For every $n$ we can choose finite $(1/n)$-net $Z_n$ of $(Z,\rho_0)$ such that every point in $Z_n$ has a $\rho_0$-antipode in it. 
	Then, the inclusion map 
	\begin{equation*}
		(Z_n,\rho_n)\xhookrightarrow{\iota_n}(Z,\rho_0) \quad \text{with}\quad \rho_n\coloneqq \rho_0|_{Z_n\times Z_n}
	\end{equation*}is a $(1/n)$-isometry for every $n$. Therefore, $(Z_n,\rho_n)\xrightarrow{AI\ conv.}(Z,\rho_0)$. We know that if $(Z_n,\rho_n)$ is finite (with say $m$ many points) then $\MM(Z_n,\rho_n)$ is a polyhedral complex (from \cite[Theorem 1.1]{biswas2024polyhedral}) embedded in $(\R^{m},\|\cdot\|_\infty)$. Moreover for $R>0$ sufficiently large, $\MM(Z_n)\setminus B_{\MM(Z_n)}(\rho_n,R)$ is union of $m$ many rays  (isometric to interval $(0,\infty)$). Then as a consequence of \Cref{forward} we have that $(Z_n,\rho_n)\xrightarrow{AI\ conv.}(Z,\rho_0)$ implies 
	\begin{equation*}
		(\MM(Z_n),\rho_n)\xrightarrow{GH\ conv.}(\MM(Z),\rho_0).
	\end{equation*}
	Therefore, we have the following proposition:
	\begin{proof}[{\bf Proof of \Cref{PolyComplex}}]
		Given a maximal Gromov product space $X$, it is isometric to a Moebius space $\MM(Z)$, where we can take $Z=\partial X$ equipped with a visual antipodal function $\rho_0=\rho_x$, for some $x\in X$ (see \Cref{structure theorem}). Then we can construct $(Z_n,\rho_n)$ as above and consider $(X_n,x_n)\coloneqq(\MM(Z_n),\rho_n)$ and have 
		\begin{equation*}
			(X_n,x_n)\cong (\MM(Z_n),\rho_n)\xrightarrow{GH\ conv.} (\MM(Z),\rho_0)\cong(X,x).
		\end{equation*}
		From \cite[Theorem 1.1]{biswas2024polyhedral} each $X_n$ is isometric to a polyhedral complex embedded in $(\R^{m(n)},\|\cdot\|_{\infty})$, where $m(n)=\#Z_n$ 
		(as discussed above).
	\end{proof}
	\medskip
	
	\subsection{Gromov-Hausdorff convergence of CAT$(-1)$ spaces}\hfill\\
	
	\noindent It is a well-known fact due to Bourdon \cite{bourdon1995structure} that- if $X$ is a CAT$(-1)$ space and $x \in X$, then the visual antipodal space $(\partial X, \rho_x)$ is a metric space. In this subsection, we shall establish that if a sequence of good CAT$(-1)$ spaces (i.e. proper geodesically complete) $\{(X_n,x_n)\}_{n\ge 1}$ pointed GH- converges to another CAT$(-1)$ space $(X,x)$ then the visual metric spaces $(\partial X_n,\rho_{x_n})\xrightarrow{d_{GH}}(\partial X,\rho_x)$ in the usual GH-distance (i.e. we will prove \Cref{CAT(-1) main}).
	
	\begin{proof}[{\bf Proof of \Cref{CAT(-1) main}}]
		Observe that good CAT$(-1)$ spaces $X_n$ are good Gromov hyperbolic spaces. Hence, they are Gromov product spaces with Gromov product boundaries being their Gromov boundaries $\partial X_n$. The visual antipodal spaces $(\partial X_n,\rho_n)$ are metric spaces. Thus, $(\partial X_n,\rho_n)$ are equicontinuous. Therefore, the theorem follows as a direct consequence of \Cref{backward}, since AI-convergence is the same as the convergence in GH-distance for compact metric spaces (see \Cref{AIconv}).
	\end{proof}
	
	\medskip
	
	\subsection{Application to CAT$(-1)$ fillable spaces and CAT$(-1)$ spaces}\label{application to CAT(-1) fillable}\hfill\\
	
	\noindent Recall, we say that, an antipodal space is $(Z,\rho)$ is CAT$(-1)$ {\it  fillable} if there exists a proper geodesically complete CAT$(-1)$ space $X$ with  $\varphi\colon (\partial X, \rho_x)\to (Z,\rho)$ a Moebius homeomorphism where $\partial X$ is the Gromov boundary and $\rho_x$ is the visual metric on $\partial X$ for some $x\in X$. For detailed exposition in CAT$(-1)$ spaces, Gromov hyperbolic spaces, Gromov boundary and visual (quasi-)metric, we refer to \cite[Chapter II.1, II.8]{bridson1999metric}, \cite[Chapter 1-3]{buyalo2007elements}, \cite{vaisala2005gromov}.

	In this subsection, we demonstrate that a sequence of CAT$(-1)$ fillable antipodal spaces is AI-convergent, if the corresponding sequence of Moebius spaces is pointed GH-convergent (see \Cref{CAT(-1) fillable}). Subsequently, we will establish that the class of CAT$(-1)$ fillable antipodal spaces are closed in the category of antipodal spaces in the following sense. If $(Z,\rho_0)$, an antipodal space, is an AI-limit of a sequence of CAT$(-1)$ fillable antipodal spaces, then $(Z,\rho_0)$ is also CAT$(-1)$ fillable (see \Cref{CAT(-1) fillable closed}). 
	
	\begin{sublemma}\label{equicontinuity of tau 2}
		Suppose $\{(Z_n,\rho_n)\}_{n\ge 1}$ be a sequence of antipodal metric spaces. Fix $R>0$, given $\epsilon>0$ there exists $\delta=\delta(\epsilon,R)>0$ such that if $\rho_n(\xi,\eta)<\delta$ then
		\begin{equation*}
			\sup_{\rho\ \in\  B_{\MM(Z_n)}(\rho_n,R)} \big|\tau_\rho(\xi)-\tau_\rho(\eta)\big|<\epsilon.
		\end{equation*}
	\end{sublemma}
	\begin{proof}
		$(Z_n,\rho_n)$ being antipodal metric spaces, it constitutes an equicontinuous family. Hence, the lemma is a direct consequence of \Cref{equicontinuity of tau main}.  
	\end{proof}
	
	It is well known for a CAT$(-1)$ space $X$ Next we will show, if $\{X_n\}_{n\ge1}$ is a family of CAT$(-1)$ spaces, then the antipodal spaces $\{(\partial X_n,\rho_n)\}_{n\ge 1}$ form an equicontinuous family. This holds even if the $\rho_n$'s are not necessarily visual metrics.
	
	\begin{sublemma}\label{equicontinuity lemma}
		Let $\{{X_n}\}_{n\ge 1}$ be a sequence of CAT$(-1)$ spaces and let $\rho_n\in \MM(\partial X_n)$ for all $n$. Then $\{(\partial X_n,\rho_n)\}_{n\ge 1}$ are equicontinuous.  
	\end{sublemma}
	\begin{proof}
		In \cite[Theorem 6.1]{biswas2015moebius}, we have seen that for any proper geodesically complete CAT$(-1)$ space $X$, the image of the visual embedding $i_X\colon X\to \MM(\partial X), x\mapsto \rho_x$ (see \Cref{Gromov ip}), is $(\log 2)/2$-dense in $\MM(\partial X)$ (cf. \cite[Theorem 8.11]{biswas2024quasi}). Therefore, for every given $\rho_n$ there exists visual metrics $\rho_{x_n}$ for some $x_n\in X_n$ such that 
		\begin{equation*}
			d_{\MM(\partial X_n)}(\rho_n,\rho_{x_n})=\left\Vert\log \frac{d\rho_n}{d\rho_{x_n}}\right\Vert_\infty\le \frac{1}{2}\log 2.
		\end{equation*}
		Let us denote $\tau_n\coloneqq \log \frac{d\rho_n}{d\rho_{x_n}}$, then $\|e^{\tau_n}\|_\infty\le 2$. Thus, $\rho_n$ and $\rho_{x_n}$ satisfies,
		\begin{equation*} 
			\frac{1}{2} \cdot \rho_n\le\rho_{x_n}\le 2\cdot\rho_n \quad \quad \hbox{(by \Cref{GMVT})}.
		\end{equation*} 
		For $\xi,\eta,\zeta \in \partial X_n$ by GMVT (\Cref{GMVT})
		\begin{equation*}
			\begin{split}
				\big|\rho_n(\xi,\zeta)-\rho_n(\eta,\zeta)\big|&\le \big|e^{\tau_n(\xi)/2}e^{\tau_n(\zeta)/2}\rho_{x_n}(\xi,\zeta)-e^{\tau_n(\eta)/2}e^{\tau_n(\zeta)/2}\rho_{x_n}(\eta,\zeta)\big|\\
				&\le \big|e^{\tau_n(\zeta)/2}\big|\cdot\big|e^{\tau_n(\xi)/2} \rho_{x_n}(\xi,\zeta)-e^{\tau_n(\eta)/2}\rho_{x_n}(\eta,\zeta)\big|\\
				&\le \big\| e^{\tau_n/2} \big\|_{\infty}\ \big|e^{\tau_n(\xi)/2}
				\rho_{x_n}(\xi,\zeta)-e^{\tau_n(\eta)/2}\rho_{x_n}(\eta,\zeta)\big|.
			\end{split}
		\end{equation*} 
		And for this last term
		\begin{equation*}
			\begin{split}
				&\ \big|e^{\tau_n(\xi)/2}\rho_{x_n}(\xi,\zeta)-e^{\tau_n(\eta)/2}\rho_{x_n}(\eta,\zeta)\big|\\
				\le&\ \big|e^{\tau_n(\xi)/2} \rho_{x_n}(\xi,\zeta)-e^{\tau_n(\eta)/2}\rho_{x_n}(\xi,\zeta)\big|+\big|e^{\tau_n(\eta)/2}\rho_{x_n}(\xi,\zeta)-e^{\tau_n(\eta)/2}\rho_{x_n}(\eta,\zeta)\big|\\
				\le&\ \big|e^{\tau_n(\xi)/2}-e^{\tau_n(\eta)/2}\big|+ \big|e^{\tau_n(\eta)/2}\big|\cdot\big|\rho_{x_n}(\xi,\zeta)-\rho_{x_n}(\eta,\zeta)\big|\\
				\le&\ \big|e^{\tau_n(\xi)/2}-e^{\tau_n(\eta)/2}\big|+ \big\| e^{\tau_n/2}\big\|_{\infty}\ \big|\rho_{x_n}(\xi,\zeta)-\rho_{x_n}(\eta,\zeta)\big|.
			\end{split}
		\end{equation*}
		For the first term, we have 
		\begin{equation*}
			\big|e^{\tau_n(\xi)/2}-e^{\tau_n(\eta)/2}\big|\le \big\| e^{\tau_n/2}\big\|_{\infty}\ \big|\tau_n(\xi)-\tau_n(\eta)\big|	
		\end{equation*} 
		and for the second, since $\rho_{x_n}$ is a metric, we have by triangle inequality, 
		\begin{equation*}
			\big|\rho_{x_n}(\xi,\zeta)-\rho_{x_n}(\eta,\zeta)\big|\le \rho_{x_n}(\xi,\eta)= e^{-\tau_n(\xi)/2}e^{-\tau_n(\eta)/2}\rho_n(\xi,\eta)
		\end{equation*} (last equality by GMVT). Since $\|\tau_n\|_\infty\le \log 2$, there exists some universal constant $c_{15}>0$ such that
		\begin{equation}\label{almost equicontinuous}
			\big|\rho_n(\xi,\zeta)-\rho_n(\eta,\zeta)\big|\le c_{15}\ (\ \big|\tau_n(\xi)-\tau_n(\eta)\big|+ \rho_n(\xi,\eta)\ ).
		\end{equation} 
		We have $(\partial X_n,\rho_{x_n})$ are antipodal metric spaces, since $X_n$ are CAT$(-1)$ spaces (see \cite{bourdon1995structure}). Now we use the previous \Cref{equicontinuity of tau 2} to get $\tau_n$'s are equicontinuous. Therefore, there exists $\delta=\delta(\epsilon)>0$ small enough (depending only on $\epsilon$) such that, if $\xi,\eta\in \partial X_n$ with $\rho_{x_n}(\xi,\eta)\le 2\rho_n(\xi,\eta)<\delta$ then (from \eqref{almost equicontinuous}) we have, 
		\begin{equation*}
			\big|\rho_n(\xi,\zeta)-\rho_n(\eta,\zeta)\big|<\epsilon
		\end{equation*} for all $\zeta\in \partial X_n$.
	\end{proof}
	
	\begin{subrmk}\label{equicontinuity remark}
		Suppose $(Z_n,\rho_n)$'s are $\text{CAT}(-1)$ fillable. Then from the above \Cref{equicontinuity lemma} we have $\{(Z_n,\rho_n)\}_{n\ge 1}$ is an  equicontinuous family of antipodal spaces.
	\end{subrmk}
	
	\medskip
	
	\begin{subprop}\label{CAT(-1) fillable}
		Let $\{(Z_n,\rho_n)\}_{n\ge 1}$ be a sequence of CAT$(-1)$ fillable antipodal spaces, and let $(Z,\rho_0)$ be another antipodal space such that $(\MM(Z_n),\rho_n)\xrightarrow{GH\ conv.} (\MM(Z),\rho_0)$ then $(Z_n,\rho_n)\xrightarrow{AI\ conv.} (Z,\rho_0)$.
	\end{subprop}
	\begin{proof}
		By \Cref{equicontinuity remark} the sequence  $\{(Z_n,\rho_n)\}_{n\ge 1}$ is an equicontinuous family of antipodal functions. The Moebius spaces are maximal Gromov product spaces, and the antipodal spaces correspond to Gromov product boundaries. Hence, the proposition follows directly from \Cref{backward}.
	\end{proof}
	
	Next we show that, the CAT$(-1)$ fillable antipodal spaces are closed (in the sense of AI-convergence) in the category of antipodal spaces:
	
	\begin{subprop}\label{CAT(-1) fillable closed}
		Let $\{(Z_n,\rho_n)\}_{n\ge 1}$ be a sequence of CAT$(-1)$ fillable spaces. Suppose $(Z_n,\rho_n)\xrightarrow{AI\ conv.}(Z,\rho_0)$, another antipodal space. Then $(Z,\rho_0)$ is also CAT$(-1)$ fillable. 
	\end{subprop}
	
	\begin{proof}
		By \Cref{forward} 
		\begin{equation*}
			(\MM(Z_n),\rho_n)\xrightarrow{GH\ conv.}(\MM(Z),\rho_0).
		\end{equation*}
		Then for every $R>0$ we have, $B_{\MM(Z_n)}(\rho_n,R)\xrightarrow{d_{GH}}B_{\MM(Z)}(\rho_0,R)$ as compact metric spaces. Therefore, for every $R>0$ and $\epsilon>0$ there exists $N(R,\epsilon)\in \mathbb{N}$ such that, for each $n$, the closed ball $B_{\MM(Z_n)}(\rho_n,R)$ admits an $\epsilon$-net with at most $N(R,\epsilon)$ points (see \cite[Chapter I.5]{bridson1999metric}). Now $(Z_n,\rho_n)$'s being CAT$(-1)$ fillable there exist proper geodesic geodesically complete CAT$(-1)$ spaces $X_n$'s equipped with a Moebius homeomorphism $\partial X_n\xrightarrow{\varphi_n} (Z_n,\rho_n)$, i.e. $\varphi_n$ preserves the canonical cross-ratios given by the visual metrics. According to \Cref{induced isometry} we have the isometric embedding of $X_n$ into $\MM(Z_n)$, say $\psi_n\coloneqq {\varphi_n}_*\circ\ i_{X_n}$,  
		
		\[\begin{tikzcd}[column sep=2em, row sep=tiny]
			{X_n} && {\MM(\partial X_n)} && {\MM(Z_n)} \\
			x && {\rho_x} && {{\varphi_n}_*\rho_x}
			\arrow["{{i_{X_n}}}", from=1-1, to=1-3]
			\arrow["{{\psi_n}}", curve={height=-40pt}, from=1-1, to=1-5]
			\arrow["{{{\varphi_n}_*}}", from=1-3, to=1-5]
			\arrow[maps to, from=2-1, to=2-3]
			\arrow[maps to, from=2-3, to=2-5]
		\end{tikzcd}\]
		where $i_{X_n}$ is the visual embedding and ${\varphi_n}_*$ is the push-forward (see \Cref{induced isometry}). Again from \cite{biswas2015moebius} we know that the image is $(\log 2)/2$-dense in $\MM(Z_n)$. Thus, for all $n$, there exists $x_n\in X_n$ such that 
		\begin{equation}\label{bounded distance}
			d_{\MM(Z_n)}(\psi_n(\rho_{x_n}),\rho_n)<\frac{1}{2}\log 2.
		\end{equation}
		Fix $R>0$. Consider close balls $B_{X_n}(x_n,R)\subset X_n$, then $\psi_n(B_{X_n}(x_n,R))\subset B_{\MM(Z_n)}(\rho_n,R+1)$. Given $\epsilon>0$ small,  there exists $N=N(R+1,\epsilon)$, such that for every $n$, we can find an $\epsilon$-net say $K_n$ of $B_{\MM(Z_n)}(\rho_n,R+1)$ with cardinality at most $N$ (as pointed out above). 
		
		From $K_n$'s we will construct $2\epsilon$-nets of $B_{X_n}(x_n,R)$'s with cardinality uniformly bounded by $N$ for each $n$. Define, 
		\begin{equation*}
			L_n\coloneqq\{ \rho\in K_n\ |\ \psi_n(B_{X_n}(x_n,R))\cap B_{\MM(Z_n)}(\rho,\epsilon)\neq\emptyset\ \}.
		\end{equation*} For each point $\rho\in L_n$ we can pick one point from the intersection $\psi_n(B_{X_n}(x_n,R))\cap B_{\MM(Z_n)}(\rho,\epsilon)$, look at its preimage under $\psi_n$ to consequently get $M_n$, which will be a $2\epsilon$-net of $B_{X_n}(x_n,R)$, and having cardinality at most $N$. Therefore, for every $R>0$ and $\epsilon>0$, we can find such $2\epsilon$-nets $M_n$ of $B_{X_n}(x_n,R)$ in $X_n$, such that the cardinality is uniformly bounded for all $n$. The bound depends only on $R$ and $\epsilon$.
		Therefore, by Gromov's pre-compactness \Cref{GCT} we can pass to a sub-sequence  of $\{(X_n,x_n)\}$, and find a pointed metric space $(X,x)$ such that 
		\begin{equation*}
			(X_n,x_n)\xrightarrow{GH\ conv.}(X,x).
		\end{equation*} Furthermore, $X$ is a proper, geodesically complete CAT$(-1)$ space, since it is a pointed GH-limit of proper, geodesically complete CAT$(-1)$ spaces. Consequently, by \Cref{CAT(-1) main}, we have $(\partial X_n,\rho_{x_n})\xrightarrow{d_{GH}} (\partial X,\rho_x)$. Recall that we have Moebius homeomorphism $\varphi_n\colon (\partial X_n,\rho_{x_n})\to(Z_n,\rho_n)$, for each $n$. From the boundedness condition \eqref{boundedness condition} we have, 
		\begin{equation*}
			d_{\MM(Z_n)}({\varphi_n}_*(\rho_{x_n}),\rho_n)<\frac{1}{2}\log 2.
		\end{equation*}
		Then by the application of \Cref{Mobius equivalence prop}, there exists a Moebius homeomorphism $\varphi\colon(\partial X,\rho_x)\to (Z,\rho_0)$. Hence $(Z,\rho_0)$ is CAT$(-1)$ fillable.
	\end{proof}
	
	\medskip
	
	We have the following corollary as a consequence of the above proposition and \Cref{CAT(-1) fillable}.
	\begin{subcoro}
		Let $\{(Z_n,\rho_n)\}_{n\ge 1}$ be CAT$(-1)$ fillable antipodal spaces and let $(Z,\rho_0)$ be another antipodal space such that $(\MM(Z_n),\rho_n)\xrightarrow{GH\ conv.} (\MM(Z),\rho_0)$. Then $(Z,\rho_0)$ is CAT$(-1)$ fillable.
	\end{subcoro}
	
	\bigskip
	
	\appendix
	\section{Equicontinuity}
	\begin{lemma}\label{equicontinuity of tau main}
		Let $\{(Z_n,\rho_n)\}_{n\ge 1}$ be a equicontinuous family of antipodal spaces. Fix $R>0$, given $\epsilon>0$ there exists  $\delta=\delta(R,\epsilon)>0$ such that for all $n$, if $\xi,\eta\in Z_n$ and $\rho_n(\xi,\eta)<\delta$ then 
		\begin{equation*}
			\big|\tau_\rho(\xi)-\tau_\rho(\eta)\big|<\epsilon
		\end{equation*}
		for all $\rho\in{B}_{\MM(Z_n)}(\rho_n,R)$.
	\end{lemma}
	\begin{proof}
		Let us suppose this is false. Then there exists $\epsilon>0$ such that for infinitely many $n$ we can find $\xi_n,\eta_n\in Z_n$ with $\rho_n(\xi_n,\eta_n)\to 0$ and $\tau_{\beta_n}(\eta_n)-\tau_{\beta_n}(\xi_n)>\epsilon$ for some $\beta_n\in {B}_{\MM(Z_n)}(\rho_n,R)$
		
		Let $\tilde{\xi_n}\in Z_n$ be $\beta_n$-antipode of $\xi_n$,  i.e. $\beta_n(\xi_n,\tilde{\xi_n})=1$. Then by GMVT (Lemma \Cref{GMVT})
		\begin{equation}\label{GMVTused}
			\begin{split}
				\tau_{\beta_n}(\eta_n)+\tau_{\beta_n}(\tilde{\xi_n})+\log \rho_n(\eta_n,\tilde{\xi_n})^2&\le 0=\tau_{\beta_n}(\xi_n)+\tau_{\beta_n}(\tilde{\xi_n})+\log \rho_n(\xi_n,\tilde{\xi_n})^2\\
				\implies \tau_{\beta_n}(\eta_n)-\tau_{\beta_n}(\xi_n)&\le \log \frac{\rho_n(\xi_n,\tilde{\xi_n})^2}{\rho_n(\eta_n,\tilde{\xi_n})^2}
			\end{split}
		\end{equation} 
		also
		\begin{equation*}
			1\ge \rho_n(\xi_n,\tilde{\xi_n})=\left(\frac{d\beta_n}{d\rho_0}(\xi_n)\right)^{1/2}\left(\frac{d\beta_n}{d\rho_0}(\tilde{\xi_n})\right)^{1/2}\ge e^{-\|\tau_{\beta_n}\|_\infty}\ge e^{-R}.
		\end{equation*}
		So passing to a sub-sequence there exist some $\ell\in(e^{-R},1]\subset\R$ such that
		\begin{equation}\label{converging2ell}
			\rho_n(\xi_n,\tilde{\xi_n}) \rightarrow \ell\ge e^{-R}.
		\end{equation} 
		\noindent Now by equicontinuity property of $(Z_n,\rho_n)$'s we have $\rho_n(\eta_n,\xi_n)\to 0+$ as $n\to \infty$ implies
		\begin{equation*}
			\big|\rho_n(\eta_n,\tilde{\xi_n})-\rho_n(\xi_n,\tilde{\xi_n})\big|\rightarrow 0 \quad {\text as} \quad n\to \infty
		\end{equation*} 
		
		\noindent From the above convergence, we get that 
		\begin{equation}\label{converging2ell'}
			\rho_n(\eta_n,\tilde{\xi_n})\rightarrow \ell\ge e^{-R}
		\end{equation}
		along the same sub-sequence as in \eqref{converging2ell}. 
		Using \eqref{converging2ell} and \eqref{converging2ell'} we get,
		\begin{equation*}
			2\log \frac{\rho_n(\xi_n,\tilde{\xi_n})}{\rho_n(\eta_n,\tilde{\xi_n})}\rightarrow 0,
		\end{equation*}
		along some sub-sequence. But from \eqref{GMVTused} we have 
		\begin{equation*}
			\epsilon<\tau_{\beta_n}(\eta_n)-\tau_{\beta_n}(\xi_n)\le 2\log \frac{\rho_n(\xi_n,\tilde{\xi_n})}{\rho_n(\eta_n,\tilde{\xi_n})},
		\end{equation*}
		for all $n$ along the sub-sequence,  which contradicts the initial hypothesis.
	\end{proof}
	
	\bigskip
	
	\section{Discrepancy}
	
	\begin{lemma}\label{discrepancy main}
		Let $\{(Z_\lambda,\rho_\lambda)\}_{\lambda\in\Lambda}$ be an equicontinuous family of antipodal spaces. Fix $R>0$, given $\epsilon>0$ there exists $\delta=\delta(R,\epsilon)>0$ such that for any other antipodal space $(Z,\rho_\omega)$, if we have a $\delta$-isometry $h_{\lambda}\colon (Z,\rho_\omega)\to (Z_\lambda,\rho_\lambda)$ for any $\lambda\in \Lambda$ 
		then,
		\begin{equation*}
			\sup_{\rho\in B_{\MM(Z_\lambda)}(\rho_\lambda,R)}\ \big\|D_{\rho_\omega}(\tau_\rho\circ h_\lambda)\big\|_\infty<\epsilon.
		\end{equation*}
	\end{lemma}
	\begin{proof}
		The proof is in two steps. Enough to prove for $\epsilon$ small, so that $2e^R\cdot \epsilon<1$ (note this choice will be useful in Step 2 of the proof).\\
		
		\noindent \underline{Step 1 $\colon$} 
		Let $\lambda\in \Lambda$. Let $\rho\in {B}_{\MM(Z_\lambda)}(\rho_\lambda,R)$ and let $h_\lambda\colon (Z,\rho_\omega)\to(Z_\lambda,\rho_\lambda)$ be a $\delta$-isometry (for some $\delta >0$ ) then for $\xi\in Z$ we have 
		\begin{align*}
			D_{\rho_\omega}(\tau_\rho\circ h_\lambda)(\xi)&\coloneqq \sup_{\eta\in Z\setminus\{\xi\}}\tau_\rho\circ h_\lambda(\xi)+\tau_\rho\circ h_\lambda(\eta)+\log\rho_\omega(\xi,\eta)^2\\
			&\le\sup_{\eta\in Z\setminus\{\xi\}}\tau_\rho\circ h_\lambda(\xi)+\tau_\rho\circ h_\lambda(\eta)+\log(\rho_\lambda(h_\lambda(\xi),h_\lambda(\eta))+\delta)^2
		\end{align*}
		We  want to compare this last term with 
		\begin{equation}\label{expression after sup}
			\sup_{\eta\in Z_\lambda\setminus\{h_\lambda(\xi)\}}\tau_\rho( h_\lambda(\xi))+\tau_\rho(\eta)+\log(\rho_\lambda(h_\lambda(\xi),\eta)+\delta)^2
		\end{equation}
		It is possible to choose $\delta$ small enough depending on $R$ such that, if $h_\lambda(\xi)=\eta$ then 
		\begin{align*}
			\tau_\rho\circ h_\lambda(\xi)+\tau_\rho(\eta)+\log(\rho_\lambda(h_\lambda(\xi),\eta)+\delta)^2
			\le 2R+\log(\delta)^2< -4R,
		\end{align*} and we know from the inequality \eqref{discrepancy less than 2R}  \begin{equation*}
			-2R\ \le\ -2\|\tau_\rho\|_\infty\ \le\ D_{\rho_\omega}(\tau_\rho\circ h_\lambda)(\xi).
		\end{equation*} So the expression 
		we considered in \eqref{expression after sup} attains the supremum on $Z_\lambda\setminus \{h_\lambda(\xi)\}$. Therefore, we can write $\colon$  for $\delta$ small enough,
		\begin{equation*}
			\begin{split}
				D_{\rho_\omega}(\tau_\rho\circ h_\lambda)(\xi)&\le\sup_{\eta\in Z\setminus\{\xi\}}\tau_\rho\circ h_\lambda(\xi)+\tau_\rho\circ h_\lambda(\eta)+\log(\rho_\lambda(h_\lambda(\xi),h_\lambda(\eta))+\delta)^2\\
				&\le \sup_{\eta\in Z_\lambda\setminus\{h_\lambda(\xi)\}}\tau_\rho(h_\lambda(\xi))+\tau_\rho(\eta)+\log(\rho_\lambda(h_\lambda(\xi),\eta)+\delta)^2\\
				&=\  \tau_\rho(h_\lambda(\xi))+\tau_\rho(\xi_\lambda')+2\log(\rho_\lambda(h_\lambda(\xi),\xi_\lambda')+\delta)
			\end{split}
		\end{equation*}
		for some $\xi_\lambda'\in Z_\lambda\setminus\{h_\lambda(\xi)\}$. 
		
		From the inequality \eqref{discrepancy less than 2R} and above observations we have, 
		\begin{equation}\label{bound for rho_lambda}
			\begin{split}
				2\log(\rho_\lambda(h_\lambda(\xi),\xi_\lambda')+\delta)&\ge D_{\rho_\omega}(\tau_\rho\circ h_\lambda)(\xi)-\tau_\rho(h_\lambda(\xi))-\tau_\rho(\xi_\lambda')\\
				& \ge-2\|\tau_\rho\|_\infty-2\|\tau_\rho\|_\infty\ge -4R\\
				\hbox{which implies}\quad \rho_\lambda(h_\lambda(\xi),\xi_\lambda')&\ge e^{-2R}-\delta.
			\end{split}
		\end{equation}
		
		We know that 
		\begin{equation*}
			\log(x + h) < \log x + \frac{h}{x}
		\end{equation*} for all $x, h > 0$, and thus we have,
		\begin{equation}\label{discrepancy tau g_n upper bound}
			\begin{split}
				D_{\rho_\omega}(\tau_\rho\circ h_\lambda)(\xi)&<\tau_\rho(h_\lambda(\xi))+\tau_\rho(\xi_\lambda')+\log(\rho_\lambda(h_\lambda(\xi),\xi_\lambda'))^2+\frac{2\delta}{\rho_\lambda(h_\lambda(\xi),\xi_\lambda')}\\ 
				&\le D_{\rho_\lambda}(\tau_\rho)(h_\lambda(\xi))+\frac{2\delta}{\rho_\lambda(h_\lambda(\xi),\xi_\lambda')}.
			\end{split}
		\end{equation}
		Note that $D_{\rho_\lambda}(\tau_\rho)\equiv 0$. 
		Combining this with inequality \eqref{bound for rho_lambda},  it is possible to choose $\delta=\delta(\epsilon,R)>0$ small enough so that for $\xi\in Z$,
		\begin{equation*}
			D_{\rho_\omega}(\tau_\rho\circ h_\lambda)(\xi)\le \frac{2\delta}{\rho_\lambda(h_\lambda(\xi),\xi_\lambda')}\le \frac{2\delta}{ e^{-2R}-\delta}\le \epsilon.
		\end{equation*}
		Hence, there exists $\delta>0$ small enough, such that if we have a $\delta$-isometry $h_{\lambda}\colon (Z,\rho_\omega)\to (Z_\lambda,\rho_\lambda)$ for any $\lambda\in \Lambda$ then for all $\rho\in B_{\MM(Z_\lambda)}(\rho_\lambda,R)$
		\begin{equation*}
			\sup_{\xi\in Z}D_{\rho_\omega}(\tau_\rho\circ h_\lambda)(\xi)<\epsilon.
		\end{equation*}
		\bigskip

		\noindent \underline{Step 2 $\colon$}  Since $(Z_\lambda,\rho_\lambda)$'s are equicontinuous and by \Cref{equicontinuity of tau main}, there exists $\delta_0=\delta_0(\epsilon,R)$ positive such that for all $\lambda\in \Lambda$ if $\zeta,\eta\in Z_\lambda$ and $\rho_\lambda(\zeta,\eta)<\delta_0$ then 
		\begin{equation*}
			\big|\tau_\rho(\eta)-\tau_\rho(\zeta)\big|<\frac{\epsilon}{2}
		\end{equation*}
		for all $ \rho\in {B}_{\MM(Z_\lambda)}(\rho_\lambda,R)$, and also
		\begin{equation*}
			\big|\rho_\lambda(\xi,\eta)-\rho_\lambda(\xi,\zeta)\big|<\frac{\epsilon}{2}
		\end{equation*}
		for all $\xi\in Z_\lambda$. \\
		
		Let $\xi\in Z$, $ \rho\in {B}_{\MM(Z_\lambda)}(\rho_\lambda,R) $, and $h_\lambda\colon (Z,\rho_\omega)\to(Z_\lambda,\rho_\lambda)$ be $\delta$-isometry, then there exists $\tilde{\xi_\lambda}\in Z_\lambda$ a $\rho$-antipode of $h_\lambda(\xi)$, i.e. $\rho(h_\lambda(\xi),\tilde{\xi_\lambda})=1$ and also $\tilde{\xi}\in Z$ such that $\rho_\lambda(h_\lambda(\tilde{\xi}),\tilde{\xi_\lambda})<\delta$.
		Now by GMVT (\Cref{GMVT}),
		\begin{equation*}
			\rho_\lambda(h_\lambda(\xi),\tilde{\xi_\lambda})=\frac{d\rho}{d\rho_\lambda}(h_\lambda(\xi))^{\frac{1}{2}}\frac{d\rho}{d\rho_\lambda}(\tilde{\xi_\lambda})^{\frac{1}{2}}\ge e^{-\|\tau_\rho\|_\infty}\ge e^{-R}.
		\end{equation*}
		If we choose $\delta<\delta_0$ then we have $\rho_\lambda(h_\lambda(\tilde{\xi}),\tilde{\xi_\lambda})<\delta<\delta_0$ and therefore
		\begin{align*}
			|\rho_\lambda(h_\lambda(\xi),h_\lambda(\tilde{\xi}))&-\rho_\lambda(h_\lambda(\xi),\tilde{\xi_\lambda})|<\frac{\epsilon}{2}.
		\end{align*}
		Further for $\delta<\epsilon$, $h_\lambda$ being $\delta$-isometry we have from above (also remember the choice of $\epsilon<(e^{-R})/2$),
		\begin{equation}\label{ineq 33}
			\begin{split}
				\rho_\omega(\xi,\tilde{\xi})>\rho_\lambda(h_\lambda(\xi),h_\lambda(\tilde{\xi}))-\delta &>\rho_\lambda(h_\lambda(\xi),\tilde{\xi_\lambda})-\frac{\epsilon}{2}-\delta\\
				&>\rho_\lambda(h_\lambda(\xi),\tilde{\xi_\lambda})-\epsilon\\
				&>e^{-R}-\epsilon.\\
			\end{split}
		\end{equation}
		Then from inequality \eqref{ineq 33} we have,
		\begin{equation}
			\begin{split}
				D_{\rho_\omega}(\tau_\rho\circ h_\lambda)(\xi)&\ge\tau_\rho(h_\lambda(\xi))+\tau_\rho(h_\lambda(\tilde{\xi}))+\log\rho_\omega(\xi,\tilde{\xi})^2\\
				&>\tau_\rho(h_\lambda(\xi))+\tau_\rho(h_\lambda(\tilde{\xi}))+\log(\rho_\lambda(h_\lambda(\xi),\tilde{\xi_\lambda})-\epsilon)^2.
			\end{split}
		\end{equation}
		Again $\rho_\lambda(h_\lambda(\tilde{\xi}),\tilde{\xi_\lambda})<\delta<\delta_0$ implies $\tau_\rho(h_\lambda(\tilde{\xi}))>\tau_\rho(\tilde{\xi_\lambda})-\epsilon/2$, hence,
		\begin{align*} 
			D_{\rho_\omega}(\tau_\rho\circ 	h_\lambda)(\xi)&>\tau_\rho(h_\lambda(\xi))+\tau_\rho(h_\lambda(\tilde{\xi}))+\log(\rho_\lambda(h_\lambda(\xi),\tilde{\xi_\lambda})-\epsilon)^2\\
			&>\tau_\rho(h_\lambda(\xi))+(\tau_\rho(\tilde{\xi_\lambda})-\epsilon)+2\log(\rho_\lambda(h_\lambda(\xi),\tilde{\xi_\lambda})-\epsilon)
		\end{align*}
		We know that 
		\begin{equation*}
			\log (x-h)>\log x-\frac{h}{x-h}
		\end{equation*} for all $x>h>0$. Then from the last expression
		\begin{align*}
			D_{\rho_\omega}(\tau_\rho\circ 	h_\lambda)(\xi)&> \tau_\rho(h_\lambda(\xi))+(\tau_\rho(\tilde{\xi_\lambda})-\epsilon)+\log\rho_\lambda(h_\lambda(\xi),\tilde{\xi_\lambda})^2-\frac{2\epsilon}{\rho_\lambda(h_\lambda(\xi),\tilde{\xi_\lambda})-\epsilon}\\
			&=\log\rho(h_\lambda(\xi),\tilde{\xi_\lambda})^2-\left(\epsilon+\frac{2\epsilon}{\rho_\lambda(h_\lambda(\xi),\tilde{\xi_\lambda})-\epsilon}\right).\\
		\end{align*}
		Recall that $\log\rho(h_\lambda(\xi),\tilde{\xi_\lambda})=0$ since $\tilde{\xi_\lambda}$ is $\rho$-antipode of $h_\lambda(\xi)$. Now by the choice of $\epsilon<(e^{-R})/2$ we have from \eqref{ineq 33},
		\begin{align*}
			\rho_\lambda(h_\lambda(\xi),\tilde{\xi_\lambda})-\epsilon>\frac{e^{-R}}{2}\implies \frac{2\epsilon}{\rho_\lambda(h_\lambda(\xi),\tilde{\xi_\lambda})-\epsilon}<4\epsilon e^R.
		\end{align*}
		From this we get,
		\begin{align*}
			D_{\rho_\omega}(\tau_\rho\circ h_\lambda)(\xi)&>-\left(\epsilon+\frac{2\epsilon}{\rho_\lambda(h_\lambda(\xi),\tilde{\xi_\lambda})-\epsilon}\right)> c_1\cdot \epsilon.
		\end{align*}
		where $c_1=c_1(R)>1$ is a constant only dependent on $R$. Thus, it follows from the above that, staring with $\epsilon/c_1$ instead of $\epsilon$, we would eventually get for any $\lambda\in \Lambda$, if $h_{\lambda}\colon (Z,\rho_\omega)\to (Z_\lambda,\rho_\lambda)$ is a $\delta$-isometry, for $\delta$ suitably small, then for all $\rho\in B_{\MM(Z_\lambda)}(\rho_\lambda,R)$
		\begin{equation*}
			\inf_{\xi\in Z}D_{\rho_\omega}(\tau_\rho\circ h_\lambda)(\xi)>-\epsilon.
		\end{equation*} 
		
		\medskip
		
		\noindent From the conclusion of the two steps, we have the claimed result.	
	\end{proof}
	\bibliography{References}

\begin{thebibliography}{{Tuz}16}

\bibitem[AKP24]{alexander2024alexandrov}
S.~Alexander, V.~Kapovitch, and A.~Petrunin.
\newblock {\em Alexandrov Geometry: Foundations}.
\newblock Graduate Studies in Mathematics. American Mathematical Society, 2024.

\bibitem[AP56]{aronszajn1956extension}
N.~Aronszajn and P.~Panitchpakdi.
\newblock Extension of uniformly continuous transformations and hyperconvex
  metric spaces.
\newblock {\em Pacific Journal of Mathematics}, 6(3):405--439, 1956.

\bibitem[BBI22]{burago2022course}
D.~Burago, Y.~Burago, and S.~Ivanov.
\newblock {\em A course in metric geometry}, volume~33.
\newblock American Mathematical Society, 2022.

\bibitem[Ben06]{benoist2006convex}
Y.~Benoist.
\newblock Convexes hyperboliques et quasiisom{\'e}tries.
\newblock {\em Geometriae Dedicata}, 122(1):109--134, 2006.

\bibitem[BH99]{bridson1999metric}
M.~R. Bridson and A.~Haefliger.
\newblock {\em Metric spaces of non-positive curvature}, volume 319 of {\em
  Grundlehren der mathematischen Wissenschaften}.
\newblock Springer Berlin, Heidelberg, Germany, 1999.

\bibitem[Bis15]{biswas2015moebius}
K.~Biswas.
\newblock On {Moebius and conformal maps between boundaries of CAT}$(-1)$
  spaces.
\newblock {\em Annales de l'Institut Fourier}, 65(3):1387--1422, 2015.

\bibitem[Bis16]{biswas4}
K.~Biswas.
\newblock Local and infinitesimal rigidity of simply connected negatively
  curved manifolds.
\newblock {\em Annales de l'Institut Fourier, Vol. 66 no. 6}, pages 2507--2523,
  2016.

\bibitem[Bis24]{biswas2024quasi}
K.~Biswas.
\newblock Quasi-metric antipodal spaces and maximal {G}romov hyperbolic spaces.
\newblock {\em Geometriae Dedicata}, 218(2):53, 2024.

\bibitem[BK85]{burns-katok}
K.~Burns and A.~Katok.
\newblock Manifolds with non-positive curvature.
\newblock {\em Ergodic Theory and Dynamical Systems, {\bf 5}, no. 2}, pages
  307--317, 1985.

\bibitem[Bou95]{bourdon1995structure}
M.~Bourdon.
\newblock Structure {conforme au bord et flot geodesic d'un CAT}$(-1)$-espace.
\newblock {\em L'Enseignement Math}, 41:63--102, 1995.

\bibitem[Bou96]{bourdon2}
M.~Bourdon.
\newblock Sur le birapport au bord des {C}{A}{T}(-1) espaces.
\newblock {\em Inst. Hautes Etudes Sci. Publ. Math. No. 83}, pages 95--104,
  1996.

\bibitem[BP24]{biswas2024polyhedral}
K.~{Biswas} and A.~{Pal Choudhury}.
\newblock {Polyhedral structure of maximal Gromov hyperbolic spaces with finite
  boundary}.
\newblock {\em arXiv e-prints}, page
  \href{10.48550/arXiv.2410.18579}{arXiv:2410.18579}, October 2024.

\bibitem[BS00]{bonk-schramm}
M.~Bonk and O.~Schramm.
\newblock Embeddings of {G}romov hyperbolic spaces.
\newblock {\em Geometric and Functional Analysis, {\bf 10}}, pages 266--306,
  2000.

\bibitem[BS07]{buyalo2007elements}
S.~Buyalo and V.~Schroeder.
\newblock {\em Elements of asymptotic geometry}, volume~3.
\newblock European Mathematical Society, 2007.

\bibitem[BS17]{beyrer-schroeder}
J.~Beyrer and V.~Schroeder.
\newblock Trees and ultrametric {M}oebius structures.
\newblock {\em p-adic Numbers, Ultrametric Analysis and Applications, {\bf 9}},
  pages 247--256, 2017.

\bibitem[Des15]{descombes2015bicombings}
D.~Descombes.
\newblock {\em Spaces with Convex Geodesic Bicombings}.
\newblock PhD thesis, ETH Zurich, 2015.

\bibitem[Edw75]{edwards1975}
D.~A. Edwards.
\newblock The structure of superspace.
\newblock In Nick~M. Stavrakas and Keith~R. Allen, editors, {\em Studies in
  Topology}, pages 121--133. Academic Press, 1975.

\bibitem[Gro81a]{gromov1981group}
M.~Gromov.
\newblock Groups of polynomial growth and expanding maps.
\newblock {\em Publications Math{\'e}matiques de l'Institut des Hautes
  {\'E}tudes Scientifiques}, 53(1):53--78, 1981.

\bibitem[Gro81b]{gromov1981structures}
M.~Gromov.
\newblock {\em Structures m{\'e}triques pour les vari{\'e}t{\'e}s
  riemanniennes}.
\newblock Textes math{\'e}matiques. Recherche. CEDIC/Fernand Nathan, 1981.

\bibitem[Hug04]{Hughes}
B.~Hughes.
\newblock Trees and ultrametric spaces: a categorical equivalence.
\newblock {\em Advances in Mathematics, Vol. 189, Issue 1}, pages 148--191,
  2004.

\bibitem[Jor10]{jordi}
J.~Jordi.
\newblock Interplay between interior and boundary geometry in {G}romov
  hyperbolic spaces.
\newblock {\em Geom Dedicata, 149}, pages 129--154, 2010.

\bibitem[Lan13]{lang2013injective}
U.~Lang.
\newblock Injective hulls of certain discrete metric spaces and groups.
\newblock {\em Journal of Topology and Analysis}, 5(03):297--331, 2013.

\bibitem[Lyt19]{lytchak2019geodesically}
K.~Lytchak, A.and~Nagano.
\newblock Geodesically complete spaces with an upper curvature bound.
\newblock {\em Geometric and Functional Analysis}, 29(1):295--342, 2019.

\bibitem[Pet23]{petrunin2023pure}
A.~Petrunin.
\newblock {\em Pure Metric Geometry}.
\newblock SpringerBriefs in Mathematics. Springer Nature Switzerland, Cham,
  Switzerland, first edition, 2023.

\bibitem[PT14]{papadopoulos2014handbook}
A.~Papadopoulos and M.~Troyanov.
\newblock {\em Handbook of Hilbert Geometry}.
\newblock IRMA lectures in mathematics and theoretical physics. European
  Mathematical Society, 2014.

\bibitem[{Shi}21]{shibahara2021gromov}
H.~{Shibahara}.
\newblock {Gromov-Hausdorff distance with boundary and its application to
  Gromov hyperbolic spaces and uniform spaces}.
\newblock {\em arXiv e-prints}, page
  \href{10.48550/arXiv.2108.03626}{arXiv:2108.03626}, August 2021.

\bibitem[{Tuz}16]{tuzhilin2016invented}
A.~A. {Tuzhilin}.
\newblock {Who Invented the Gromov-Hausdorff Distance?}
\newblock {\em arXiv e-prints}, page
  \href{10.48550/arXiv.1612.00728}{arXiv:1612.00728}, December 2016.

\bibitem[V{\"a}i05]{vaisala2005gromov}
J.~V{\"a}is{\"a}l{\"a}.
\newblock Gromov hyperbolic spaces.
\newblock {\em Expositiones Mathematicae}, 23(3):187--231, 2005.

\bibitem[Wil31]{wilson1931semi}
Wallace~Alvin Wilson.
\newblock On semi-metric spaces.
\newblock {\em American Journal of Mathematics}, 53(2):361--373, 1931.

\end{thebibliography}
	\bibliographystyle{alpha}
	
\end{document}